\documentclass[envcountsect]{svjour3}                     % onecolumn (standard format)
\smartqed  % flush right qed marks, e.g. at end of proof
\usepackage{amsmath}
\usepackage{amsfonts}
\usepackage{eucal}	% use Zapf's beautiful calligraphic characters
\newcommand{\defeq}{:=}

\journalname{Constructive Approximation}

\numberwithin{equation}{section}

\begin{document}

\title{Associated Legendre Functions and Spherical Harmonics of Fractional Degree and Order}

\titlerunning{Legendre Functions of Fractional Degree and Order}

\author{Robert S. Maier}

\institute{Robert S. Maier \at
              Depts.\ of Mathematics and Physics \\
              University of Arizona \\
              Tucson, AZ\ \ 85721\\
              Tel.: +1 520 621 2617\\
              Fax:  +1 520 621 8322\\
              \email{rsm@math.arizona.edu}}

\date{Received: date / Accepted: date}
% The correct dates will be entered by the editor

\maketitle

\begin{abstract}
Trigonometric formulas are derived for certain families of associated
Legendre functions of fractional degree and order, for use in
approximation theory.  These functions are algebraic, and when viewed
as Gauss hypergeometric functions, belong to types classified by
Schwarz, with dihedral, tetrahedral, or octahedral monodromy.  The
dihedral Legendre functions are expressed in terms of Jacobi
polynomials.  For the last two monodromy types, an underlying
`octahedral' polynomial, indexed by the degree and order and having a
non-classical kind of orthogonality, is identified, and recurrences
for~it are worked~out.  It is a (generalized) Heun polynomial, not a
hypergeometric one.  For each of these families of algebraic
associated Legendre functions, a representation of the
rank\nobreakdash-$2$ Lie algebra $\mathfrak{so}(5,\mathbb{C})$ is
generated by the ladder operators that shift the degree
and order of the corresponding solid harmonics.  All such
representations of $\mathfrak{so}(5,\mathbb{C})$ are shown to have a
common value for each of its two Casimir invariants.  The Dirac
singleton representations of $\mathfrak{so}(3,2)$ are included.
\keywords{Associated Legendre function \and Algebraic function \and
  Spherical harmonic \and Solid harmonic \and Jacobi polynomial \and
  Heun polynomial \and Ladder operator}
% \PACS{PACS code1 \and PACS code2 \and more}
\subclass{33C45 \and 33C47 \and 33C55 \and 22E70}
\end{abstract}

\section{Introduction}
\label{sec:intro}

The first-kind associated Legendre functions $P_\nu^\mu(z)$, or the
Ferrers versions ${\rm P}_\nu^\mu(z)$, are classical.  ($P_\nu^\mu(z)$
and ${\rm P}_\nu^\mu(z)$ are continuations of each other, with
respective real domains $z\in(1,\infty)$ and $z\in(-1,1)$.)  The roles
they play when the degree~$\nu$ and order~$\mu$ equal integers $n,m$
are familiar. The Legendre, or Ferrers polynomials ${\rm P}_n(z)\defeq
{\rm P}_n^0(z)$, $n=0,1,2,\dots$, are orthogonal on~$[-1,1]$ and are
used in series expansions.  The spherical harmonics
$Y_n^m(\theta,\phi)\propto {\rm P}_n^m(\cos\theta){\rm e}^{{\rm
    i}m\phi}$ are orthogonal on the symmetric space $S^2=SO(3)/SO(2)$,
and appear in harmonic analysis based on the Lie group~$SO(3)$.

It is less well known that Ferrers functions ${\rm P}_\nu^\mu(z)$ of a
fixed order~$\mu$, and degrees that may be non-integral but are spaced
by integers, can also be used in series expansions.  The fundamental
relation, due to Love and Hunter~\cite{Love92}, is one of
biorthogonality:
\begin{equation}
\label{eq:orthgonalityLove}
  \int_{-1}^1 {\rm P}_{\nu}^\mu(z) {\rm P}_{\nu'}^{-\mu}(-z) \,{\rm d}z =0,
\end{equation}
which holds if (i) ${\rm Re}\,\mu\in(-1,1)$, and (ii)~the degrees
$\nu,\nu'$ differ by a nonzero even integer and are not
half-odd-integers.  For suitable $\nu_0,\mu\in\mathbb{C}$, this makes
possible bilateral expansions of the form
\begin{equation}
\label{eq:seriesLove}
  f(z) = \sum_{n=-\infty}^\infty c_{n}{\rm P}_{\nu_0+2n}^\mu(z), \qquad z\in(-1,1),
\end{equation}
and in particular, the calculation of the coefficients~$c_n$ as inner
products in $L^2[-1,1]$.  (This is the usual Legendre expansion if
$(\nu_0,\mu)=(0,0)$, as ${\rm P}_{-\nu-1}={\rm P}_\nu$ for all~$\nu$.)  For
conditions on~$f$ sufficient for (interior) pointwise convergence,
see~\cite{Love94,Love92}.

The restriction to ${\rm Re}\,\mu\in(-1,1)$ comes from the requirement
that the expansion functions lie in $L^2[-1,1]$.  If the order~$\mu$
is not a positive integer, ${\rm P}_\nu^\mu(z)$ will have leading
behavior as~$z\to1^-$ proportional to $(1-\nobreak z)^{-\mu/2}$, but
in~general its leading behavior as~$z\to(-1)^+$ comprises two terms:
one proportional to $(1+\nobreak z)^{-\mu/2}$, and one to $(1+\nobreak
z)^{+\mu/2}$.  The implications for convergence of the integral
in~(\ref{eq:orthgonalityLove}) are obvious.  These asymptotics have
motivated the suggestion by Pinsky~\cite{Pinsky99} that when ${{\rm
    Re}\,\mu<0}$, the series~(\ref{eq:seriesLove}) should really be
viewed as an expansion of $[(1-\nobreak z)/\allowbreak (1+\nobreak
  z)]^{\mu/2}f(z)$ in the functions $[(1-\nobreak z)/\allowbreak
  (1+\nobreak z)]^{\mu/2}\,\allowbreak{\rm P}_{\nu_0+2n}^\mu(z)$.
This enables a discussion of endpoint convergence, because the latter
functions do not diverge as $z\to1^-\!,(-1)^+$.

It is not usually the case that ${\rm P}_\nu^\mu(z)$ and ${P}_\nu^\mu(z)$
are elementary functions, unless of course $\nu$ and~$\mu$ are integers.
This may be why such expansions as~(\ref{eq:seriesLove}) have been used
infrequently.  In this paper, we derive explicit, trigonometrically
parametrized formulas for several families of Legendre functions,
expressing ${\rm P}_\nu^\mu(z)$, ${P}_\nu^\mu(z)$, and their second-kind
counterparts ${\rm Q}_\nu^\mu(z)$, ${Q}_\nu^\mu(z)$, as elementary
functions.  In each family, $\nu,\mu$~are non-integral but are spaced by
integers: $(\nu,\mu)\in(\nu_0,\mu_0)+\mathbb{Z}^2$ for some
fractional~$\nu_0,\mu_0$.

The simplest example is
\begin{align}
\label{eq:formulaone}
  {\rm P}_{-\frac16+{n}}^{\frac14+{m}}(\cos\theta) &= 2^{-2{m}-3{n}} \Gamma(\tfrac34-{m})^{-1}\\
&\qquad{}\times(\sin\theta)^{-\frac14-{m}} \,B_+^{\frac14+3{m}+3{n}} \,r^{m}_{n}(B_-/B_+),\qquad \theta\in(0,\pi),\nonumber\\
B_\pm &= B_\pm(\theta) \defeq \cos(\theta/3)\pm \sqrt{\frac{4\,\cos^2(\theta/3)-1}3},\nonumber
\end{align}
where $({n},{m})\in\mathbb{Z}^2$.  Here, $r^{m}_{n}=r^{m}_{n}(u)$ is
an `octahedral' rational function that if $n,m\ge0$ is a polynomial of
degree $3{n}+2{m}$ in~$u$; in the base case ${n}={m}=0$, it equals
unity.  It satisfies differential recurrences on ${n}$ and~${m}$, and
three-term non-differential recurrences, as~well.  

The function $r^{m}_{0}(u)$ has a hypergeometric representation in the
Gauss function~${}_2F_1$: it equals
${}_2F_1\left(-2{m},-\frac14-\nobreak3{m};\allowbreak\frac34-\nobreak
{m}\bigm| u\right)$.  But $r^{0}_{n}(u)$, which according
to~(\ref{eq:formulaone}), appears in series of the
form~(\ref{eq:seriesLove}) when $\mu=\frac14$, is less classical.  It
satisfies a second-order differential equation on the Riemann $u$-sphere
with four singular points, not three; so (if ${n}\ge0$) it is a \emph{Heun
  polynomial}, not a hypergeometric one.  The functions
$\{r^{0}_{n}(u)\}_{{n}\in\mathbb{Z}}$ are mutually orthogonal on the
$u$-interval $[0,1]$, in a sense that follows
from~(\ref{eq:orthgonalityLove}), but the orthogonality is of an unusual
Sturm--Liouville kind.

It is clear from (\ref{eq:formulaone}) that for any $n,m\in\mathbb{Z}$, the
function ${\rm P}_{-\frac16+{n}}^{\frac14+{m}}(z=\cos\theta)$ depends
\emph{algebraically} on~$z$, and can be evaluated using radicals.  Each of
the function families considered in this paper is similarly algebraic, and
because any Legendre function can be written in~terms of~${}_2F_1$, the
results below are really trigonometric parametrizations of families of
algebraic~${}_2F_1$'s.  To see a link to prior work, recall from Frobenius
theory that each Legendre function of degree~$\nu$ and order~$\mu$
satisfies a differential equation on the Riemann sphere with three singular
points, the characteristic exponent differences at which are
$\mu,\mu,2\nu+\nobreak1$.  It is a classical result of Schwarz
(see~\cite{Schwarz1873}, and for more recent expositions,
\cite[\S\,2.7.2]{Erdelyi53}, \cite[chap.~VII]{Poole36}
and~\cite{Matsuda85}) that any such equation will have \emph{only algebraic
  solutions} only if the (unordered, unsigned) triple of exponent
differences falls into one of several classes.  The triples from
$(\nu,\mu)=(-\frac16,\frac14)+\allowbreak({n},{m})$, as
in~(\ref{eq:formulaone}), are
$(\frac14,\frac14,\frac23)+\allowbreak({m},{m},2{n})$, and they lie in
Schwarz's octahedral class~V\null.

The families treated below include octahedral ones, with
$(\nu+\nobreak\frac12,\mu)\in\allowbreak(\pm\frac13,\pm\frac14)+\nobreak\mathbb{Z}^2$,
and tetrahedral ones, with
$(\nu+\frac12,\mu)\in\allowbreak(\pm\frac14,\pm\frac13)+\nobreak\mathbb{Z}^2$
or $(\pm\frac13,\pm\frac13)+\nobreak\mathbb{Z}^2$; the Schwarz classes for
the latter being II and~III\null.  The resulting Legendre functions are
octahedral or tetrahedral in the sense that their defining differential
equation, on the Riemann $z$-sphere, has as its projective monodromy group
a finite subgroup of the M\"obius group $PSL(2,\mathbb{R})$, which is
octahedral or tetrahedral.  This will not be developed at length, but there
is a strong geometric reason why $\{r^{m}_{n}(u)\}_{n,m\in\mathbb{Z}}$
deserve to be called octahedral functions, or (when ${n},{m}\ge0$)
polynomials.  For general ${n},{m}$, the lifted function $\tilde r^{m}_{n}
= \tilde r^{m}_{n}(s)\defeq r^{m}_{n}(u=s^4)$ turns~out to satisfy an
equation on the Riemann $s$-sphere with $14$~singular points.  These
include $s=0,\pm1,\pm{\rm i},\infty$, which are the six vertices of an
octahedron inscribed in the sphere; and also, the centers of its eight
faces.

Up to normalization, the doubly indexed functions $r^{m}_{n}(u)$ are
identical to specializations of triply-indexed ones introduced by
Ochiai and Yoshida in their groundbreaking work on algebraic
${}_2F_1$'s~\cite{Ochiai2004}.  For Schwarz classes such as the
octahedral and tetrahedral, they considered the effects of displacing
the triple of exponent differences, not by $({m},{m},2{n})$ as in the
Legendre case, but by general elements of~$\mathbb{Z}^3$.  It is a key
result of the present paper that in the Legendre case, when the triple
has only two degrees of freedom, it is far easier to derive and solve
recurrences on exponent displacements.

Schwarz's classification of algebraic ${}_2F_1$'s also includes a dihedral
class (class~I) and a related `cyclic' class (unnumbered but called class~O
here).  Legendre functions lie in class~I when the order~$\mu$ is a
half-odd-integer, and in class~O when the degree~$\nu$ is an integer.  We
obtain explicit formulas for the Legendre (and Ferrers) functions in the
respective families, of the first and second kinds.  The simplest dihedral
example is
\begin{align}
\label{eq:Jacobirep}
{\rm P}_{-\frac12+\alpha}^{\frac12+m}(\cos\theta) &= \sqrt{\frac2\pi}\:{m}!\\
&\qquad {}\times (\sin\theta)^{-1/2}\,
\left\{{\rm i}^m{\rm e}^{{\rm i}\alpha\theta}\, P_{m}^{(\alpha,-\alpha)}({\rm i}\,\cot\theta)\right\}_{\alpha,+}, \qquad \theta\in(0,\pi),\nonumber
\end{align}
where $m=0,1,2,\dots$, and $\alpha\in\mathbb{C}$ is arbitrary.  Here,
$P_m^{(\alpha,-\alpha)}$ is the Jacobi polynomial of degree~$m$, and
$\{\cdot\}_{\alpha,+}$ signifies the even part under
$\alpha\mapsto-\alpha$.

When ${m}=0$, this becomes a trigonometric version of a well-known
algebraic formula~\cite[3.6(12)]{Erdelyi53}; and when $\alpha=\frac12$, it
expresses ${\rm P}_0^{\frac12+{m}}$ in~terms of the ${m}$th Chebyshev
polynomial of the third kind.  But the general Jacobi
representation~(\ref{eq:Jacobirep}) is new.  There is a significant
literature on `dihedral polynomials' appearing in dihedrally symmetric
${}_2F_1$'s \cite{Ochiai2004,Vidunas2011}, and Vid\=unas has shown they can
be expressed as terminating Appell series~\cite{Vidunas2011}.  Focusing on
the Legendre case, when two of the three exponent differences are equal,
leads to such simpler formulas as~(\ref{eq:Jacobirep}), for both the
dihedral and cyclic families.

Constructing bilateral Ferrers series of the form~(\ref{eq:seriesLove}) is
facilitated by the explicit formulas derived below for the Legendre and
Ferrers functions in the several families.  But the functions $\{ {\rm
  P}^{\mu_0+m}_{\nu_0+n}(z=\nobreak\cos\theta) \}$, $(n,m)\in\mathbb{Z}^2$,
and the corresponding spherical harmonics $\{
Y^{\mu_0+m}_{\nu_0+n}(\theta,\phi) \}$, do not fit into conventional
$SO(3)$-based harmonic analysis unless $(\nu_0,\mu_0)=(0,0)$, when the
latter are the usual surface harmonics on $S^2=SO(3)/SO(2)$.  In the
octahedral and tetrahedral families (and also the dihedral and cyclic, if
$\nu_0$ resp.~$\mu_0$ is rational), it is nonetheless the case that each
spherical harmonic can be viewed as a \emph{finite-valued} function
on~$S^2$.  This is due both to $\nu_0,\mu_0$ being rational, and to the
algebraicity of ${\rm P}_{\nu_0+n}^{\mu_0+m}(z)$ in its argument~$z$, as
seen in~(\ref{eq:formulaone}).

To begin to relate the present results to harmonic analysis, we interpret
in Lie-theoretic terms the recurrences satisfied by any family $\{{\rm
  P}_{\nu_0+n}^{\mu_0+m}(z)\}$ or $\{Y_{\nu_0+n}^{\mu_0+m}(\theta,\phi)\}$,
${(n,m)\in\mathbb{Z}^2}$, which are based on first-order differential
operators that perform ladder operations.  It is well known that for any
$(\nu_0,\mu_0)$, there is an infinite-dimensional representation of the Lie
algebra $\mathfrak{so}(3,\mathbb{R})$ on the span of
$\{Y_{\nu_0}^{\mu_0+m}(z)\}_{m\in\mathbb{Z}}$.  (In the case when
$(\nu_0,\mu_0)\in\mathbb{Z}_\ge\times\mathbb{Z}$, this includes as an
irreducible constituent the familiar $(2\nu_0+\nobreak1)$-dimensional
representation carried by the span of
$Y_{\nu_0}^{-\nu_0},\dots,Y_{\nu_0}^{\nu_0}$.)  There is also a
representation of $\mathfrak{so}(2,1)$ on the span of
$\{Y_{\nu_0+n}^{\mu_0}(z)\}_{n\in\mathbb{Z}}$.  The real algebras
$\mathfrak{so}(3,\mathbb{R}),\allowbreak \mathfrak{so}(2,1)$ are real forms
of the complex Lie algebra $\mathfrak{so}(3,\mathbb{C})$.

As we explain, these `order' and `degree' algebras generate
over~$\mathbb{C}$ a 10\nobreakdash-dimensional, rank\nobreakdash-2 complex
Lie algebra isomorphic to $\mathfrak{so}(5,\mathbb{C})$, which for any
$(\nu_0,\mu_0)$, acts differentially on the family
$\{r^{\nu_0+n}Y_{\nu_0+n}^{\mu_0+m}(\theta,\phi)\}$,
${(n,m)\in\mathbb{Z}^2}$, of \emph{generalized solid harmonics} on
$\mathbb{R}^3$.  The root system of $\mathfrak{so}(5,\mathbb{C})$, of
type~$B_2$, comprises the eight displacement vectors
$\Delta(\nu,\mu)=(0,\nobreak\pm1),\allowbreak(\pm1,\nobreak0),\allowbreak
(\pm1,\nobreak\pm1)$, which yield four ladders on~$(\nu,\mu)$; and for each
ladder, there are differential operators for raising and lowering, a
differential recurrence satisfied by~${\rm P}_\nu^\mu(z=\cos\theta)$, and a
three-term non-differential recurrence.  The ones coming from the roots
$(\pm1,\pm1)$, such as the `diagonal' recurrences
\begin{multline}
  \sqrt{1-z^2}\:\,{\rm P}_{\nu\pm1}^{\mu+1}(z)
  +
  \bigl[\pm(2\nu+1)(1-z^2) + 2\mu\bigr]\,{\rm P}_\nu^\mu(z)
  \\
  +
  \bigl[(\nu+\tfrac12)\pm(\mu-\tfrac12)\bigr]\,
  \bigl[(\nu+\tfrac12)\pm(\mu-\tfrac32)\bigr]\,
  \sqrt{1-z^2}\:\,{\rm P}_{\nu\mp1}^{\mu-1}(z)
  =0,
\end{multline}
may be given here for the first time.  (In this identity, ${\rm P}$ may be
replaced by~${\rm Q}$.)

Connections between associated Legendre/Ferrers functions, or spherical
harmonics, and the complex Lie algebra $\mathfrak{so}(5,\mathbb{C})$ [or
  its real forms $\mathfrak{so}(3,2),\mathfrak{so}(4,1)$
  and~$\mathfrak{so}(5,\mathbb{R})$] are known to exist.  (See
\cite{Miller73a} and \cite[Chaps.~3,4]{Miller77}, and
\cite{Celeghini2013,Durand2003b} in the physics literature; also
\cite{Kyriakopoulos68} for hyperspherical extensions.)  But, most work has
focused on functions of integral degree and order.  The octahedral,
tetrahedral, dihedral, and cyclic families yield explicit
infinite-dimensional representations of $\mathfrak{so}(5,\mathbb{C})$ and
its real forms, which are carried by finite-valued solid harmonics
on~$\mathbb{R}^3$.  When $(\nu_0,\mu_0)=(\tfrac12,\tfrac12)$ or~$(0,0)$,
the representation of $\mathfrak{so}(3,2)$ turns~out to include a known
skew-Hermitian one, of the Dirac singleton type (the `Di' or the `Rac' one,
respectively).  But in general, these Lie algebra representations are new,
non-skew-Hermitian ones, which do not integrate to unitary representations
of the corresponding Lie group.  It is shown below that any of these
representations of $\mathfrak{so}(5,\mathbb{C})$ [or~any of its real
  forms], carried by a harmonic family
$\{r^{\nu_0+n}Y_{\nu_0+n}^{\mu_0+m}(\theta,\phi)\}$,
${(n,m)\in\mathbb{Z}^2}$, is of a distinguished kind, in the sense that it
assigns special values to the two Casimir invariants of the algebra, these
values being independent of~$(\nu_0,\mu_0)$;
cf.~\cite{Celeghini2013,Kyriakopoulos68}.

This paper is structured as follows.  In~\S\,\ref{sec:prelims}, facts
on Legendre/Ferrers functions that will be needed are reviewed.
In~\S\,\ref{sec:octa}, the key results on the octahedral
functions~$r_n^m$ are stated, and explicit formulas for octahedral
Legendre/Ferrers functions are derived.  These are extended to the
tetrahedral families in~\S\S\,\ref{sec:tetra1} and~\ref{sec:tetra2}.
In~\S\,\ref{sec:proofs}, the results in~\S\,\ref{sec:octa} are proved.
In~\S\,\ref{sec:biorthog}, Love--Hunter biorthogonality is related to
Sturm--Liouville biorthogonality.  In~\S\,\ref{sec:cyclicdihedral},
formulas for Legendre/Ferrers functions in the cyclic and dihedral
families are derived, and Love--Hunter expansions in dihedral Ferrers
functions are briefly explored.  In~\S\,\ref{sec:last}, recurrences on
the degree and order, valid for any $(\nu_0,\mu_0)$, are derived, and
are given a Lie-theoretic interpretation:
$\mathfrak{so}(5,\mathbb{C})$ and its real forms are introduced, and
their representations carried by solid harmonics are examined.

\section{Preliminaries}
\label{sec:prelims}

The (associated) Legendre equation is the second-order differential equation
\begin{equation}
\frac{{\rm{d}}}{{\rm{d}}z}
\left[
{(1-z^2)} \frac{{\rm d}{p}}{{\rm d}{z}}
\right]
 + \left[\nu({\nu}+1) - \frac{{\mu^2}}{1-z^2}\right]{p}=0
\label{eq:legendre}
\end{equation}
on the complex $z$-plane.  For there to be single-valued solutions, the
plane is cut along the real axis either from $-\infty$ to~$1$ (the Legendre
choice), or from $-\infty$ to~$-1$ and from $1$ to~$+\infty$ (the Ferrers
choice).  The respective solution spaces have $P_\nu^\mu(z),Q_\nu^\mu(z)$
and ${\rm P}_\nu^\mu(z),{\rm Q}_\nu^\mu(z)$ as bases, except in degenerate
cases indicated below.

At fixed real~$\mu$, Eq.~(\ref{eq:legendre}) can be viewed as a singular
Sturm--Liouville equation on the real Ferrers domain $[-1,1]$, the
endpoints of which are of Weyl's `limit circle' type if $\mu\in(-1,1)$.
(See~\cite{Everitt2005}.)  In this case, all solutions $p=p(z)$ lie in
$L^2[-1,1]$, irrespective of~$\nu$; but the same is not true when
$\mu\notin(-1,1)$, which is why such orthogonality relations
as~(\ref{eq:orthgonalityLove}) can only be obtained if $\mu\in(-1,1)$, or
more generally if ${\rm Re}\,\mu\in(-1,1)$.

Further light on endpoint behavior is shed by Frobenius theory.
Equation~(\ref{eq:legendre}) has regular singular points at $z=-1,1$
and~$\infty$, with respective characteristic exponents expressed in~terms
of the degree~$\nu$ and order~$\mu$ as $+\mu/2,-\mu/2$; $+\mu/2,-\mu/2$;
and $-\nu,\nu+1$.  The exponent differences are $\mu,\mu,2\nu+\nobreak1$.
The functions $P_\nu^\mu,{\rm P}_\nu^\mu$ are Frobenius solutions
associated to the exponent $-\mu/2$ at~$z=1$, and the second Legendre
function~$Q_\nu^\mu$ is associated to the exponent ${\nu+1}$ at~$z=\infty$.
(The second Ferrers function~${\rm Q}_\nu^\mu$ is a combination of two
Frobenius solutions.)  These functions are defined to be analytic (or
rather meromorphic) in~$\nu,\mu$~\cite{Olver97}, the Legendre functions
having the normalizations
\begin{subequations}
\begin{alignat}{2}
\label{eq:Pasympt}
P_\nu^\mu(z) &\sim \frac{2^{\mu/2}}{\Gamma(1-\mu)}\,(z-1)^{-\mu/2}, &\qquad& z\to1,\\
\hat Q_\nu^\mu(z) &\sim\frac{\sqrt\pi}{2^{\nu+1}}\,\frac{\Gamma(\nu+\mu+1)}{\Gamma(\nu+3/2)}\,z^{-\nu-1}, &\qquad& z\to\infty,
\label{eq:Qasympt}
\end{alignat}
\end{subequations}
by convention \cite[chap.~III]{Erdelyi53}.  The notation $\hat
Q_\nu^\mu\defeq {\rm e}^{-\mu\pi{\rm i}}Q_\nu^\mu$ will be used
henceforth; it removes an awkward ${\rm e}^{\mu\pi{\rm i}}$ factor.

The formulas (\ref{eq:Pasympt}),(\ref{eq:Qasympt}) apply if the gammas
are finite; the asymptotics when they are not are given
in~\cite{Olver2010}.  One such degenerate case is when
$\mu=1,2,\dots$.  Then, $P_{-\mu}^\mu,\dots,P_{\mu-1}^\mu\equiv 0$.  A
familiar example is when the degree is a non-negative integer~$n$.
Then, $P_n^m,{\rm P}_n^m\equiv0$ if the order is an integer~$m>n$;
though not if~$m<-n$.  Another degenerate case is when
$\nu+\nobreak\mu$ is a negative integer.  If so, $\hat Q_\nu^\mu$~is
undefined, as (\ref{eq:Qasympt}) suggests; except when
$\nu=-\frac32,-\frac52,\dots$.  Then, $\hat Q_\nu^{\nu+1},\dots, \hat
Q_\nu^{-(\nu+1)}$ are defined.

The Ferrers functions are related to the Legendre ones on their common
domains, which are the upper and lower half-planes
${\pm{\rm{Im}}\,z>0}$, by
\begin{subequations}
\begin{align}
{\rm P}_\nu^\mu &= {{\rm e}}^{\pm{\mu\pi{{\rm i}}/2}} P_\nu^\mu\\
{\rm Q}_\nu^\mu &= {{\rm e}}^{\mp{\mu\pi{{\rm i}}/2}} \hat{Q}_\nu^\mu
\pm{\rm i}\,{(\pi/2)}\,{{\rm e}}^{\pm{\mu\pi{{\rm i}}/2}} P_\nu^\mu.
\label{eq:lincomb2}
\end{align}
\end{subequations}
Thus $P_\nu^\mu,{\rm P}_\nu^\mu$ are related by analytic continuation,
up to phase.  Going from ${P}_\nu^\mu$ to~${\rm P}_\nu^\mu$ typically
involves replacing a factor $(z-1)^{-\mu/2}$ by $(1-z)^{-\mu/2}$; for
instance, $P_1^{-1}(z),\allowbreak{\rm P}_1^{-1}(z)$ are
$\tfrac12\sqrt{z^2-1}$, $\tfrac12\sqrt{1-z^2}$.  Also, owing
to~(\ref{eq:lincomb2}), ${\rm Q}_\nu^\mu$~is undefined iff $\hat
Q_\nu^\mu$~is.

Equation~(\ref{eq:legendre}) is invariant under $\nu\mapsto-\nu-1$,
$\mu\mapsto-\mu$, and $z\mapsto -z$, so that in nondegenerate cases, the
Legendre and Ferrers functions with $\nu$ replaced by $-\nu-\nobreak1$,
$\mu$ by~$-\mu$, and/or $z$ by~$-z$, can be expressed as combinations of
any two (at most) of $P_\nu^\mu,\hat Q_\nu^\mu,\allowbreak {\rm P}_\nu^\mu,
{\rm Q}_\nu^\mu$.  Some `connection' formulas of this type, which
will be needed below, are $P_{-\nu-1}^\mu=P_{\nu}^\mu$, ${\rm
  P}_{-\nu-1}^\mu={\rm P}_{\nu}^\mu$,
\begin{equation}
\label{eq:QtoQ}
  \hat Q_\nu^{-\mu}/\,\Gamma(\nu-\mu+1)
  =
  \hat Q_\nu^{\mu}\,/\,\Gamma(\nu+\mu+1),  
\end{equation}
the $P\to\hat Q$ reduction
\begin{equation}
\label{eq:QtoP}
P_\nu^\mu =
\sec(\nu\pi)\,\Gamma(\nu-\mu+1)^{-1} \Gamma(-\mu-\nu)^{-1}
\left(\hat Q_{-\nu-1}^{-\mu} - \hat Q_\nu^{-\mu}\right),
\end{equation}
the $\hat Q\to{P}$ reduction
\begin{equation}
\label{eq:PtoQ}
(2/\pi)\hat{Q}_\nu^\mu =
 \csc(\mu\pi)\, {P}_\nu^\mu
  -\csc(\mu\pi)\,\frac{\Gamma(\nu+\mu+1)}{\Gamma(\nu-\mu+1)}\,{P}_\nu^{-\mu},
\end{equation}
and the ${\rm Q}\to\rm P$ reduction
\begin{equation}
\label{eq:rmPtormQ}
(2/\pi){\rm Q}_\nu^\mu =
 \cot(\mu\pi)\, {\rm P}_\nu^\mu
  -\csc(\mu\pi)\,\frac{\Gamma(\nu+\mu+1)}{\Gamma(\nu-\mu+1)}\,{\rm
    P}_\nu^{-\mu}.
\end{equation}
(See~\cite{Erdelyi53}.)  It follows from~(\ref{eq:rmPtormQ}) that if
$\mu=\frac12,\frac32,\dots$, then ${\rm Q}_{-\mu}^{\mu},\dots,{\rm
  Q}_{\mu-1}^{\mu}\equiv0$.

The functions $P_\nu^\mu,\hat Q_\nu^\mu$ are known to have the
hypergeometric representations
\begin{subequations}
\label{eq:reps}
\begin{align}
\label{eq:Prep}
P_\nu^\mu(z) &= \frac1{\Gamma(1-\mu)}\,\left({\frac{z+1}{z-1}}\right)^{\mu/2} {}_2F_1\left(-\nu,\,\nu+1;\,1-\mu;\,\frac{1-z}{2}\right)
\\
&= \frac{2^\mu}{\Gamma(1-\mu)}\, \frac{z^{\nu+\mu}}{(z^2-1)^{\mu/2}}\,
{}_2F_1\left(-\frac\nu2-\frac\mu2,\,-\frac\nu2-\frac\mu2+\frac12;\,1-\mu;\,1-\frac1{z^2}\right),
\label{eq:altPrep}
\\[\jot]
{\hat{Q}}_\nu^\mu(z) &= \frac{\sqrt\pi}{2^{\nu+1}}
\,\frac{\Gamma(\nu+\mu+1)}{\Gamma(\nu+3/2)}\,\frac{(z+1)^{\mu/2}}{(z-1)^{\mu/2 + \nu + 1}}\nonumber
\\
&\qquad\qquad\qquad\qquad\quad {}\times {}_2F_1\left(\nu+1,\nu+\mu+1;\,2\nu+2;\,\frac{2}{1-z}\right)
\label{eq:Qrep}
\\
&= \frac{\sqrt\pi}{2^{\nu+1}}
\,\frac{\Gamma(\nu+\mu+1)}{\Gamma(\nu+3/2)}\,\frac{(z^2-1)^{\mu/2}}{z^{\nu+\mu+1}}\nonumber
\\
&\qquad\qquad\qquad\qquad\quad {}\times {}_2F_1\left(\frac\nu2+\frac\mu2+\frac12,\frac\nu2+\frac\mu2+1;\,\nu+\frac32;\,\frac1{z^2}\right).
\label{eq:altQrep}
\end{align}
\end{subequations}
(For ${\rm P}_\nu^\mu$, replace $z-1$ in the prefactor on the right
of~(\ref{eq:Prep}) by $1-z$; the alternative expressions
(\ref{eq:altPrep}),(\ref{eq:altQrep}) come from
(\ref{eq:Prep}),(\ref{eq:Qrep}) by quadratic hypergeometric
transformations.)  Here, ${}_2F_1(a,b;c;x)$ is the Gauss function with
parameters $a,b;c$, defined (on~the disk $|x|<1$, at~least) by the
Maclaurin series $\sum_{k=0}^\infty [(a)_k(b)_k/(c)_k(1)_k]\,x^k$.  In this
and below, the notation $(d)_k$ is used for the rising factorial, i.e.,
\begin{align*}
  (d)_k \defeq
  \begin{cases}
    (d)\dots(d+k-1),&k\ge0;\\
    \left[(d-k')\dots(d-1)\right]^{-1},&k=-k'\le0.
  \end{cases}
\end{align*}
(The unusual second half of this definition, which extends the meaning of
$(d)_k$ to negative~$k$ so that $(d)_k = \left[(d+k)_{-k}\right]^{-1}$ for
all $k\in\mathbb{Z}$, will be needed below.)  If in any ${}_2F_1(a,b;c;x)$
in~(\ref{eq:reps}), the denominator parameter~$c$ is a non-positive integer
and there is an apparent division by zero, the taking of a limit is to be
understood.

The Gauss equation satisfied by ${}_2F_1(a,b;c;x)$ has the three singular
points $x=0,1,\infty$, with respective exponent differences $1-\nobreak
c,\allowbreak c-\nobreak a-\nobreak b,\allowbreak b-\nobreak a$.  Taking
into account either of (\ref{eq:Prep}),(\ref{eq:Qrep}), one sees that this
triple is consistent with the exponent differences $\mu,\mu,2\nu+\nobreak1$
at the singular points $z=1,-1,\infty$ of the Legendre
equation~(\ref{eq:legendre}).  Schwarz's results on algebraicity were
originally phrased in~terms of the Gauss equation, its solutions such as
${}_2F_1$, and the (unordered, unsigned) triple $1-\nobreak c,\allowbreak
c-\nobreak a-\nobreak b,\allowbreak b-\nobreak a$; but they extend to the
Legendre equation, its solutions, and the triple $\mu,\mu,2\nu+\nobreak1$.

\section{Octahedral Formulas (Schwarz Class~V)}
\label{sec:octa}

This section and Sections \ref{sec:tetra1} and~\ref{sec:tetra2} derive
parametric formulas for Legendre and Ferrers functions that are either
octahedral or tetrahedral (two types).  The formulas involve the octahedral
polynomials, or functions, $\{r^{m}_{n}(u)\}_{n,m\in\mathbb{Z}}$.
Section~\ref{subsec:polys} defines these rational functions and states
several results, the proofs of which are deferred to~\S\,\ref{sec:proofs}.

\subsection{Indexed Functions and Polynomials}
\label{subsec:polys}

\begin{definition}
\label{def:rdef}
  For $n,m\in\mathbb{Z}$, the rational functions $r^{m}_{n}=r^{m}_{n}(u)$
  and their `conjugates' $\overline r^{m}_{n}=\overline r^{m}_{n}(u)$ are
  defined implicitly by
  \begin{align*}%{\small}
    &{}_2F_1\left(
    {{-\tfrac1{24}-\frac{m}2-\frac{n}2,\,\tfrac{11}{24}-\frac{m}2-\frac{n}2}\atop{\tfrac34-{m}}}
    \biggm| R(u)\right) = \bigl[(1+u)(1-34u+u^2)\bigr]^{-1/12-{m}-{n}}\,r^{m}_{n}(u),\\
    &{}_2F_1\left(
            {{\tfrac5{24}+\frac{m}2-\frac{n}2,\,\tfrac{17}{24}+\frac{m}2-\frac{n}2}\atop{\tfrac54+{m}}}
            \biggm| R(u)\right) =\\
&\qquad\qquad\qquad\qquad\qquad\qquad\qquad \bigl[(1+u)(1-34u+u^2)\bigr]^{5/12+{m}-{n}}(1-u)^{-1-4{m}}\,\overline r^{m}_{n}(u),
  \end{align*}
  which hold on a neighborhood of $u=0$.  Here,
  \begin{align*}
    R(u) &\defeq \frac{-108\,u(1-u)^4}{\bigl[(1+u)(1-34u+u^2)\bigr]^2} = 1 -
    \frac{(1+14u+u^2)^3}{\bigl[(1+u)(1-34u+u^2)\bigr]^2},\\
    &\defeq -108\,\frac{p_{\rm v}(u)}{p_{\rm e}(u)^2} = 1 - \,\frac{p_{\rm f}(u)^3}{p_{\rm e}(u)^2},
  \end{align*}
  where $p_{\rm v},p_{\rm e},p_{\rm f}$ are the polynomials $u(1-u)^4$,
  $(1+\nobreak u)\allowbreak (1-\nobreak34u +\nobreak u^2)$, $1+\nobreak
  14u+\nobreak u^2$, which satisfy
  $p_{\rm e}^2 - p_{\rm f}^3 + 108\,p_{\rm v}=0$.
  Equivalently,
  \begin{equation}
    \label{eq:tripleangle}
    R = \frac{T(3+T)^2}{(1+3T)^2} = 1-\,\frac{(1-T)^3}{(1+3T)^2},
  \end{equation}
  where $T=T(u)\defeq-12u/(1+u)^2$.\hfil\break[For later use, note that
    (\ref{eq:tripleangle}) is familiar from trigonometry as a
    `triple-angle' formula: $R=\tanh^2\xi$ if $T=\tanh^2(\xi/3)$;
    $R=\coth^2\xi$ if $T=\coth^2(\xi/3)$; and $R=-\tan^2\theta$ if
    $T=-\tan^2(\theta/3)$.]
\end{definition}

It is clear from the definition that $r^{m}_{n},\overline r^{m}_{n}$ are
analytic at $u=0$, at which they equal unity; though it is not obvious that
they are rational in~$u$.  But it is easily checked that the Gauss
equations satisfied by the two ${}_2F_1(x)$ functions have respective
exponent differences (at the singular points $x=0,1,\infty$) equal to
$(\frac14,\frac13,\frac12)+\allowbreak({m},{n},0)$ and
$(-\frac14,\frac13,\frac12)+\allowbreak(-{m},{n},0)$.  These triples lie in
Schwarz's octahedral class~IV, so each ${}_2F_1(x)$ must be an algebraic
function of~$x$.  The definition implicitly asserts that if these algebraic
${}_2F_1$'s are \emph{parametrized} by the degree\nobreakdash-$6$ rational
function $x=R(u)$, the resulting dependences on~$u$ will be captured by
certain \emph{rational} $r^{m}_{n}=r^{m}_{n}(u)$, $\overline
r^{m}_{n}=\overline r^{m}_{n}(u)$.  In the terminology
of~\cite{Ochiai2004}, these are octahedral functions of~$u$.

\begin{theorem}
\label{thm:1}
  {\rm(i)} For\/ $n,m\ge0$, $r^{m}_{n}(u)$ is a polynomial of degree
  $3{n}+\nobreak2{m}$ in\/~$u$, to be called the octahedral polynomial
  indexed by\/ $(n,m)\in \mathbb{Z}^2_\ge$.  Its coefficient of\/~$u^0$ is
  unity, and its coefficient of\/~$u^{3{n}+2{m}}$ is
  \begin{equation}
    \label{eq:ddef}
    d^{m}_{n} \defeq (-)^{{m}+{n}} \,3^{3{m}}\, \frac{(\frac{5}{12})_{{m}-{n}}(\frac{13}{12})_{{m}+{n}}}{(\frac14)_{m}\,(\frac54)_{m}}.
  \end{equation}
  {\rm (ii)} For unrestricted\/ $(n,m)\in\mathbb{Z}^2$, $r^{m}_{n}(u)$ is a
  rational function that equals unity at\/~$u=0$ and is asymptotic to\/
  $d^{m}_{n}u^{3{n}+2{m}}$ as\/~$u\to\infty$.\hfil\break {\rm (iii)} The conjugate
  function\/ $\overline r^{m}_{n}$ is related to\/~$r^{m}_{n}$ by
  \begin{displaymath}
    \overline r^{m}_{n}(u) = \left(d^{m}_{n}\right)^{-1} u^{3{n}+2{m}} r^{m}_{n}(1/u),
  \end{displaymath}
  so that if\/ $n,m\ge0$, $\overline r^{m}_{n}$ is a reversed version of the
  polynomial\/~$r^{m}_{n}$, scaled to equal unity at\/~$u=0$.
\end{theorem}

The functional form of the octahedral functions that are not polynomials,
which are indexed by $(n,m)\in\mathbb{Z}^2\setminus\mathbb{Z}^2_\ge$, is
not complicated.

\begin{theorem}
\label{thm:nonpoly}
  For any\/ $n,m\ge0$,
  \begin{align*}
    r^{m}_{n}(u) &= \Pi_{3{n}+2{m}}(u),\\
    r^{-{m}-1}_{n}(u) &= (1-u)^{-3-4{m}}\,\Pi_{1+3{n}+2{m}}(u),\\
    r^{m}_{-{n}-1}(u) &= (1+14u+u^2)^{-2-3{n}}\,\Pi_{1+3{n}+2{m}}(u),\\
    r^{-{m}-1}_{-{n}-1}(u) &= (1-u)^{-3-4{m}}\,(1+14u+u^2)^{-2-3{n}}\,\Pi_{2+3{n}+2{m}}(u),
  \end{align*}
  where on each line, $\Pi_k(u)$ signifies a polynomial of degree\/~$k$
  in\/~$u$, with its coefficient of\/~$u^0$ equalling unity, and its
  coefficient of\/~$u^k$ coming from the preceding theorem.
\end{theorem}

On their indices $n,m$, the $r^{m}_{n}$ satisfy both differential recurrences
and three-term non-differential recurrences.  The former are given
in~\S\,\ref{sec:proofs} (see Theorem~\ref{thm:diffrec}), and the latter are
as follows.

\begin{theorem}
\label{thm:recs}
  The octahedral functions\/ $r^{m}_{n}=r^{m}_{n}(u)$, indexed by\/
  $(n,m)\in\mathbb{Z}^2$, satisfy second-order\/ {\rm(}i.e.,
  three-term{\rm)} recurrences on\/ ${m}$ and\/ ${n}$, namely
  \begin{align*}
    &(4{m}-3)(4{m}+1)\,r^{{m}+1}_{n} - (4{m}-3)(4{m}+1)\,
    p_{\rm e}(u)
\, r^{m}_{n} \\
    &\qquad\qquad\qquad\qquad\qquad {}-3(12{m}-12{n}-7)(12{m}+12{n}+1)\, 
    p_{\rm v}(u)
\,r^{{m}-1}_{n}
    = 0,\\
    &(12{n}-12{m}+7)\,r^{m}_{{n}+1} - 8(3{n}+1)\,
    p_{\rm e}(u)
\,r^{m}_{n} \\
    &\qquad\qquad\qquad\qquad\qquad {}+(12{n}+12{m}+1)\,
    p_{\rm f}^3(u)
\,r^{m}_{{n}-1} = 0,
  \end{align*}
  where\/ $p_{\rm v},p_{\rm e},p_{\rm f}$ are the polynomials in\/~$u$,
  satisfying\/ $p_{\rm e}^2 - p_{\rm f}^3 + 108\,p_{\rm v} = 0$, that were introduced
  in Definition\/~{\rm\ref{def:rdef}}.  Moreover, they satisfy
  \begin{align*}
    &3(4{m}-3)(4{m}+1)\,r^{{m}+1}_{{n}+1} - (4{m}-3)\left[(12{m}+12{n}+7)p_{\rm f}^3(u) -
      4(3{n}+1)p_{\rm e}^2(u)\right]r^{m}_{n} \\
    &\qquad\qquad\qquad\qquad{}+
    9(12{m}+12{n}+1)(12{m}+12{n}-11)\,p_{\rm v}(u)\,p_{\rm f}^3(u)\, r^{{m}-1}_{{n}-1} = 0,\\
    &3(4{m}-3)(4{m}+1)\,p_{\rm f}^3(u)\,r^{{m}+1}_{{n}-1} - (4{m}-3)\left[(12{m}-12{n}-1)p_{\rm f}^3(u) +
      4(3{n}+1)p_{\rm e}^2(u)\right]r^{m}_{n} \\
    &\qquad\qquad\qquad\qquad{}+
    9(12{m}-12{n}-7)(12{m}-12{n}-19)\,p_{\rm v}(u)\, r^{{m}-1}_{{n}+1} = 0,
  \end{align*}
  which are second-order `diagonal' recurrences.
\end{theorem}

From the first two recurrences in this theorem, one can compute $r^{m}_{n}$ for
any $n,m\in\mathbb{Z}$, if one begins with $r^{0}_{0},r^{1}_{0},r^{0}_{1}$, which are
low-degree polynomials in~$u$ computable `by hand.'  In fact,
\begin{equation}
\label{eq:byhand1}
\begin{gathered}
  r^{0}_{0}(u)=1,\qquad r^{1}_{0}(u)=1-26u-39u^2, \qquad
r^{0}_{1}(u) = 1-39u - \tfrac{195}7 u^2 + \tfrac{13}7 u^3,\\
r^{1}_{1}(u) = 1+175u - 150u^2 + 3550u^3 + 325u^4 + 195u^5.
\end{gathered}
\end{equation}
By specializing to $u=1$ (at which $p_{\rm v},p_{\rm e},p_{\rm f}$
equal $0,-64,36$), one can prove by induction that
$r_n^m(1)=(-64)^{m+n}$ if $m\ge0$.  Examples of octahedral functions
that are not polynomials because they have at least one negative
index, illustrating Theorem~\ref{thm:nonpoly}, are
\begin{align}
\label{eq:byhand2}
  r^{-1}_{0}(u) &= (1-u)^{-3}\,(1+\tfrac17 u)  ,\nonumber\\
  r^{0}_{-1}(u) &= (1+14u+u^2)^{-2}\,(1-5 u) ,\\
  r^{-1}_{-1}(u) &= (1-u)^{-3}\,(1+14u+u^2)^{-2}\,(1+2u-\tfrac1{11}u^2).
  \nonumber
\end{align}
These also follow from the recurrences of Theorem~\ref{thm:recs}.

The recurrences are non-classical, not least because they are bilateral:
they extend to $n,m<0$.  It is shown in~\S\,\ref{sec:proofs} that for
$n,m\ge0$, the degree-$(3{n}+\nobreak2{m})$ polynomials $r^{m}_{n}$ in~$u$ are
(generalized) Heun polynomials, rather than hypergeometric ones; they are
not orthogonal polynomials in the conventional sense.  A~useful third-order
(i.e., four-term) recurrence on~$k$ for the coefficients
$\{a_k\}_{k=0}^{3{n}+2{m}}$ of~$r^{m}_{n}$ is given
in Theorem~\ref{thm:genHeunrep}.

An important degenerate case is worth noting: the case ${n}=0$.  For any
${m}\ge0$, there are hypergeometric representations in~${}_2F_1$ for the
degree-$2{m}$ octahedral polynomials $r^{m}_{0}$ and~$\overline r^{m}_{0}$, namely
\begin{displaymath}
r^{m}_{0}(u) = {}_2F_1\left({{-2{m},\,-\frac14-3{m}}\atop{\frac34-{m}}} \Biggm|
u\right),\qquad 
\overline r^{m}_{0}(u) = {}_2F_1\left({{-2{m},\,\frac14-{m}}\atop{\frac54+{m}}} \Biggm| u\right).
\end{displaymath}
These follow by a sextic hypergeometric transformation of the
${}_2F_1$'s in Definition~1, as well as by the methods
of~\S\,\ref{sec:proofs}.  The first can also be deduced from the
${n}=0$ case of the recurrence on~${m}$ in Theorem~\ref{thm:recs}.
These representations extend to ${m}\in\mathbb{Z}$.

\subsection{Explicit Formulas}
\label{subsec:formulas}

The following two theorems (Theorems \ref{thm:oct1} and~\ref{thm:oct2})
give trigonometrically parametrized formulas for the Legendre/Ferrers
functions $P_\nu^\mu,{\rm P}_{\nu}^\mu$ when $(\nu,\mu)$ equals
$(-\frac16,\nobreak\frac14)+\allowbreak ({n},{m})$ and
$(-\frac16,\nobreak-\frac14)+\allowbreak ({n},-{m})$, with
$({n},{m})\in\mathbb{Z}^2$.  The triple of exponent differences
$(\mu,\mu,2\nu+\nobreak1)$ is respectively equal to
$(\frac14,\frac14,\frac23)+\allowbreak({m},{m},2{n})$ and
$(-\frac14,-\frac14,\frac23)+\allowbreak(-{m},-{m},2{n})$, both lying in
Schwarz's octahedral class~V\null.  An interesting application of these
formulas to the evaluation of certain Mehler--Dirichlet integrals appears
in Theorem~\ref{thm:mehler}.

Let hyperbolic-trigonometric functions $A_\pm$, positive on~$(0,\infty)$,
be defined by
\begin{align}
\label{eq:Adef}
A_\pm=A_\pm(\xi) &\defeq \pm\cosh(\xi/3) + \sqrt{\frac{\sinh\xi}{3\sinh(\xi/3)}}\\
&\hphantom{:}=\pm\cosh(\xi/3) + \sqrt{\frac{4\cosh^2(\xi/3)-1}3},\nonumber
\end{align}
so that $A_+A_-(\xi)=\frac13\sinh^2(\xi/3)$.  This choice is motivated by
Definition~\ref{def:rdef}: if $R=R(u)$ and $T=T(u)=-12u/(1+u)^2$ are
alternatively parametrized as $\tanh^2\xi$ and $\tanh^2(\xi/3)$,
respectively, it is not difficult to verify that the three polynomials
in~$u$ that appear in Definition~\ref{def:rdef} will have the
$\xi$\nobreakdash-parametrizations
\begin{subequations}
\begin{alignat}{2}
p_{\rm v} &= u(1-u)^4 &&= -\,\tfrac{16}{27}\,A_+^{-6}\sinh^2\xi,\label{eq:parampolyvert}\\
p_{\rm e} &= (1+u)(1-34u+u^2) &&= 8\,A_+^{-3}\cosh\xi,\label{eq:parampolyedgemidpt}\\
p_{\rm f} &= 1+14u+u^2 &&= 4\,A_+^{-2}.\label{eq:parampolyfacecenter}
\end{alignat}
\end{subequations}
Moreover, and more fundamentally, $u=-A_-/A_+$.

Also in this section, let $\hat r^{m}_{n}$ signify $r^{m}_{n}/d^{m}_{n}$; so that when
$n,m\ge0$, $\hat r^{m}_{n}$~is a scaled version of the octahedral
polynomial~$r^{m}_{n}$, with its leading rather than its trailing coefficient
equal to unity.  Equivalently, $\hat r^{m}_{n}(u) = u^{3{n}+2{m}}\,\overline
r^{m}_{n}(1/u)$.

\begin{theorem}
\label{thm:oct1}
The formulas
  \begin{align*}
&{P}_{-\frac16+{n}}^{\frac14+{m}}(\cosh\xi) =\left[2^{-2{m}-3{n}}\,\Gamma(\tfrac34-{m})^{-1}\right]\\
&\qquad\qquad\qquad\qquad\qquad{}\times(\sinh\xi)^{-1/4-{m}} A_+^{1/4+3{m}+3{n}}\, r^{m}_{n}(-A_-/A_+),\\
&{P}_{-\frac16+{n}}^{-\frac14-{m}}(\cosh\xi)=\left[(-)^{n}2^{-2{m}-3{n}}\,3^{3/4+3{m}}\,\Gamma(\tfrac54+{m})^{-1}\right]\\
&\qquad\qquad\qquad\qquad\qquad{}\times(\sinh\xi)^{-1/4-{m}} A_-^{1/4+3{m}+3{n}}\, \hat r^{m}_{n}(-A_+/A_-)
\end{align*}
hold for\/ $(n,m)\in\mathbb{Z}^2$ and\/ $\xi\in(0,\infty)$.\hfil\break
{\rm[Note that as $\xi$ increases from $0$ to~$\infty$, the argument
    $u=-A_-/A_+$ of the first~$r^{m}_{n}$, which satisfies
    $T(u)=\tanh^2(\xi/3)$ and $R(u)=\tanh^2\xi$, decreases from $0$ to
    $-(2-\nobreak\sqrt3)^2\allowbreak \approx-0.07$, which is a root
    of $p_{\rm f}(u)=1+\nobreak 14u+\nobreak u^2$.]}
\end{theorem}
\begin{proof}
  These formulas follow from the hypergeometric
  representation~(\ref{eq:altPrep}) of $P_\nu^\mu=P_\nu^\mu(z)$, together
  with the implicit definitions of $r^{m}_{n},\overline r^{m}_{n}$ (see
  Definition~\ref{def:rdef}).

  If $z=\cosh\xi$, the argument $1-1/z^2$ of the right-hand ${}_2F_1$
  in~(\ref{eq:altPrep}) will equal $\tanh^2\xi$.  This is why it is natural
  to parametrize Definition~\ref{def:rdef} by letting $R=R(u)$ equal
  $\tanh^2\xi$, with the just-described consequences.  In deriving the
  formulas, one needs the representation~(\ref{eq:parampolyedgemidpt}); and
  for the second formula, the definition~(\ref{eq:ddef}) of~$d^{m}_{n}$.\qed
\end{proof}

In the following, the circular-trigonometric functions $B_\pm$, positive
on~$(0,\pi)$, are defined by
\begin{align*}
B_\pm=B_\pm(\theta) &\defeq \cos(\theta/3) \pm \sqrt{\frac{\sin\theta}{3\sin(\theta/3)}}\\
&\hphantom{:}=\cos(\theta/3) \pm \sqrt{\frac{4\cos^2(\theta/3)-1}3},
\end{align*}
so that $B_+B_-(\theta)=\frac13\sin^2(\theta/3)$.

\begin{theorem}
\label{thm:oct2}
The formulas
  \begin{align*}
&{\rm P}_{-\frac16+{n}}^{\frac14+{m}}(\cos\theta) =\left[2^{-2{m}-3{n}}\,\Gamma(\tfrac34-{m})^{-1}\right]\\
&\qquad\qquad\qquad\qquad\qquad{}\times(\sin\theta)^{-1/4-{m}} B_+^{1/4+3{m}+3{n}}\, r^{m}_{n}(B_-/B_+),\\
&{\rm P}_{-\frac16+{n}}^{-\frac14-{m}}(\cos\theta)=\left[2^{-2{m}-3{n}}\,3^{3/4+3{m}}\,\Gamma(\tfrac54+{m})^{-1}\right]\\
&\qquad\qquad\qquad\qquad\qquad{}\times(\sin\theta)^{-1/4-{m}} B_-^{1/4+3{m}+3{n}}\, \hat r^{m}_{n}(B_+/B_-)
\end{align*}
hold for\/ $(n,m)\in\mathbb{Z}^2$ and\/
$\theta\in(0,\pi)$.\hfil\break{\rm[Note that as $\theta$ increases
    from $0$ to~$\pi$, the argument $u=B_-/B_+$ of the
    first~$r^{m}_{n}$, which satisfies $T(u)=-\tan^2(\theta/3)$ and
    $R(u)=-\tan^2\theta$, increases from $0$ to~$1$.]}
\end{theorem}
\begin{proof}
  By analytic continuation of Theorem~\ref{thm:oct1}; or in effect, by
  letting $\xi={\rm i}\theta$.\qed
\end{proof}

Because $P_\nu^\mu=P_{-\nu-1}^\mu$ and ${\rm P}_\nu^\mu={\rm
  P}_{-\nu-1}^\mu$, Theorems \ref{thm:oct1} and~\ref{thm:oct2} also supply
formulas for $P_{-\frac56-{n}}^{\frac14+{m}}, P_{-\frac56-{n}}^{-\frac14-{m}}$ and
${\rm P}_{-\frac56-{n}}^{\frac14+{m}}, {\rm P}_{-\frac56-{n}}^{-\frac14-{m}}$.  By
exploiting the $\hat Q\to P$ and ${\rm Q}\to{\rm P}$ reductions
(\ref{eq:PtoQ}) and~(\ref{eq:rmPtormQ}), one easily obtains additional
formulas, for $\hat Q_{-\frac16+{n}}^{\frac14+{m}}, \hat
Q_{-\frac16+{n}}^{-\frac14-{m}},\allowbreak \hat{Q}_{-\frac56-{n}}^{\frac14+{m}},
\hat Q_{-\frac56-{n}}^{-\frac14-{m}}$ and ${\rm Q}_{-\frac16+{n}}^{\frac14+{m}},
     {\rm Q}_{-\frac16+{n}}^{-\frac14-{m}},\allowbreak {\rm
       Q}_{-\frac56-{n}}^{\frac14+{m}}, {\rm Q}_{-\frac56-{n}}^{-\frac14-{m}}$.

\smallskip
Theorems \ref{thm:oct1} and~\ref{thm:oct2} permit certain
Mehler--Dirichlet integrals to be evaluated in closed form.  For
example, consider \cite[3.7(8) and 3.7(27)]{Erdelyi53}
\begin{align}
  P_\nu^\mu(\cosh\xi) &=
  \sqrt{\frac2{\pi}}\:
  \frac{(\sinh\xi)^\mu}{\Gamma(\frac12-\mu)}\,
  \int_0^\xi \frac{\cosh[(\nu+\frac12)t]}{(\cosh\xi - \cosh t)^{\mu+\frac12}}\,{\rm d}t,
  \label{eq:md1}
  \\
  {\rm P}_\nu^\mu(\cos\theta) &=
  \sqrt{\frac2{\pi}}\:
  \frac{(\sin\theta)^\mu}{\Gamma(\frac12-\mu)}\,
  \int_0^\theta \frac{\cos[(\nu+\frac12)\phi]}{(\cos \phi - \cos \theta)^{\mu+\frac12}}\,{\rm d}\phi,
  \label{eq:md2}
\end{align}
which hold when ${\rm Re}\,\mu<\frac12$ for $\xi\in(0,\infty)$ and
$\theta\in(0,\pi)$.  These integral representations of the Legendre
and Ferrers functions of the first kind are classical~\cite{Hobson31,MacRobert32}.
\begin{theorem}
  \label{thm:mehler}
  The formulas
  \begin{align*}
    \int_0^\xi \frac{\cosh[(\frac13+n)t]}{{(\cosh\xi-\cosh t)}^{\frac14-m}}\,{\rm d}t
    &= (-1)^n K_n^m \, A_-^{1/4 + 3m + 3n} \,\hat r_n^m(-A_+/A_-),\\
    \int_0^\theta \frac{\cos[(\frac13+n)\phi]}{{(\cos \phi-\cos \theta)}^{\frac14-m}}\,{\rm d}\phi
    &=  K_n^m \, B_-^{1/4 + 3m + 3n} \,\hat r_n^m(B_+/B_-),
  \end{align*}
  with
  \begin{displaymath}
    K_n^m = \sqrt{\frac{\pi}2}\:2^{-2m-3n}\, 3^{3/4+3m}\frac{\Gamma(\frac34+m)}{\Gamma(\frac54+m)},
  \end{displaymath}
  hold when $m$ is a non-negative integer and $n$ an integer, for
  $\xi\in(0,\infty)$ and $\theta\in(0,\pi)$.
\end{theorem}
\begin{proof}
  Substitute the second formulas of Theorems \ref{thm:oct1}
  and~\ref{thm:oct2} into the $(\nu,\mu) = (-\frac16+\nobreak
  n,\allowbreak -\frac14-\nobreak m)$ specializations of
  (\ref{eq:md1}) and~(\ref{eq:md2}).\qed
\end{proof}

\section{Tetrahedral Formulas (Schwarz Class~II)}
\label{sec:tetra1}

The following theorem gives trigonometrically parametrized formulas for the
second Legendre function~$\hat Q_\nu^\mu$ when $(\nu,\mu) =\allowbreak
(-\frac34,\nobreak-\frac13)+\allowbreak (-{m},-{n})$ and
$(-\frac14,\nobreak-\frac13)+\allowbreak ({m},-{n})$, with
$(n,m)\in\mathbb{Z}^2$.  The triple of exponent differences,
$(\mu,\mu,2\nu+\nobreak1)$, is respectively equal to
$(-\frac13,-\frac13,-\frac12)+\allowbreak(-{n},-{n},-2{m})$ and
$(-\frac13,-\frac13,\frac12)+\allowbreak(-{n},-{n},2{m})$, both lying in
Schwarz's tetrahedral class~II\null.  The hyperbolic-trigonometric
functions $A_\pm=A_\pm(\xi)$ on~$(0,\infty)$ are defined as
in~(\ref{eq:Adef}).

\begin{theorem}
\label{thm:tetr1}
The formulas
\begin{align*}
&  (2/\pi)\hat Q_{-\frac34-{m}}^{-\frac13-{n}}(\coth\xi) = 
\left[
2^{11/4-2{m}-3{n}}\,3^{-3/8}\left(\tfrac14\right)_{m}\left(\tfrac{13}{12}\right)_{{m}+{n}}^{-1}\Gamma(\tfrac43)^{-1}
\right]\\
&\qquad\qquad\qquad{}\times (\sinh\xi)^{1/4-{m}}
\left[
-(-)^{n} \sqrt{\sqrt3+1}\:A_+^{1/4+3{m}+3{n}}\,r^{m}_{n}(-A_-/A_+)
\right],\\
&  (2/\pi)\hat Q_{-\frac14+{m}}^{-\frac13-{n}}(\coth\xi) = 
\left[
2^{11/4-2{m}-3{n}}\,3^{-3/8}\left(\tfrac14\right)_{m}\left(\tfrac{13}{12}\right)_{{m}+{n}}^{-1}\Gamma(\tfrac43)^{-1}
\right]\\
&\qquad\qquad\qquad{}\times (\sinh\xi)^{1/4-{m}}
\left[
+(-)^{m} \sqrt{\sqrt3-1}\:A_-^{1/4+3{m}+3{n}}\,r^{m}_{n}(-A_+/A_-)
\right]
\end{align*}
hold for\/ $(n,m)\in\mathbb{Z}^2$ and\/ $\xi\in(0,\infty)$.
\end{theorem}

\begin{proof}
Combine Whipple's $\hat Q\to P$ transformation~\cite[3.3(13)]{Erdelyi53},
\begin{equation}
\label{eq:whipple}
  \hat Q_\nu^\mu(\coth\xi)
=
\sqrt{\pi/2}\, \Gamma(\nu+\mu+1)\, (\sinh\xi)^{1/2}\,
P_{-\mu-\frac12}^{-\nu-\frac12}(\cosh\xi),
\end{equation}
with the results in Theorem~\ref{thm:oct1}.
The symmetrical forms of the right-hand prefactors are obtained with the
aid of the gamma-function identities
\begin{align*}
\sqrt{\sqrt{3}+1} 
&= \pi^{1/2}\,  2^{1/4}\, 3^{-3/8}\, \Gamma(1/12)\,\Gamma(1/4)^{-1}\,\Gamma(1/3)^{-1}\\
&= \pi^{-3/2} \,2^{-3/4} \,3^{3/8}\, \Gamma(11/12)\,\Gamma(1/4)\,\Gamma(1/3)
\\
\sqrt{\sqrt{3}-1} 
&= \pi^{-1/2}\,2^{-1/4}\,3^{1/8}\, \Gamma(5/12)\,\Gamma(1/4)^{-1}\,\Gamma(1/3)\\
&= \pi^{-1/2}\,2^{-1/4}\,3^{-1/8}\, \Gamma(7/{12})\,\Gamma(1/4)\,\Gamma(1/3)^{-1}
\end{align*}
of Vid\={u}nas~\cite{Vidunas2005a}.\qed
\end{proof}

Because $\hat Q_\nu^\mu, \hat Q_\nu^{-\mu}$ are proportional to each
other (see~(\ref{eq:QtoQ})), Theorem~\ref{thm:tetr1} also supplies
formulas for $\hat Q_{-\frac34-{m}}^{\frac13+{n}},\hat
Q_{-\frac14+{m}}^{\frac13+{n}}$.  Moreover, it leads to the following
two theorems.

\begin{theorem}
\label{thm:tetr2}
The formulas
\begin{align*}
&P_{-\frac34-{m}}^{-\frac13-{n}}(\coth\xi) = 
\left[
(-)^{n}
2^{5/4-2{m}-3{n}}\,3^{-3/8}\left(\tfrac14\right)_{m}\left(\tfrac{13}{12}\right)_{{m}+{n}}^{-1}\Gamma(\tfrac43)^{-1}
\right]
\\
&\qquad\qquad\qquad{}\times (\sinh\xi)^{1/4-{m}}\,\Bigl[ 
(-)^{n} \sqrt{\sqrt3-1}\:A_+^{1/4+3{m}+3{n}}\,r^{m}_{n}(-A_-/A_+)
\\
&\qquad\qquad\qquad\qquad\qquad\qquad\qquad{}-
(-)^{m} \sqrt{\sqrt3+1}\:A_-^{1/4+3{m}+3{n}}\,r^{m}_{n}(-A_+/A_-)
\Bigr],\\
&P_{-\frac34-{m}}^{\frac13+{n}}(\coth\xi) = 
\left[
(-)^{n}
2^{-1/4-2{m}-3{n}}\,3^{-3/8}\left(\tfrac14\right)_{m}\left(\tfrac{5}{12}\right)_{{m}-{n}}^{-1}\Gamma(\tfrac23)^{-1}
\right]
\\
&\qquad\qquad\qquad{}\times (\sinh\xi)^{1/4-{m}}\,\Bigl[ 
(-)^{n} \sqrt{\sqrt3+1}\:A_+^{1/4+3{m}+3{n}}\,r^{m}_{n}(-A_-/A_+)
\\
&\qquad\qquad\qquad\qquad\qquad\qquad\qquad{}+
(-)^{m} \sqrt{\sqrt3-1}\:A_-^{1/4+3{m}+3{n}}\,r^{m}_{n}(-A_+/A_-)
\Bigr]
\end{align*}
hold for\/ $(n,m)\in\mathbb{Z}^2$ and\/ $\xi\in(0,\infty)$.
\end{theorem}
\begin{proof}
  Combine the $P\to\hat Q$ reduction~(\ref{eq:QtoP}) with the results in
  Theorem~\ref{thm:tetr1}.\qed
\end{proof}

\noindent
Because $P_\nu^\mu=P_{-\nu-1}^\mu$, Theorem~\ref{thm:tetr2} also supplies
formulas for $P_{-\frac14+{m}}^{-\frac13-{n}},P_{-\frac14+{m}}^{\frac13+{n}}$.

\smallskip
In Theorem~\ref{thm:tetr3}, the hyperbolic-trigonometric functions
$C_\pm$, positive on~$(-\infty,\infty)$, are defined by
\begin{align*}
C_\pm=C_\pm(\xi) &\defeq \pm\sinh(\xi/3) +\sqrt{\frac{\cosh\xi}{3\cosh(\xi/3)}}\\
&\hphantom{:}=\pm\sinh(\xi/3) +\sqrt{\frac{4\sinh^2(\xi/3)+1}3},
\end{align*}
so that $C_+C_-(\xi)=\frac13\cosh^2(\xi/3)$.

\begin{theorem}
\label{thm:tetr3}
The formulas
\begin{align*}
&{\rm P}_{-\frac34-{m}}^{-\frac13-{n}}(\tanh\xi) = 
\left[
2^{5/4-2{m}-3{n}}\,3^{-3/8}\left(\tfrac14\right)_{m}\left(\tfrac{13}{12}\right)_{{m}+{n}}^{-1}\Gamma(\tfrac43)^{-1}
\right]
\\
&\qquad\qquad\qquad{}\times (\cosh\xi)^{1/4-{m}}\,\Bigl[ 
-(-)^{n} \sqrt{\sqrt3-1}\:C_+^{1/4+3{m}+3{n}}\,r^{m}_{n}(-C_-/C_+)
\\
&\qquad\qquad\qquad\qquad\qquad\qquad\qquad{}+
(-)^{m} \sqrt{\sqrt3+1}\:C_-^{1/4+3{m}+3{n}}\,r^{m}_{n}(-C_+/C_-)
\Bigr],\\
&{\rm P}_{-\frac34-{m}}^{\frac13+{n}}(\tanh\xi) = 
\left[
(-)^{n}
2^{-1/4-2{m}-3{n}}\,3^{-3/8}\left(\tfrac14\right)_{m}\left(\tfrac{5}{12}\right)_{{m}-{n}}^{-1}\Gamma(\tfrac23)^{-1}\right]\\
&\qquad\qquad\qquad{}\times (\cosh\xi)^{1/4-{m}}\,\Bigl[ 
(-)^{n} \sqrt{\sqrt3+1}\:C_+^{1/4+3{m}+3{n}}\,r^{m}_{n}(-C_-/C_+)
\\
&\qquad\qquad\qquad\qquad\qquad\qquad\qquad{}+
(-)^{m} \sqrt{\sqrt3-1}\:C_-^{1/4+3{m}+3{n}}\,r^{m}_{n}(-C_+/C_-)
\Bigr]
\end{align*}
hold for\/ $(n,m)\in\mathbb{Z}^2$ and\/ $\xi\in(-\infty,\infty)$.
\hfil\break {\rm[Note that as $\xi$ decreases from $\infty$
    to~$-\infty$, the argument $u=-C_-/C_+$ of the first~$r^{m}_{n}$,
    which satisfies $T(u)=\coth^2(\xi/3)$ and $R(u)=\coth^2\xi$,
    decreases from $0$ to $-(2+\nobreak\sqrt3)^2\allowbreak
    \approx-14.0$, which is a root of $p_{\rm f}(u)=1+\nobreak
    14u+\nobreak u^2$.]}
\end{theorem}
\begin{proof}
  By analytic continuation of the results in Theorem~\ref{thm:tetr2}; or in
  effect, by replacing $\xi$ by $\xi+\nobreak{\rm i}\pi/2$.\qed
\end{proof}

Because ${\rm P}_\nu^\mu={\rm P}_{-\nu-1}^\mu$, Theorem~\ref{thm:tetr3}
also supplies formulas for ${\rm P}_{-\frac14+{m}}^{-\frac13-{n}},{\rm
  P}_{-\frac14+{m}}^{\frac13+{n}}$.  By exploiting the ${\rm Q}\to{\rm P}$
reduction~(\ref{eq:rmPtormQ}), one easily obtains additional formulas, for
${\rm Q}_{-\frac34-{m}}^{-\frac13-{n}},\allowbreak {\rm
  Q}_{-\frac34-{m}}^{\frac13+{n}},\allowbreak \allowbreak {\rm
  Q}_{-\frac14+{m}}^{-\frac13-{n}},{\rm Q}_{-\frac14+{m}}^{\frac13+{n}}$.

\section{Tetrahedral Formulas (Schwarz Class~III)}
\label{sec:tetra2}

The following theorems give parametrized formulas for the Legendre/Ferrers
functions $P_\nu^\mu,\allowbreak\hat Q_\nu^\mu,{\rm P}_{\nu}^\mu$ when $(\nu,\mu)
=\allowbreak (-\frac16,\nobreak-\frac13)+\allowbreak ({n},-{n})$ and
$(-\frac56,\nobreak\frac13)+\allowbreak (-{n},{n})$, with
${n}\in\mathbb{Z}$.  The triple of exponent differences,
$(\mu,\mu,2\nu+\nobreak1)$, is respectively equal to
$(-\frac13,-\frac13,\frac23)+\allowbreak(-{n},-{n},2{n})$ and
$(\frac13,\frac13,-\frac23)+\allowbreak({n},{n},-2{n})$, both lying in
Schwarz's tetrahedral class~III\null.

\begin{theorem}
\label{thm:tetr2_1}
The formulas
  \begin{align*}
    {\rm P}_{-\frac16+{n}}^{-\frac13-{n}}\bigl(\sqrt{1-{\rm e}^{-2\xi}}\,\bigr) 
&=2^{-\frac13-{n}}\bigl(1-{\rm e}^{-2\xi}\bigr)^{-1/4} \, P_{-\frac34}^{-\frac13-{n}}(\coth\xi),\\
    {\rm P}_{-\frac56-{n}}^{\frac13+{n}}\bigl(\sqrt{1-{\rm e}^{-2\xi}}\,\bigr) 
&=2^{\frac13+{n}}\bigl(1-{\rm e}^{-2\xi}\bigr)^{-1/4} \, P_{-\frac34}^{\frac13+{n}}(\coth\xi)
  \end{align*}
where expressions for the right-hand Legendre functions are provided
by Theorem\/~{\rm\ref{thm:tetr2}}, hold for\/ $n\in\mathbb{Z}$ and\/
$\xi\in(0,\infty)$.
\end{theorem}

\begin{theorem}
\label{thm:tetr2_2}
The formulas
  \begin{align*}
    {P}_{-\frac16+{n}}^{-\frac13-{n}}\bigl(\sqrt{1+{\rm e}^{-2\xi}}\,\bigr) 
&=2^{-\frac13-{n}}\bigl(1+{\rm e}^{-2\xi}\bigr)^{-1/4} \: {\rm P}_{-\frac34}^{-\frac13-{n}}(\tanh\xi),\\
    {P}_{-\frac56-{n}}^{\frac13+{n}}\bigl(\sqrt{1+{\rm e}^{-2\xi}}\,\bigr) 
&=2^{\frac13+{n}}\bigl(1+{\rm e}^{-2\xi}\bigr)^{-1/4} \: {\rm P}_{-\frac34}^{\frac13+{n}}(\tanh\xi),
  \end{align*}
where expressions for the right-hand Ferrers functions are provided by
Theorem\/~{\rm\ref{thm:tetr3}}, hold for\/ $n\in\mathbb{Z}$ and\/
$\xi\in(-\infty,\infty)$.
\end{theorem}

\begin{theorem}
\label{thm:tetr2_3}
The formulas
  \begin{align*}
    (2/\pi)\hat{Q}_{-\frac16+{n}}^{-\frac13-{n}}\bigl(\sqrt{1+{\rm e}^{2\xi}}\,\bigr) 
&=2^{-\frac13-{n}}\bigl(1+{\rm e}^{2\xi}\bigr)^{-1/4}\,\sqrt2 \: {\rm P}_{-\frac34}^{-\frac13-{n}}(-\tanh\xi),\\
    (2/\pi)\hat{Q}_{-\frac56-{n}}^{\frac13+{n}}\bigl(\sqrt{1+{\rm e}^{2\xi}}\,\bigr) 
&=2^{\frac13+{n}}\bigl(1+{\rm e}^{2\xi}\bigr)^{-1/4}\,\sqrt2 \: {\rm P}_{-\frac34}^{\frac13+{n}}(-\tanh\xi),
  \end{align*}
where expressions for the right-hand Ferrers functions are provided by
Theorem\/~{\rm\ref{thm:tetr3}}, hold for\/ $n\in\mathbb{Z}$ and\/
$\xi\in(-\infty,\infty)$.
\end{theorem}

Because ${P}_\nu^\mu={P}_{-\nu-1}^\mu$ and ${\rm P}_\nu^\mu={\rm
  P}_{-\nu-1}^\mu$, Theorems \ref{thm:tetr2_1} and~\ref{thm:tetr2_2} also
supply formulas for $P_{-\frac56-{n}}^{-\frac13-{n}},
P_{-\frac16+{n}}^{\frac13+{n}},\allowbreak{\rm
  P}_{-\frac56-{n}}^{-\frac13-{n}},{\rm P}_{-\frac16+{n}}^{\frac13+{n}}$.  And
because $\hat Q_\nu^\mu, \hat Q_\nu^{-\mu}$ are proportional to each other
(see (\ref{eq:QtoQ})), Theorem~\ref{thm:tetr2_3} also supplies formulas for
$\hat Q_{-\frac16+{n}}^{\frac13+{n}},\hat Q_{-\frac56-{n}}^{-\frac13-{n}}$.  By
exploiting the ${\rm Q}\to{\rm P}$ reduction~(\ref{eq:rmPtormQ}), one
easily obtains additional formulas, for ${\rm
  Q}_{-\frac16+{n}}^{-\frac13-{n}},{\rm Q}_{-\frac56-{n}}^{\frac13+{n}},
\allowbreak {\rm Q}_{-\frac56-{n}}^{-\frac13-{n}},{\rm
  Q}_{-\frac16+{n}}^{\frac13+{n}}$.

The formulas in Theorems \ref{thm:tetr2_1}, \ref{thm:tetr2_2}, and
\ref{thm:tetr2_3} are straightforward reparametrizations of the
$\alpha=\frac13+{n}$ and $\alpha=-\frac13-{n}$ cases of the identities
  \begin{align*}
    P_{-\frac14}^{-\alpha}(\cosh\xi) &= 2^\alpha\sqrt{{\rm sech}(\xi/2)}\:\,
    {\rm P}_{\alpha-\frac12}^{-\alpha} \left({\rm sech}(\xi/2)\right),\\
    {\rm P}_{-\frac14}^{-\alpha}(\cos\theta) &= 2^\alpha\sqrt{\sec(\theta/2)}\:\,
    P_{\alpha-\frac12}^{-\alpha} \left(\sec(\theta/2)\right),\\
    \sqrt2\, {\rm P}_{-\frac14}^{-\alpha}(-\cos\theta) &= 2^\alpha\sqrt{\sec(\theta/2)}\:\,
    (2/\pi)\hat Q_{\alpha-\frac12}^{-\alpha} \left(\sec(\theta/2)\right),
  \end{align*}
which hold when $\xi\in(0,\infty)$ and $\theta\in(0,\pi)$.  These appear as
identities $I_4(i)$, $I_4(ii)$, and $I_4(\overline{ii})$ in~\cite{Maier24},
and are really quadratic hypergeometric transformations in disguise.

\section{Proofs of Results in Section~\ref{subsec:polys}}
\label{sec:proofs}

The octahedral and tetrahedral formulas in \S\S\,\ref{subsec:formulas},
\ref{sec:tetra1}, and~\ref{sec:tetra2} followed from the theorems
in~\S\,\ref{subsec:polys} on the octahedral functions $r^{m}_{n} =
r^{m}_{n}(u)$, which were stated without proof.  The present section
provides proofs, in some cases sketched, and obtains a few additional
results.  These are Theorem~\ref{thm:diffrec} (on~the differential equation
and differential recurrences satisfied by~$r^{m}_{n}$), and Theorems
\ref{thm:hypergrep}, \ref{thm:Heunrep}, and~\ref{thm:genHeunrep} (on~the
interpretation of~$r^{m}_{n}$ when $n,m\ge0$ as a hypergeometric, Heun, or
generalized Heun polynomial).  This section also reveals the origin of the
degree\nobreakdash-$6$ rational function $x=R(u)$ in
Definition~\ref{def:rdef}.

Consider a Riemann sphere $\mathbb{P}^1_s$, parametrized by~$s$ and
identified by stereographic projection with the complex $s$-plane.  (As
usual, $s=0$ is at the bottom and $s=\infty$ is at the top; points with
$|s|=1$ are taken to lie on the equator.)  Let a regular octahedron
(a~Platonic solid) be inscribed in the sphere, with its six vertices
$v_1,\dots,v_6$ at $s=0,\pm1,\pm{\rm i},\infty$, i.e., at the five roots of
$q_{\rm v}(s)\defeq s(1-\nobreak s^4)$ and at $s=\infty$.  By some
trigonometry \cite[chap.~VII]{Poole36}, one can show that the twelve
edge-midpoints $e_1,\dots,e_{12}$ of the octahedron, radially projected
onto the sphere, are located at $s=\allowbreak (\pm1\pm\nobreak{\rm
  i})/\sqrt2$ and $s=\sqrt{\pm1}\,(\pm1\pm\nobreak\sqrt2)$, which are the
roots of $q_{\rm e}\defeq(1+\nobreak s^4)\allowbreak (1-\nobreak 34s^4 +
\nobreak s^8) = \allowbreak 1-\nobreak 33s^4 -\nobreak 33s^8 +\nobreak
s^{12}$.  Similar, its eight face-centers $f_1,\dots f_8$, when radially
projected, are located at $s=\allowbreak(\pm1\pm\nobreak{\rm
  i})(1\pm\nobreak\sqrt3)/2$, which are the roots of $q_{\rm
  f}(s)\defeq\allowbreak 1+\nobreak 14s^4 + \nobreak s^8$.  The polynomials
$q_{\rm v},q_{\rm e},q_{\rm f}$ are (relative) invariants of the symmetry
group of the octahedron, which is an order\nobreakdash-$24$ subgroup of the
group of rotations of the Riemann $s$\nobreakdash-sphere.

The well-known \emph{octahedral equation} states that $q_{\rm e}^2 - q_{\rm
  f}^3 + 108\,q_{\rm v}^4=0$.  The validity of this identity (a~syzygy, in
the language of invariant theory) suggests considering the
degree\nobreakdash-24 rational function $\tilde R=\tilde R(s)$ equal to
$1-\nobreak q_{\rm f}^3/q_{\rm e}^2$, i.e.,
\begin{displaymath}
  \tilde R(s) \defeq \frac{-108\,[s(1-s^4)]^4}{\bigl[(1+s^4)(1-34s^4+s^8)\bigr]^2} = 1 -
  \frac{(1+14s^4+s^8)^3}{\bigl[(1+s^4)(1-34s^4+s^8)\bigr]^2}.
\end{displaymath}
On the $s$-sphere, $\tilde R(s)$ equals $0,1,\infty$ at (respectively)
the vertices, the face-centers, and the edge-midpoints.  It is an
\emph{absolute invariant} of the symmetry group of the octahedron.
(Its derivative ${\rm d}\tilde R(s)/{\rm d}s$ can be written as
$-432\, q_{\rm v}^3 q_{\rm f}^2 / q_{\rm e}^3$ but is only a relative
invariant.)  The covering $\mathbb{P}^1_s \to \mathbb{P}^1_x$ given by
$x=\tilde R(s)$ is ramified above $x=0,1,\infty$, and its ramification
structure can be written as $(6)4=(8)3=(12)2$: each of the six points
above $s=0$ (i.e., the vertices) appears with multiplicity~$4$, etc.

Following and extending Schwarz~\cite{Poole36,Schwarz1873}, consider the
effect of lifting the Gauss hypergeometric equation satisfied by
${}_2F_1\left({-\tfrac1{24}-\frac{m}2-\frac{n}2,\,\tfrac{11}{24}-\frac{m}2-\frac{n}2};{\tfrac34-{m}};
x\right)$, the ${}_2F_1(x)$ appearing in Definition~\ref{def:rdef}, from
the $x$\nobreakdash-sphere to the $s$\nobreakdash-sphere, along $x=\tilde
R(s)$.  It should be recalled that the Gauss equation satisfied by
$f(x)={}_2F_1(a,b;c;x)$ is the Fuchsian differential equation
($D_x\defeq {\rm d}/{\rm d}x)$)
\begin{equation}
\label{eq:GaussHE}
D_x^2 f + \left[
\frac{c}x + \frac{a+b-c+1}{x-1} 
\right]D_xf
+\left[
\frac{ab}{x(x-1)}
\right]f=0,
\end{equation}
which has characteristic exponents $\{0,1-c\}$, $\{0,c-a-b\}$, $\{a,b\}$ at
its singular points $x=0,1,\infty$, with (unsigned) exponent differences
$1-\nobreak c$; $c-\nobreak a-\nobreak b$; $b-\nobreak a$.  (The function
${}_2F_1(a,b;c;x)$ is the Frobenius solution associated to the zero
exponent at $x=0$.)  The effects of the ramified lifting by $x=\tilde R(s)$
are conveniently expressed in the classical notation of Riemann
P\nobreakdash-symbols, which display the exponents at each singular point
\cite{Poole36,Whittaker27}.  For the ${}_2F_1$ of
Definition~\ref{def:rdef}, one can write
\begin{subequations}
\begin{align}
{}_2F_1\left(x=\tilde R(s)\right) &=
P
\left\{
\begin{array}{ccccccc|c}
  0 &&& 1 &&& \infty & \,x = \tilde R(s) \\
\hline
0 &&& 0 &&& -\frac1{24}-\frac{m}2-\frac{n}2 & \\
\frac14+{m} &&& \frac13 + {n} &&& \frac{11}{24}-\frac{m}2-\frac{n}2 & 
\end{array}
\right\}\\
&=
P
\left\{
\begin{array}{ccccc|c}
  v_1,\dots,v_6 &&& f_1,\dots,f_8 & e_1,\dots,e_{12} & \,s \\
\hline
0 &&& 0 & -\frac1{12}-{m}-{n} & \\
1+4{m} &&& 1 + 3{n} & \frac{11}{12}-{m}-{n} & 
\end{array}
\right\},
\label{eq:prepreP}
\end{align}
\end{subequations}
because any pair of characteristic exponents at a point $x=x_0$ beneath a
ramification point $s=s_0$ of order~$k$ is multiplied by~$k$ when lifted.
This function of~$s$ satisfies a differential equation on the $s$-sphere
with the indicated singular points and exponents.

If ${\mathcal{L}}g=0$ is any Fuchsian differential equation on the $s$-sphere,
the modified equation ${{\mathcal{L}}}'g'=0$ obtained by the
change of dependent variable $g'=(1-\nobreak s/s_0)^\alpha g$ has its
exponents at $s=s_0$ shifted upward by~$\alpha$, and those at $s=\infty$
shifted downward by the same.  As an application of this, one deduces
from~(\ref{eq:prepreP}) that
\begin{subequations}
\label{eq:uPsym}
\begin{align}
\tilde f(s) &= \left[q_{\rm e}(s)\right]^{\frac1{12}+{m}+{n}}\,{}_2F_1\left(x=\tilde R(s)\right)\\
&=P
\left\{
\begin{array}{cccccc|c}
  v_1,\dots,v_5 && f_1,\dots,f_8 && e_1,\dots,e_{12} & \infty & \,s \\
\hline
0 && 0 && 0 & -1-12{m}-12{n} & \\
1+4{m} && 1 + 3{n} && 1 & -8{m}-12{n} &
\end{array}
\right\}\\
&=
P
\left\{
\begin{array}{cccc|c}
  v_1,\dots,v_5 && f_1,\dots,f_8 & \infty & \,s \\
\hline
0 && 0 & -1-12{m}-12{n} & \\
1+4{m} && 1 + 3{n} & -8{m}-12{n} &
\end{array}
\right\},
\label{eq:uPsym3}
\end{align}
\end{subequations}
because $e_1,\dots, e_{12}$ are the roots of~$q_{\rm e}$, and $v_6=\infty$.  The
left-hand function $\tilde f=\tilde f(s)$, which by examination is $\tilde
r^{m}_{n}(s)\defeq r^{m}_{n}(u=s^4)$, will be the solution of a `lifted and shifted'
differential equation on the $s$\nobreakdash-sphere, with the indicated
exponents.  After the shifting, the edge-midpoints $e_1,\dots,e_{12}$ cease
being singular points, because the new exponents at each are $0,1$, which
are those of an ordinary point.

It is straightforward if tedious to compute the differential equation
satisfied by $\tilde f = \tilde r^{m}_{n}(s)\defeq r^{m}_{n}(u=s^4)$
explicitly, by applying to the appropriate Gauss equation of the
form~(\ref{eq:GaussHE}) the changes of variable that perform (i)~the
lifting along $s\mapsto x = \tilde R(s)$, and (ii)~the multiplication by
$[q_{\rm e}(s)]^{\frac1{12}+{m}+{n}}$.  One finds that $\tilde f$~satisfies
$\tilde{\mathcal{L}}^{m}_{n}\tilde f=0$, where
\begin{align}
\label{eq:ude}
&\tilde{\mathcal{L}}^{m}_{n} =
D_s^2 - \left[
4{m}\, \frac{5s^4-1}{s(s^4-1)}
+3{n}\,\frac{8s^7+56s^3}{s^8+14s^4+1}
\right]\,D_s\\
&\qquad\qquad{}+ 
4(12{m}+12{n}+1)\,
\frac{s^3\left[(2m+3n)(s^8+14s^4+1) - 12n\,(3s^4+1)\right]}
{s(s^4-1)(s^8+14s^4+1)}.\nonumber
\end{align}
That the singular points of this operator are the roots $v_1,\dots,v_5$
of~$q_{\rm v}$ (plus $v_6=\infty$), and the roots $f_1,\dots,f_8$
of~$q_{\rm f}$, is clear; as is the fact that their exponents are as shown
in the P\nobreakdash-symbol~(\ref{eq:uPsym3}).  The degenerate case $n=m=0$
is especially interesting.  As one expects from the P-symbol, the operator
$\tilde{\mathcal{L}}^{0}_{0}$ is simply the Laplacian~$D_s^2$, the kernel
of which is spanned by~$1,s$.  For $\tilde f = \tilde f(s) =
r^{m}_{n}(s^4)$, it is easy to rule~out any admixture of the latter
solution by examining Definition~\ref{def:rdef}; and because $r^{0}_{0}(u)$
equals unity at~$u=0$, the base octahedral function~$r^{0}_{0}$ must be
identically equal to unity.

Because $r^{0}_{0}\equiv1$, it follows from Definition~\ref{def:rdef}
that the hypergeometric function appearing in the definition
of~$r^{m}_{n}$ when $n=m=0$, which is
${}_2F_1\bigl(-\frac1{24},\frac{11}{24};\frac34 \mid x)$, must be
algebraic in its argument~$x$.  This is essentially the 1873 result of
Schwarz~\cite{Schwarz1873}, the proof of which was later restated in a
P\nobreakdash-symbol form by Poole~\cite{Poole36}.  However, it is not
trivial to extend this result on~$r^{0}_{0}$ to a constructive proof
that $r^{m}_{n}=r^{m}_{n}(u)$ is a rational function of~$u$ for each
$(n,m)\in\mathbb{Z}^2$. This is because the differential equation
$\tilde{\mathcal{L}}^{m}_{n} \tilde f=0$, as one sees
from~(\ref{eq:ude}), is far more complicated than $D_s^2\tilde f=0$
(Laplace's equation) when $(n,m)\neq(0,0)$.  A constructive proof is
best based on \emph{contiguity relations} between adjacent~$(n,m)$,
i.e., recurrences in the spirit of Gauss, derived as follows.

First, simplify the lifting along the covering $s\mapsto x$, i.e., along
the degree\nobreakdash-24 map $x=\tilde R(s)$.  Each octahedral function
$\tilde r^{m}_{n}(s)$ turns~out to `factor through' $u=s^4$, so it suffices to
lift the Gauss hypergeometric equation from the $x$\nobreakdash-sphere
$\mathbb{P}^1_x$ to the $u$\nobreakdash-sphere $\mathbb{P}^1_u$, along the
degree\nobreakdash-$6$ map $x=R(u)$ of Definition~\ref{def:rdef}, i.e.,
\begin{displaymath}
  R(u) \defeq -108\, \frac{p_{\rm v}(u)}{p_{\rm e}(u)^2} = 1 - \frac{p_{\rm
      f}(u)^3}{p_{\rm e}(u)^2}, 
\end{displaymath}
where
\begin{alignat*}{2}
  p_{\rm v}(u) &=  q_{\rm v}(s)^4  &{}={}& u(1-u)^4   \\
  p_{\rm e}(u) &=  q_{\rm e}(s)    &{}={}&   (1+ u) (1-34u + u^2)  \\
  p_{\rm f}(u) &=  q_{\rm f}(s)    &{}={}&   1+ 14u + u^2.
\end{alignat*}
Replacing the lifted variable $s$ by $u=s^4$ quotients out an
order\nobreakdash-4 cyclic group of rotations of the $s$\nobreakdash-sphere
(and hence of the octahedron), about the axis passing through its north and
south poles.

The syzygy becomes $p_{\rm e}^2 - p_{\rm f}^3 + 108\,p_{\rm v}=0$.  The roots
$u=0,1$ of~$p_{\rm v}$, and $u=\infty$, correspond to the south-pole
vertex, the four equatorial ones, and the north-pole one.  The three roots
$u=(3+\nobreak 2\sqrt2)^2,\allowbreak -1,\allowbreak (3-\nobreak2\sqrt2)^2$
of~$p_{\rm e}$ correspond to the four edge-midpoints in the northern
hemisphere, the four on the equator, and the four to the south.  The two
roots $u=-(2\pm\sqrt3)^2$ of~$p_{\rm f}^2$ correspond to the four
face-centers in the north, and the four in the south.  The covering
$\mathbb{P}^1_u\to\mathbb{P}^1_x$ is still ramified above $x=0,1,\infty$,
but its ramification structure is
$1+\nobreak4+\nobreak1=\allowbreak(2)3=\nobreak(3)2$.

Taking the multiplicities in this ramification structure into account, one
finds that if the ${}_2F_1$ of Definition~\ref{def:rdef} is lifted along
$x=R(u)$ rather than $x=\tilde R(s)$, the P\nobreakdash-symbol
identity~(\ref{eq:uPsym}) is replaced by
\begin{subequations}
\begin{align}
  r^{m}_{n}(u) &= \left[p_{\rm e}(u)\right]^{\frac1{12}+{m}+{n}}\,{}_2F_1\left(x=R(u)\right)\\
&=P
\left\{
\begin{array}{cccccc|c}
  0 &&& 1 & -(2\pm\sqrt3)^2 & \infty & u \\
\hline
  0 &&& 0 & 0 & -\tfrac14-3{m}-3{n} & \\
\frac14 +{m} &&& 1 + 4{m} & 1+3{n} & -2{m}-3{n} &
\end{array}
\right\}.
\label{eq:displayedP}
\end{align}
\end{subequations}
This P-symbol has five singular points (at most; fewer if $n=0$ or $m=0$).
By the preceding explanation, the points $u=0,1,\infty$ represent $1,4,1$
vertices of the octahedron, and each of $u=\allowbreak
-(2\pm\nobreak\sqrt3)^2$ represents a cycle of four face-centers.

\begin{theorem}
\label{thm:diffrec}
  The octahedral function\/ $r^{m}_{n}=r^{m}_{n}(u)$ satisfies the Fuchsian
  differential equation\/ ${\mathcal{L}}^{m}_{n} r^{m}_{n}=0$, where
  \begin{align*}
&   {\mathcal{L}}^{m}_{n} = D_u^2 +
\left[
\frac{3/4-{m}}{u} + \frac{-4{m}}{u-1} + \frac{-3{n}(2u+14)}{u^2 + 14u+1}
\right]\,D_u\\
&\qquad\qquad\quad{}+(1+12{m}+12{n})\: \frac{(2{m}+3{n})(u^2+14u+1) -
  12{n}(3u+1)}{4\,u(u-1)(u^2+14u+1)},
  \end{align*}
the P-symbol of which appears in Eq.\/~{\rm(\ref{eq:displayedP})}.  The
function\/~$r^{m}_{n}$ is the Frobenius solution associated to the zero
characteristic exponent of the singular point\/ $u=0$.  It satisfies eight
differential recurrences of the form
\begin{displaymath}
  K\,r^{{m}+\Delta {m}}_{{n}+\Delta {n}}
=
p_{\rm v}^{(-\sigma_{\rm v}+\varepsilon_{\rm v})/4}
\,p_{\rm e}^{-\sigma_{\rm e}+\varepsilon_{\rm e}}
\,p_{\rm f}^{-\sigma_{\rm f}+\varepsilon_{\rm f}}
\,(4u^{3/4})
\frac{{\rm d}}{{\rm d}u}
\left[\,
p_{\rm v}^{\sigma_{\rm v}/4}
p_{\rm e}^{\sigma_{\rm e}}\,p_{\rm f}^{\sigma_{\rm f}}\cdot r^{m}_{n}
\right],
\end{displaymath}
in which\/ $\Delta({n},{m}) = (0,\pm1)$,\/ $(\pm1,0)$ and\/ $(\pm1,\pm1)$.  For
each recurrence, the exponents\/ $\sigma_{\rm v}, \sigma_{\rm e},
\sigma_{\rm f}$, the exponents\/ $\varepsilon_{\rm v}, \varepsilon_{\rm e},
\varepsilon_{\rm f}$, and the prefactor\/~$K$ are listed in
Table\/~{\rm\ref{tab:1}}.
\end{theorem}

\begin{proof}
  The differential equation comes by applying to the appropriate Gauss
  equation of the form~(\ref{eq:GaussHE}) the changes of variable that
  perform (i)~the lifting along $u\mapsto x=R(u)$, and (ii)~the
  multiplication by $\left[p_{\rm e}(u)\right]^{\frac1{12} + {m} + {n}}$.
  Or, one can merely substitute $u=s^4$ into the equation
  $\tilde{\mathcal{L}}^{m}_{n} \tilde r^{m}_{n}=0$, with
  $\tilde{\mathcal{L}}^{m}_{n}$ as in~(\ref{eq:ude}).

  The differential recurrences of Jacobi, which shift the parameters
  $a,b,c$ of the function ${}_2F_1(a,b;c;x)$ by integers, are well known.
  (See \cite[2.8(20)--2.8(27)]{Erdelyi53}.)  And if
  $\Delta({n},{m})=(\pm1,\pm1)$, the ${}_2F_1(R(u))$ in the definition of
  $r^{m}_{n}(u)$ has its parameters shifted by integers.  (See
  Definition~\ref{def:rdef}.)  By some calculus, one can change the
  independent variable in the relevant differential recurrences of Jacobi
  from $x=R(u)$ to~$u$, thereby obtaining the final four recurrences in
  Table~\ref{tab:1} (the diagonal ones).  The change uses the fact that
  $u^{3/4}\,{\rm d}R/{\rm d}u$ equals $-108\,p_{\rm v}^{3/4}p_{\rm f}^2 /
  p_{\rm e}^3$, and the details are straightforward.

  The first four recurrences in the table, with
  $\Delta({n},{m})=(0,\pm1),(\pm1,0)$, come a bit less easily, because they
  shift the ${}_2F_1$ parameters in Definition~\ref{def:rdef} by
  half-integers rather than integers.  But by examination,
  Definition~\ref{def:rdef} is equivalent to
  \begin{equation}
    \label{eq:new2F1}
    {}_2F_1\left(
     {{-\tfrac1{12}-{m}-{n},\,\frac14-{m}}\atop{\tfrac12-2{m}}}
     \biggm| S(t)\right) = (1+6t-3t^2)^{1/4-3{m}-3{n}}\,r^{m}_{n}(-3t^2),
  \end{equation}
  where
  \begin{displaymath}
     S(t) \defeq \frac{36\,t(1+3t^2)^2}{(1+6t - 3t^2)^3} = 1 -
     \left(\frac{1-6t-3t^2}{1+6t-3t^2}\right)^3. 
   \end{displaymath}
  This follows by a quadratic hypergeometric transformation, $u$ being
  related to~$t$ by $u=-3t^2$, and $R$ to~$S$ by
  $R=S^2/\allowbreak(S-\nobreak2)^2$.  When $\Delta({n},{m})=(0,\pm1)$
  or~$(\pm1,0)$, the parameters of the ${}_2F_1$ in~(\ref{eq:new2F1}) are
  shifted by integers, and the same technique can be applied.\qed
\end{proof}

\begin{table}
% table caption is above the table
\caption{Parameters for the differential recurrences of Theorem~\ref{thm:diffrec}.}
\label{tab:1}       % Give a unique label
% For LaTeX tables use
\begin{tabular}{llll}
\hline\noalign{\smallskip}
$\Delta({n},{m})$ & $\sigma_{\rm v},\sigma_{\rm e},\sigma_{\rm f}$ & $\varepsilon_{\rm v},\varepsilon_{\rm e},\varepsilon_{\rm f}$ & $K$  \\
\noalign{\smallskip}\hline\noalign{\smallskip}
$(0,+1)$ & $-1-4{m},0, \frac58+\frac32{m}-\frac32{n}$ & $1,0,1$ & $-(1+4{m})$ \\
$(0,-1)$ & $0,0, -\frac18-\frac32{m}-\frac32{n}$ & $-3,0,1$ & $\frac{3(1+12{m}+12{n})(7-12{m}+12{n})}{4{m}-3}$ \\
$(+1,0)$ & $\frac76-2{m}+2{n},0,-1-3{n}$ & $1,0,1$ & $\frac16(7-12{m}+12{n})$ \\
$(-1,0)$ & $-\frac16-2{m}-2{n},0,0$ & $1,0,-2$ & $-\frac16(1+12{m}+12{n})$ \\
$(+1,+1)$ & $-1-4{m},\frac{13}{12}+{m}+{n},-1-3{n}$ & $1,1,1$ & $-(1+4{m})$ \\
$(-1,-1)$ & $0,-\frac{1}{12}-{m}-{n},0$ & $-3,1,-2$ & $\frac{3(1+12{m}+12{n})(-11+12{m}+12{n})}{4{m}-3}$ \\
$(-1,+1)$ & $-1-4{m},\frac{5}{12}+{m}-{n},0$ & $1,1,-2$ & $-(1+4{m})$ \\
$(+1,-1)$ & $0,\frac{7}{12}-{m}+{n},-1-3{n}$ & $-3,1,1$ & $\frac{3(-7+12{m}-12{n})(-19+12{m}-12{n})}{4{m}-3}$ \\
\noalign{\smallskip}\hline
\end{tabular}
\end{table}

The four three-term non-differential recurrences in Theorem~\ref{thm:recs}
follow by a familiar elimination procedure from the differential
recurrences of Theorem~\ref{thm:diffrec}, taken in pairs.  They are
analogous to the contiguity relations (or~`contiguous function relations')
of Gauss, for~${}_2F_1$, which follow by elimination from the differential
recurrences of Jacobi, though Gauss did not derive them in this way.

The explicit formulas for the functions ~$r^{m}_{n}$ with small $n,m$ given
in~\S\,\ref{subsec:polys} (see
Eqs.\ (\ref{eq:byhand1}),\allowbreak(\ref{eq:byhand2})) also follow
from the differential recurrences of Theorem~\ref{thm:diffrec}.

\begin{theorem}
\label{thm:hypergrep}
  For any\/ ${m}\in\mathbb{Z}$, the octahedral function\/ $r^{m}_{0}$ has the
  hypergeometric representation
  \begin{equation}
    \label{eq:hyprep}
    r^{m}_{0}(u) = {}_2F_1\bigl(-2{m},-\tfrac14-3{m};\,\tfrac34-{m} \bigm| u\bigr).
  \end{equation}
Thus when\/ ${m}\ge0$, $r^{m}_{0}$ is a degree-\/$2{m}$ hypergeometric
polynomial.\hfil\break Moreover, for any\/ $n,m\ge0$, $r^{m}_{n}$ is a
polynomial of degree\/ $3{n}+\nobreak2{m}$.
\end{theorem}
\begin{proof}
  When ${n}=0$, ${\mathcal{L}}^{m}_{n}f=0$ loses two singular points and
  degenerates to a Gauss hypergeometric equation of the
  form~(\ref{eq:GaussHE}), with independent variable~$u$ and parameters
  $a=-2{m}$, $b=-\frac14-\nobreak3{m}$, $c=\frac34-\nobreak{m}$.  Hence
  $r^{m}_{0}(u)$ has the claimed representation, and if ${m}\ge0$, is a
  degree\nobreakdash-$2{m}$ polynomial in~$u$.  It follows by induction
  from the differential recurrence with $\Delta({n},{m})=(+1,0)$ that
  $r^{m}_{n}(u)$ is a polynomial in~$u$ for all ${n}\ge0$.  It must be of
  degree $3{n}+\nobreak2{m}$, because in the P\nobreakdash-symbol
  of~${\mathcal{L}}^{m}_{n}$ [see~(\ref{eq:displayedP})], the only
  characteristic exponent at~$u=\infty$ that is a (nonpositive) integer is
  $-2{m}-\nobreak3{n}$.\qed
\end{proof}

The statement of this theorem includes additional claims that were made
in~\S\,\ref{subsec:polys}.  The following related theorem mentions the Heun
function ${Hn}(a,q;\alpha,\beta,\gamma,\delta\mid z)$, for the definition
of which see~\cite{Ronveaux95}.  This is a Frobenius solution (at~$z=0$) of
a canonical Fuchsian differential equation that has four singular points,
namely $z=0,1,a,\infty$, and an `accessory' parameter~$q$ that unlike
$\alpha,\beta,\gamma,\delta$, does not affect their characteristic
exponents.  It has a convergent expansion $\sum_{k=0}^\infty h_kz^k$, where
the~$\{h_k\}_{k=0}^\infty$ satisfy a second-order recurrence with
coefficients quadratic in~$k$.

\begin{theorem}
\label{thm:Heunrep}
  For any\/ ${n}\in\mathbb{Z}$, the octahedral function\/ $r^{0}_{n}$ has the
  Heun representation
  \begin{displaymath}
    r^{0}_{n}(u) = {Hn}\left(
\left(\tfrac{2+\sqrt3}{2-\sqrt3}\right)^2  \!,
\tfrac{9(2+\sqrt3)^2\,{n}(12{n}+1)}{4} ;\,
-3{n},-\tfrac14-3{n};\,
\tfrac34,-3{n};\,
-(2+\sqrt3)^2u
\right)
  \end{displaymath}
and the equivalent expansion\/ $\sum_{k=0}^\infty a_ku^k$, where\/
$\{a_k\}_{k=0}^\infty$ satisfy the second-order recurrence
\begin{align}
&  (k+1)(4k+3)\,a_{k+1} + \left[
14k(4k-12{n}-1)+9{n}(12{n}+1)\right]\,a_k      \label{eq:Heunlike}\\
&\qquad\qquad\qquad\qquad\qquad {}+ (k-3{n}-1)(4k-12{n}-5)\,a_{k-1} = 0,
\nonumber
\end{align}
with $a_0=1$, $a_{-1}=0$.  Thus when\/ ${n}\ge0$, $r^{0}_{n}$ is a degree-\/$3{n}$
Heun polynomial.
\end{theorem}
\begin{proof}
  If ${m}=0$, the $u=1$ singular point of~${\mathcal{L}}^{m}_{n}$ drops out, i.e.,
  becomes ordinary, and
  \begin{displaymath}
    {\mathcal{L}}^{0}_{n} = D_u^2 + 
\left[\frac{3/4}u + \frac{-3{n}(2u+14)}{u^2+14u+1}
\right]
\,D_u
+\frac{3{n}(1+12{n})}4 \, \frac{u+3}{u(u^2+14u+1)}.
  \end{displaymath}
The substitution $z=-(2+\sqrt3)^2 u$ reduces ${\mathcal{L}}^{0}_{n} f=0$ to the
standard Heun equation~\cite{Ronveaux95}, with the stated values of
$a,q;\alpha,\beta,\gamma,\delta$.  The recurrence~(\ref{eq:Heunlike}),
based on a second-order difference operator, comes by substituting
$f=r^{0}_{n}=\sum_{k=0}^\infty a_ku^k$ into ${\mathcal{L}}^{0}_{n} f=0$.\qed
\end{proof}

For general $(n,m)\in\mathbb{Z}^2$, ${\mathcal{L}}^{m}_{n}f=0$ has
five singular points.  The theory of such generalized Heun equations
is underdeveloped at present, but the coefficients of their series
solutions are known to satisfy third-order (i.e., four-term)
recurrences.

\begin{theorem}
\label{thm:genHeunrep}
  For any\/ $(n,m)\in\mathbb{Z}^2$, the octahedral function\/ $r^{m}_{n}$ has
  the expansion\/ $r_n^m(u)=\sum_{k=0}^\infty a_ku^k$, where\/
  $\{a_k\}_{k=0}^\infty$ satisfy the third-order recurrence
  \begin{align}
    & (k+1)(4k-4{m}+3)\,a_{k+1}\label{eq:gHeunrec}\\
    &\qquad
    {}+\left[ k(52k-36{m}-168{n}-13) - (2{m}-9{n})(12{m}+12{n}+1) \right]\,a_k\nonumber\\
    &\qquad
    {}-\left[(k-1)(52k-276{m}-144{n}-65)+2(14{m}+3{n})(12{m}+12{n}+1)\right]\,a_{k-1} \nonumber\\
    &\qquad
    {}-(k-2{m}-3{n}-2)(4k-12{m}-12{n}-9)\,a_{k-2}=0,\nonumber
  \end{align}
  with $a_0=1$, $a_{-1}=0$, $a_{-2}=0$.  Thus when\/ $n,m\ge0$, $r^{m}_{n}(u)$ is
  a degree-\/$(3{n}+2{m})$ generalized Heun polynomial.
\end{theorem}
\begin{proof}
  The recurrence comes by substituting $f=r^{m}_{n}=\sum_{k=0}^\infty a_ku^k$
  into ${\mathcal{L}}^{m}_{n} f=0$.\qed
\end{proof}

It can be shown that if ${m}=0$, the third-order (i.e.\ generalized Heun)
difference operator in~(\ref{eq:gHeunrec}) has the second-order
(i.e.\ Heun) difference operator in~(\ref{eq:Heunlike}) as a right factor;
and if ${n}=0$, it has a first-order (i.e.\ hypergeometric) difference
operator as a right factor, which is why the representation in
Theorem~\ref{thm:hypergrep} exists.  The coefficients of all these
difference operators are quadratic in~$k$.

\smallskip
As stated in Theorem~\ref{thm:nonpoly}, it is not merely the case that
when $n,m\ge0$, the rational function $r^{m}_{n}=r^{m}_{n}(u)$ is a
polynomial of degree $3{n}+\nobreak 2{m}$.  In each quadrant of the
$({n},{m})$-plane, it is the quotient of a polynomial of known degree
(the numerator) by a known polynomial (the denominator).  To obtain
the formulas in Theorem~\ref{thm:nonpoly} that refer to quadrants
other than the first, reason as follows.  Consider the second formula:
it says that if ${n},{m}\ge0$, $r^{-{m}-1}_{n}(u)$~equals a polynomial
of degree $1+\nobreak 3{n}+\nobreak 2{m}$, divided by $(1-\nobreak
u)^{3+4{m}}$.  This is proved by induction on~${n}$, the base case
(${n}=0$) being
\begin{equation}
\label{eq:nneg}
    r^{-{m}-1}_{0}(u) = (1-u)^{-3-4{m}}\,{}_2F_1\bigl(-1-2{m},-\tfrac14-{m};\,\tfrac34-{m} \bigm| u\bigr),
\end{equation}
which comes from~(\ref{eq:hyprep}) by Euler's transformation
of~${}_2F_1$.  The inductive step uses the differential recurrence
with $\Delta({n},{m}) = (+1,0)$, as in the proof of
Theorem~\ref{thm:hypergrep}.  In the same way, the third and fourth
formulas follow from the $\Delta({n},{m})=(-1,0)$ recurrence.

One sees from the four formulas in Theorem~\ref{thm:nonpoly} that
irrespective of quadrant, $r^{m}_{n}\sim {\rm const}\times\allowbreak
u^{3{n}+2{m}}$, which partially confirms the claims of
Theorem~\ref{thm:1}(i,ii).  A consequence of this asymptotic behavior
is that besides being the Frobenius solution associated to the
exponent~$0$ at $u=0$, $r^{m}_{n}$~is the Frobenius solution
associated to the exponent $-2{m}-\nobreak 3{n}$ at $u=\infty$, which
appeared in the P\nobreakdash-symbol~(\ref{eq:displayedP}).

Theorem~\ref{thm:1} states specifically that $r^{m}_{n}\sim d^{m}_{n}
\times\allowbreak u^{3{n}+2{m}}$, with $d^{m}_{n}$ defined
in~(\ref{eq:ddef}).  This too is proved by induction.  The base case
(${n}=0$) has sub-cases ${m}\ge0$ and ${m}\le0$, which follow by
elementary manipulations from (\ref{eq:hyprep}) and~(\ref{eq:nneg}),
respectively.  The inductions toward ${n}\ge0$ and ${n}\le0$ come from
the differential recurrences with $\Delta({n},{m}) = (\pm1,0)$, the
$u\to\infty$ asymptotics of which yield expressions for
$d^{m}_{{n}\pm1}\!/d^{m}_{n}$.  As one can check, these two expression
agree with what (\ref{eq:ddef}) predicts.

\smallskip
The only claim in~\S\,\ref{subsec:polys} remaining to be proved is
Theorem~\ref{thm:1}(iii): the statement that the conjugate function
$\overline r^{m}_{n}=\overline r^{m}_{n}(u)$ is related to
$r^{m}_{n}=r^{m}_{n}(u)$ by $\overline r^{m}_{n}(u)\propto
u^{3{n}+2{m}}\, r^{m}_{n}(1/u)$, or equivalently $r^{m}_{n}(u)\propto
u^{3{n}+2{m}}\, \overline r^{m}_{n}(1/u)$.  (The constant of
proportionality comes from $\overline
r^{m}_{n}(0)$ equalling unity.)

Just as one derives the differential equation ${\mathcal{L}}^{m}_{n}
r^{m}_{n}=0$ of Theorem~\ref{thm:diffrec} by `lifting and shifting,' one
can derive an equation $\overline {\mathcal{L}}^{m}_{n} \overline
r^{m}_{n}=0$ on the $u$\nobreakdash-sphere from the definition
of~$\overline r^{m}_{n}$, given in Definition~\ref{def:rdef}; and a further
equation satisfied by $\mathbf{r}^{m}_{n}(u)\defeq\allowbreak
u^{3{n}+2{m}}\, \overline r^{m}_{n}(1/u)$.  The latter turns~out to be
${\mathcal{L}}^{m}_{n}\mathbf{r}^{m}_{n}=0$, i.e., to be identical to the
equation of Theorem~\ref{thm:diffrec}.  But $\overline r^{m}_{n} =
\overline r^{m}_{n}(u)$, analytic at $u=0$, is the Frobenius solution
associated to the exponent~$0$ at~$u=0$ of $\overline {\mathcal{L}}^{m}_{n}
\overline r^{m}_{n}=0$.  Hence, $\mathbf{r}^{m}_{n}$ is the Frobenius
solution associated to the exponent $-2{m}-\nobreak 3{n}$ at~$u=\infty$ of
${\mathcal{L}}^{m}_{n}\mathbf{r}^{m}_{n}=0$.  But as was noted three
paragraphs ago, this is~$r^{m}_{n}$; so $\mathbf{r}^{m}_{n} \propto
r^{m}_{n}$, and Theorem~\ref{thm:1}(iii) follows.

\section{Biorthogonality of Octahedral Functions}
\label{sec:biorthog}

The octahedral functions $r^{m}_{n}=r^{m}_{n}(u)$, which are polynomials if
$n,m\ge0$, satisfy recurrences, such as the three-term ones of
Theorem~\ref{thm:recs}, that are quite unlike the ones satisfied by the
classical orthogonal polynomials.  But at~least if ${m}=0,-1$, it can be
shown that the family $\{r^{m}_{n}\}_{m\in\mathbb{Z}}$ displays
orthogonality on the $u$\nobreakdash-interval $[0,1]$, or rather a form of
biorthogonality.

The biorthogonality is best expressed in terms of the lifted functions
$\tilde r^{m}_{n}(s)\defeq r^{m}_{n}(u=s^4)$ of the last section, the full
domain of which is the Riemann $s$\nobreakdash-sphere in which the defining
octahedron is inscribed.  These are solutions of
$\tilde{\mathcal{L}}^{m}_{n}\tilde f=0$, where the
operator~$\tilde{\mathcal{L}}^{m}_{n}$ was defined in~(\ref{eq:ude}).  By
inspection, it has the simpler representation
\begin{equation}
  \tilde{\mathcal{L}}^{m}_{n} =
\left(q_{\rm v}^{2{m}}q_{\rm f}^{3{n}/2}\right)
D_s^2 \left(q_{\rm v}^{2{m}}q_{\rm f}^{3{n}/2}\right)^{-1}
+
\left\{
\left[-2{m}(1+2{m})\right]\frac{q_{\rm f}}{q_{\rm v}^2}
+
\left[144{n}(2+3{n})\right]\frac{q_{\rm v}^2}{q_{\rm f}^2}
\right\},
\label{eq:SLop}
\end{equation}
where $q_{\rm v}(s)=s(1-s^4)$ and $q_{\rm f}(s) = 1+14s^4+s^8$ are the
usual polynomials that equal zero at the five finite vertices and the eight
face-centers of the octahedron.

By~(\ref{eq:SLop}), $\tilde{\mathcal{L}}^{m}_{n}$ is conjugated by a
similarity transformation to a formally self-adjoint operator of the
Schr\"odinger type.  For any fixed~${m}$, the calculation of the
eigenfunctions $q_{\rm v}^{-2{m}}q_{\rm f}^{-3{n}/2}\tilde
r^{m}_{n}(s)$ of the latter, on the $s$\nobreakdash-interval $[0,1]$,
can be viewed as solving a Sturm--Liouville problem.  The coefficient
$144{n}(2+\nobreak 3{n})$ in~(\ref{eq:SLop}) plays the role of the
Sturm--Liouville eigenvalue, and $q_{\rm v}^2/q_{\rm f}^2$ that of the
Sturm--Liouville weight function.

Because the coefficient function $q_{\rm f}/q_{\rm v}^2$ diverges at
the endpoints $s=0,1$, this Sturm--Liouville problem is typically a
singular one.  To avoid a discussion of endpoint classifications and
boundary conditions, it is best to derive orthogonality results not
from~$\tilde{\mathcal{L}}^{m}_{n}$, but rather from the Love--Hunter
biorthogonality relation~(\ref{eq:orthgonalityLove}), i.e.,
\begin{equation}
\label{eq:orthgonalityLove2}
  \int_{-1}^1 {\rm P}_{\nu}^\mu(z) {\rm P}_{\nu'}^{-\mu}(-z) \,{\rm d}z =0,
\end{equation}
which holds if $\mu\in(-1,1)$ and $\nu,\nu'$ differ by a nonzero even
integer.  (See \cite[Appendix]{Love92} for a proof.)
Equation~(\ref{eq:orthgonalityLove2}) is a relation of orthogonality
between the eigenfunctions of a singular boundary value problem based
on~(\ref{eq:legendre}), the associated Legendre equation (i.e., ${\rm
  P}_{\nu_0+2n}^\mu(z)$, $n\in\mathbb{Z}$), and the eigenfunctions of
the \emph{adjoint} boundary value problem (i.e., ${\rm
  P}_{\nu_0+2n}^{-\mu}(-z)$, $n\in\mathbb{Z}$).  The first problem is
non-self-adjoint because the boundary conditions that single~out ${\rm
  P}_{\nu_0+2n}^\mu(z)$, $n\in\mathbb{Z}$, as eigenfunctions are not
self-adjoint.

However, one feature of the operator $\tilde{\mathcal{L}}^{m}_{n}$ must be
mentioned.  If $\tilde f=\tilde f(s)$ solves
$\tilde{\mathcal{L}}^{m}_{n}\tilde f=0$, then so does $(1-\nobreak
s)^{1+12{m}+12{n}}\tilde f\left((1+s)/(1-s)\right)$.  This claim can be
verified by a lengthy computation, but its correctness is indicated by the
P\nobreakdash-symbol of~$\tilde{\mathcal{L}}^{m}_{n}$, which appeared
in~(\ref{eq:uPsym3}).  The map $s\mapsto(1+\nobreak
s)/\allowbreak(1-\nobreak s)$ is a $90^\circ$ rotation of the
$s$\nobreakdash-sphere, and hence of the inscribed octahedron, around the
axis through the equatorial vertices $s=\pm{\rm i}$. This rotation takes
vertices to vertices, edges to edges, and faces to faces.  The subsequent
multiplication by $(1-\nobreak s)^{1+12{m}+12{n}}$ shifts the
characteristic exponents at the most affected vertices ($s=1,\infty$) to
the values they had before the rotation.

\begin{theorem}
\label{thm:biorthog}
For\/ ${m}=0$ and\/ ${m}=-1$, the lifted family\/ $\{\tilde r^{m}_{n}(s)
\defeq r^{m}_n(s^4)\}_{n\in\mathbb{Z}}$ is biorthogonal on the\/
$s$-interval\/ $[0,1]$ in the following sense: the inner product integral
\begin{align*}
&\int_0^1
\left[
q_{\rm v}^{-2{m}} q_{\rm f}^{-3{n}/2}(s)\cdot \tilde r^{m}_{n}(s)
\right]
\\
&\qquad{}\times\left[
q_{\rm v}^{-2{m}} q_{\rm f}^{-3{n}'/2}(s)\cdot (1-s)^{1+12{m}+12{n}}\: \tilde r^{m}_{{n}'}\!\left(\frac{1+s}{1-s}\right)
\right]
\,\frac{q_{\rm v}^2}{q_{\rm f}^2}(s)
\,{\rm d}s
\end{align*}
equals zero if\/ ${n},{n}'$ differ by a nonzero even integer.
\end{theorem}
\begin{proof}
Substitute the ${m}=0,-1$ cases of the formulas for ${\rm
  P}_{-\frac16+{n}}^{\pm\left(\frac14+{m}\right)}(\cos\theta)$ in
Theorem~\ref{thm:oct2} into~(\ref{eq:orthgonalityLove2}), and change the
variable of integration from $z=\cos\theta$ to $u=B_-/B_+$, and then to
$s=u^{1/4}$.  The involution $z\mapsto -z$ corresponds to $s\mapsto
(1-\nobreak s)/(1+\nobreak s)$.\qed
\end{proof}

This biorthogonality theorem is formulated so as to indicate its close
connection to Sturm--Liouville theory: evaluating the integral over
$0<\nobreak s<\nobreak1$ computes the inner product of the two
square-bracketed factors in the integrand, which come from ${\rm
  P}_\nu^\mu(z)$ and ${\rm P}_{\nu'}^{-\mu}(-z)$, with respect to the
weight function ${q_{\rm v}^2}/{q_{\rm f}^2}$.  The two factors are
eigenfunctions of adjoint Sturm--Liouville problems on $0<\nobreak
s<\nobreak1$ (i.e., ones with adjoint boundary conditions), with different
eigenvalues.

Theorem~\ref{thm:biorthog} cannot be extended to general $m\in\mathbb{Z}$,
because the integral diverges unless ${m}=0$ or ${m}=-1$, owing to rapid
growth of one or the other of the bracketed factors at each of the
endpoints $s=0,1$.  This divergence follows readily from the results
on~$r^{m}_{n}$ given in Theorems \ref{thm:1} and~\ref{thm:nonpoly}.
Alternatively, the divergence arises from the Ferrers function ${\rm
  P}_\nu^\mu$ not lying in $L^2[-1,1]$ when $\mu$~is non-integral, unless
${\rm Re}\,\mu\in(-1,1)$.

The formulas for the \emph{tetrahedral} Ferrers functions ${\rm
  P}_{-\frac34-{m}}^{\pm(\frac13+{n})}$ given in
Theorem~\ref{thm:tetr3} (cases ${m}=0,-1$) can also be substituted
usefully into the Love--Hunter relation~(\ref{eq:orthgonalityLove2}).
But the resulting statement of biorthogonality is more complicated
than Theorem~\ref{thm:biorthog} and is not given here.

\section{Cyclic and Dihedral Formulas (Schwarz Classes O and~I)}
\label{sec:cyclicdihedral}

This section derives parametric formulas for Legendre and Ferrers functions
that are cyclic or dihedral.  The formulas involve the Jacobi polynomials
$P_n^{(\alpha,\beta)}$\! and are unrelated to the octahedral and
tetrahedral ones in \S\S\,\,\ref{sec:octa}, \ref{sec:tetra1},
and~\ref{sec:tetra2}.  They are of independent interest, and subsume
formulas that have previously appeared in the literature.

As used here, `cyclic' and `dihedral' have extended meanings.  The
terms arise as follows.  The associated Legendre
equation~(\ref{eq:legendre}) has $(\mu,\mu,\allowbreak
2\nu+\nobreak1)$ as its (unordered, unsigned) triple of characteristic
exponent differences.  By the results of Schwarz on the algebraicity
of hypergeometric functions, this differential equation will have
\emph{only algebraic solutions} if $(\nu+\frac12,\mu)$ lies in
$(\pm\frac12,\pm\frac1{2k})+\mathbb{Z}^2$ or
$(\pm\frac1{2k},\pm\frac12)+\mathbb{Z}^2$, for some positive
integer~$k$.  These restrictions cause the equation to lie in
Schwarz's cyclic class (labelled~O here), resp.\ his dihedral
class~I\null.  The terms refer to the projective monodromy group of
the equation, which is a (finite) subgroup of~$PSL(2,\mathbb{R})$.

However, the formulas derived below are more general, in that they
allow $k$ to be arbitrary: they are formulas for \emph{continuously}
parametrized families of Legendre and Ferrers functions, which are
generically transcendental rather than algebraic.  Because of this, we
call a Legendre or Ferrers function cyclic, resp.\ dihedral, if
$(\nu+\frac12,\mu)$ lies in $(\pm\frac12,*)+\mathbb{Z}^2$,
resp.\ $(*,\pm\frac12)+\mathbb{Z}^2$; the asterisk denoting an
unspecified value.  That is, the degree~$\nu$ should be an integer or
the order~$\mu$ a half-odd-integer, respectively.

Explicit formulas in terms of Jacobi polynomials are derived
in~\S\,\ref{subsec:explicitcyclicdihedral}, and how dihedral Ferrers
functions can be used for expansion purposes is explained
in~\S\,\ref{sec:biorthog2}.

\subsection{Explicit Formulas}
\label{subsec:explicitcyclicdihedral}

The Jacobi polynomials $P_n^{(\alpha,\beta)}(z)$ are well known
\cite[\S\,10.8]{Erdelyi53}. They have the hypergeometric and Rodrigues
representations
\begin{subequations}
\begin{align}
\label{eq:Jacobirep2}
  P_n^{(\alpha,\beta)}(z) &=
  \frac{(\alpha+1)_n}{n!}\:
       {}_2F_1\left(
       {{-n,\,n+\alpha+\beta+1}\atop{\alpha+1}} \Bigm|
       {\frac{1-z}{2}}
       \right)\\
       &= \frac{(-1)^n}{2^nn!}\, (1-z)^{-\alpha}(1+z)^{-\beta}
       \,\frac{{\rm d}^n}{{\rm d}z^n}\left[
           (1-z)^{\alpha+n}(1+z)^{\beta+n}
           \right]
\end{align}
\end{subequations}
and are orthogonal on $[-1,1]$ with respect to the weight function
$(1-\nobreak x)^\alpha\allowbreak (1+\nobreak x)^\beta$, if
$\alpha,\beta>-1$ and the weight function is integrable.

Legendre and Ferrers functions that are \emph{cyclic} (i.e., of integer
degree) are easily expressed in~terms of Jacobi polynomials.

\begin{theorem}
  \label{thm:cyclic1}
  The formulas
  \begin{align*}
    &P^{\mu}_{-\frac12\pm(n+\frac12)}(z) = \frac{n!}{\Gamma(n-\mu+1)}\,
    {\left(
    \frac{z+1}{z-1}
    \right)}^{\mu/2}
    P_n^{(-\mu,\mu)}(z),\\
    &P^{\mu}_{-\frac12\pm(n+\frac12)}(\cosh\xi) = \frac{n!}{\Gamma(n-\mu+1)}\,
    [\coth(\xi/2)]^{\mu}\,
    P_n^{(-\mu,\mu)}(\cosh\xi)
  \end{align*}
  hold when $n$~is a non-negative integer, for $z\in(1,\infty)$ and
  $\xi\in(0,\infty)$.  {\rm(}In the degenerate case when $\mu-n$ is a
  positive integer, $P_{-\frac12\pm(n+\frac12)}^\mu\equiv0$.{\rm)}
\end{theorem}
\begin{proof}
  Compare the representations (\ref{eq:Prep})
  and~(\ref{eq:Jacobirep2}).\qed
\end{proof}

\begin{theorem}
  \label{thm:cyclic2}
  The formulas
  \begin{align*}
    &{\rm P}^{\mu}_{-\frac12\pm(n+\frac12)}(z) = \frac{n!}{\Gamma(n-\mu+1)}\,
    {\left(
    \frac{1+z}{1-z}
    \right)}^{\mu/2}
    P_n^{(-\mu,\mu)}(z),\\
    &{\rm P}^{\mu}_{-\frac12\pm(n+\frac12)}(\cos\theta) = \frac{n!}{\Gamma(n-\mu+1)}\,
    [\cot(\theta/2)]^{\mu}\,
    P_n^{(-\mu,\mu)}(\cos\theta)
  \end{align*}
  hold when $n$~is a non-negative integer, for $z\in(-1,1)$,
  $\xi\in(-\infty,\infty)$, and $\theta\in(0,\pi)$.  {\rm(}In the
  degenerate case when $\mu-n$ is a positive integer, ${\rm
    P}_{-\frac12\pm(n+\frac12)}^\mu\equiv0$.{\rm)}
\end{theorem}
\begin{proof}
  By analytic continuation of Theorem~\ref{thm:cyclic1}; or in effect, by
  letting $\xi={\rm i}\theta$.\qed
\end{proof}

By exploiting the $\hat Q\to P$ and ${\rm Q}\to{\rm P}$ reductions
(\ref{eq:PtoQ}) and~(\ref{eq:rmPtormQ}), one can derive additional formulas
from Theorems \ref{thm:cyclic1} and~\ref{thm:cyclic2}, for $\hat
Q^\mu_{-\frac12\pm(n+\frac12)}$ and~${\rm Q}^\mu_{-\frac12\pm(n+\frac12)}$
respectively.  However, the coefficients in (\ref{eq:PtoQ})
and~(\ref{eq:rmPtormQ}) diverge when $\mu\in\mathbb{Z}$.  Hence, following
this approach to formulas for $\hat Q_{-\frac12\pm(n+\frac12)}^m,
{\rm{Q}}_{-\frac12\pm(n+\frac12)}^m$, when $n$~is a non-negative integer
and $m$~an integer, requires the taking of a limit.  In the commonly
encountered case when $-n\le\nobreak m\le\nobreak n$ (but not otherwise),
the resulting expressions turn~out to be logarithmic.  Such expressions can
be computed in other ways~\cite[\S\,3.6.1]{Erdelyi53}.  Perhaps the best
method is to express $\hat Q_{-\frac12\pm(n+\frac12)}^m$ in~terms of
a~${}_2F_1$ by using~(\ref{eq:altQrep}), and then use known formulas for
logarithmic~${}_2F_1$'s~\cite{Detrich79}.

\smallskip
Legendre and Ferrers functions that are \emph{dihedral} (i.e., are of
half-odd-integer order) are the subject of the following theorems.
For conciseness, a special notation is used: $[A|B]_\pm$ signifies
$A$, resp.~$B$, in the $+$, resp.~$-$ case; and $\{C\}_{\alpha,\pm}$,
where $C$~depends on~$\alpha$, signifies the even or odd part of~$C$
under $\alpha\mapsto-\alpha$, i.e., $\frac12[C(\alpha)\pm\nobreak
  C(-\alpha)]$.

\begin{theorem}
  \label{thm:dihedral1}
  The formulas
  \begin{align*}
    &\hat Q_{-\frac12+\alpha}^{\pm(\frac12+m)}(z) =
    \sqrt{\frac{\pi}2}\: m!\,
    \left[1 \Bigm | \frac1{(\alpha-m)_{2m+1}}\right]_\pm \\
    &\qquad\qquad\qquad\qquad{}\times (z^2-1)^{-1/4}\,\bigl(z+\sqrt{z^2-1}\bigr)^{-\alpha}\,P_m^{(\alpha,-\alpha)}\!\left(\frac{z}{\sqrt{z^2-1}}\right),\\
    &\hat Q_{-\frac12+\alpha}^{\pm(\frac12+m)}(\cosh\xi) =
    \sqrt{\frac{\pi}2}\: m!\,
    \left[1 \Bigm | \frac1{(\alpha-m)_{2m+1}}\right]_\pm \\
    &\qquad\qquad\qquad\qquad{}\times (\sinh\xi)^{-1/2}\,
    {\rm e}^{-\alpha\xi}\:
         P_m^{(\alpha,-\alpha)}(\coth\xi)
  \end{align*}
  hold when $m$~is a non-negative integer, for $z\in(1,\infty)$ and
  $\xi\in(0,\infty)$.
\end{theorem}
\begin{proof}
  Combine Whipple's $\hat Q\to P$
  transformation~\cite[3.3(13)]{Erdelyi53}, which appeared as
  Eq.~(\ref{eq:whipple}), with the results in
  Theorem~\ref{thm:cyclic1}; and write $m$ for~$n$, and $-\alpha$
  for~$\mu$.\qed
\end{proof}

In these formulas, the proportionality of $\hat
Q_{-\frac12+\alpha}^{\pm(\frac12+m)}$ to each other is expected;
cf.~(\ref{eq:QtoQ}).  Also, the division in the `minus' case by
\begin{displaymath}
 (\alpha-m)_{2m+1} = (\alpha-m)\dots(\alpha+m),
\end{displaymath}
which equals zero if $\alpha=-m,\dots, m$, is not unexpected.  As was
noted in~\S\,\ref{sec:prelims}, $\hat Q_\nu^\mu$~is undefined if
$\nu+\nobreak\mu$ is a negative integer, except when
$\nu=-\frac32,-\frac52,\dots$, in which case $\hat
Q_\nu^{\nu+1},\dots, \hat Q_\nu^{-(\nu+1)}$ are defined.  This implies
that for $m=0,1,2,\dots$, $\hat Q_{-\frac12+\alpha}^{+(\frac12+m)}$~is
defined for all~$\alpha$, and that $\hat
Q_{-\frac12+\alpha}^{-(\frac12+m)}$ is undefined if and only if
$\alpha=-m,\dots,m$.

\begin{theorem}
  \label{thm:dihedral2}
  The formulas
  \begin{align*}
    &P_{-\frac12+\alpha}^{\pm(\frac12+m)}(z) =
    \sqrt{\frac2{\pi}}\: m!\,
    \left[(-1)^m \Bigm | \frac{(-1)^{m+1}}{(\alpha-m)_{2m+1}}\right]_\pm \\
    &\qquad\qquad\qquad\qquad{}\times (z^2-1)^{-1/4}\,\left\{\bigl(z+\sqrt{z^2-1}\bigr)^{-\alpha}\,P_m^{(\alpha,-\alpha)}\!\left(\frac{z}{\sqrt{z^2-1}}\right)\right\}_{\alpha,\pm},\\
    &P_{-\frac12+\alpha}^{\pm(\frac12+m)}(\cosh\xi) =
    \sqrt{\frac2{\pi}}\: m!\,
    \left[(-1)^m \Bigm | \frac{(-1)^{m+1}}{(\alpha-m)_{2m+1}}\right]_\pm \\
    &\qquad\qquad\qquad\qquad{}\times (\sinh\xi)^{-1/2}\,\left\{{\rm e}^{-\alpha\xi}\:P_m^{(\alpha,-\alpha)}(\coth\xi)\right\}_{\alpha,\pm}
  \end{align*}
  hold when $m$~is a non-negative integer, for $z\in(1,\infty)$ and
  $\xi\in(0,\infty)$; it being understood in the `minus' case that
  when $\alpha=-m,\dots,m$ and there is an apparent division by zero,
  each right-hand side requires the taking of a limit.
\end{theorem}
\begin{proof}
  Combine the $P\to\hat Q$ reduction~(\ref{eq:QtoP}) with the results in
  Theorem~\ref{thm:dihedral1}.\qed
\end{proof}

\begin{theorem}
\label{thm:dihedralPQ}
The formulas
\begin{align*}
    {\rm P}_{-\frac12+\alpha}^{\pm(\frac12+m)}(\cos\theta) &=
    \sqrt{\frac2{\pi}}\: m!\,
    \left[{\rm i}^{m} \Bigm | \frac{{\rm i}^{-m-1}}{(\alpha-m)_{2m+1}}\right]_\pm \\
    &\qquad\qquad{}\times (\sin\theta)^{-1/2}\,\left\{{\rm e}^{{\rm i}\alpha\theta}\:P_m^{(\alpha,-\alpha)}({\rm i}\,\cot\theta)\right\}_{\alpha,\pm},\\
    {\rm Q}_{-\frac12+\alpha}^{\pm(\frac12+m)}(\cos\theta) &=
    \sqrt{\frac{\pi}2}\: m!\,
    \left[{\rm i}^{m+1} \Bigm | \frac{{\rm i}^{-m}}{(\alpha-m)_{2m+1}}\right]_\pm \\
    &\qquad\qquad{}\times (\sin\theta)^{-1/2}\,\left\{{\rm e}^{{\rm i}\alpha\theta}\:P_m^{(\alpha,-\alpha)}({\rm i}\,\cot\theta)\right\}_{\alpha,\mp}
\end{align*}
hold when $m$~is a non-negative integer, for $\theta\in(0,\pi)$.  In
the sub-cases $\alpha=-m,\dots,m$ of the `minus' case, the apparent
division by zero in the first formula is handled by interpreting its
right-hand side in a limiting sense; but the division by zero in the
second formula causes both its sides to be undefined.
\end{theorem}
\begin{proof}
  The first formula follows by analytic continuation of the latter formula
  in Theorem~\ref{thm:dihedral2}; in effect, by letting $\xi=-{\rm
    i}\theta$.  The second formula then follows from the ${\rm Q}\to{\rm
    P}$ reduction~(\ref{eq:rmPtormQ}), after some algebraic
  manipulations.\qed
\end{proof}

As was noted in \S\,\ref{sec:prelims}, ${\rm Q}_\nu^\mu$~is undefined if
and only if $\hat Q_\nu^\mu$~is.  It was also noted that if
$\mu=\frac12,\frac32,\dots$, then ${\rm Q}_{-\mu}^\mu,\dots,{\rm
  Q}_{\mu-1}^\mu\equiv0$.  It follows that in the `plus' case of the second
formula of the theorem, the right-hand side must equal zero if
$\alpha=-m,\dots,m$.  This yields the interesting Jacobi-polynomial
identity
\begin{displaymath}
  {\rm e}^{{\rm i}\alpha\theta}\,P_m^{(\alpha,-\alpha)}({\rm i}\,\cot\theta)
  =
  {\rm e}^{-{\rm i}\alpha\theta}\,P_m^{(-\alpha,\alpha)}({\rm i}\,\cot\theta),
\end{displaymath}
which holds for $m=0,1,2,\dots$, when $\alpha=0,1,\dots,m$.

\subsection{Dihedral Ferrers Functions and Love--Hunter Expansions}
\label{sec:biorthog2}

In this subsection, we show that an expansion in dihedral Ferrers functions
can be, in effect, an expansion in Chebyshev polynomials (of the fourth
kind); and as an application, show that the result of~\cite{Pinsky99} on
the convergence of Love--Hunter expansions can be slightly extended.

The first formula on dihedral Ferrers functions in
Theorem~\ref{thm:dihedralPQ} specializes when $m=0$ to the known pair of
formulas~\cite[\S\,14.5]{Olver2010}
\begin{equation}
\label{eq:PPformulas}
{\rm
    P}_{-\frac12+\alpha}^{-\frac12}(\cos\theta) =\sqrt\frac2{\pi}\:
  \frac{\sin(\alpha\theta)}{\alpha\sqrt{\sin\theta}},
  \qquad
  {\rm P}_{-\frac12+\alpha}^{\frac12}(\cos\theta) =\sqrt\frac2{\pi}\:
  \frac{\cos(\alpha\theta)}{\sqrt{\sin\theta}}.
\end{equation}
These hold for $\theta\in(0,\pi)$, the $\alpha=0$ case of the former
requiring the taking of a limit.

Love--Hunter biorthogonality, i.e., the orthogonality of the functions
${\rm P}_\nu^\mu(z)$ and ${\rm P}_\nu^{-\mu}(-z)$ in $L^2[-1,1]$ when (i)
$\textrm{Re}\,\mu\in(-1,1)$ and (ii) $\nu,\nu'$ differ by an even integer
and are not half-odd-integers, specializes when $\mu=-\frac12$ and
$z=\cos\theta$ to
\begin{displaymath}
\int_0^\pi {\rm P}_\nu^{-\frac12}(\cos\theta){\rm P}_{\nu'}^{\frac12}(-\cos\theta)\,\sin\theta\,{\rm d}\theta = 0,
\end{displaymath}
and thus to
\begin{equation}
\label{eq:notwellknown}
\int_0^\pi \sin(\alpha\theta)\,\cos[\alpha'(\pi-\theta)]\,{\rm
  d}\theta = 0,
\end{equation}
which holds if $\alpha,\alpha'$ differ by an even integer.  (By continuity,
the restriction to $\alpha,\alpha'$ that are not integers can be dropped.)
The orthogonality in~(\ref{eq:notwellknown}) is not well known.

A Love--Hunter expansion of an `arbitrary' function $f=f(z)$ on $-1<z<1$ is
a bilateral expansion of~$f$ in the Ferrers functions ${\rm
  P}_{\nu_0+2n}^\mu$, of the form~(\ref{eq:seriesLove}), in which the
coefficients $\{c_n\}_{n\in\mathbb{Z}}$ are computed as inner products,
i.e.,
\begin{displaymath}
  c_n^{(\alpha)} = \frac
  {\int_{-1}^1 {\rm P}_{\nu_0+2n}^{-\mu}(-z)\,f(z)\,{\rm d}z}
  {\int_{-1}^1 {\rm P}_{\nu_0+2n}^{-\mu}(-z)\,{\rm P}_{\nu_0+2n}^{\mu}(z)\,{\rm d}z}.
\end{displaymath}
Existing results on the convergence of such expansions~\cite{Love94,Love92}
require that $\left|{\textrm{Re}}\,\mu\right|<\frac12$, or in the real
case, $\mu\in(-\frac12,\frac12)$.

It is of interest to examine whether convergence results can also be
obtained in the boundary cases $\mu=\pm\frac12$.  To treat the case when
$(\nu_0,\mu)=\allowbreak (-\frac12+\nobreak\alpha,-\frac12)$, define the
indexed ($n\in\mathbb{Z}$) and continuously parametrized
($\alpha\in\mathbb{R}$) functions
\begin{displaymath}
  \psi_{n}^{(\alpha)}(\theta) = \frac{\sin[(2n+\alpha)\theta]}{\sqrt{\sin\theta}},\qquad
  \chi_{n}^{(\alpha)}(\theta) = \frac{\cos[(2n+\alpha)(\pi-\theta)]}{\sqrt{\sin\theta}}\qquad
\end{displaymath}
on $0<\theta<\pi$, which are biorthogonal with respect to the weight
function $\sin\theta$.  (They differ only in normalization from ${\rm
  P}_{-\frac12+\alpha+2n}^{-\frac12}$ and ${\rm
  P}_{-\frac12+\alpha+2n}^{\frac12}$.)  In~terms of the first, one has a
formal $\mu=-\frac12$ Love--Hunter expansion
\begin{displaymath}
  f(\cos\theta) = \sum_{n=-\infty}^\infty c_n^{(\alpha)}\psi_{n}^{(\alpha)}(\theta),
\end{displaymath}
where
\begin{displaymath}
  c_n^{(\alpha)} = \frac
  {\int_0^\pi \chi^{(\alpha)}_{n}(\theta)\,f(\cos(\theta))\,\sin\theta\,{\rm d}\theta}
  {\int_0^\pi \chi^{(\alpha)}_{n}(\theta)\,\psi_{n}^{(\alpha)}(\theta)\,\sin\theta\,{\rm d}\theta}.
\end{displaymath}
The denominator inner product equals $(\pi/2)\sin(\alpha\pi)$ for all~$n$,
by examination; hence the restriction
$\alpha\in\mathbb{R}\setminus\mathbb{Z}$ must obviously be imposed.

This expansion is not fully satisfactory, because each
$\psi_{n}^{(\alpha)}(\theta)$ diverges as $\theta\to\pi^-$; though it
converges to zero, asymmetrically, as $\theta\to0^+$.  The underlying
problem is that if $\textrm{Re}\,\mu<0$, the function ${\rm P}_\nu^\mu(z)$
has leading behavior as~$z\to1^-$ proportional to $(1-\nobreak
z)^{-\mu/2}$, but its leading behavior as~$z\to(-1)^+$ comprises two terms:
one proportional to $(1+\nobreak z)^{-\mu/2}$, and one to $(1+\nobreak
z)^{+\mu/2}$.

In order (i)~to make endpoint behavior more symmetrical and less
divergent, and (ii)~to study endpoint convergence,
Pinsky~\cite{Pinsky99} has proposed modifying Love--Hunter expansions
by treating $[(1-\nobreak z)/\allowbreak (1+\nobreak
  z)]^{\mu/2}\,\allowbreak{\rm P}_{\nu}^\mu(z)$ rather than ${\rm
  P}_\nu^\mu(z)$ as the expansion function.  By~(\ref{eq:Prep}), this
amounts to replacing each ${\rm P}_\nu^\mu(z)$ by the ${}_2F_1$
function in~terms of which it is defined; i.e., performing a
\emph{hypergeometric} expansion.

Adopting the suggestion of~\cite{Pinsky99} when $\mu=-\frac12$ amounts
to replacing $\psi_n^{(\alpha)}\!,\chi_n^{(\alpha)}$ by versions that
are multiplied by $[(1-\nobreak z)/\allowbreak (1+\nobreak
  z)]^{-1/4}$, which equals $\cot^{1/2}(\theta/2)$.  With a trivial
change in normalization, these are the functions
\begin{equation}
\label{eq:psichihatdefs}
  \hat\psi_{n}^{(\alpha)}(\theta) =
  \frac{\sin[(2n+\alpha)\theta]}{\sin(\theta/2)},\qquad
  \hat\chi_{n}^{(\alpha)}(\theta) =
  \frac{\cos[(2n+\alpha)(\pi-\theta)]}{\sin(\theta/2)}\qquad
\end{equation}
($n\in\mathbb{Z}$) on $0<\theta<\pi$, which are biorthogonal with
respect to the weight function $\sin^2(\theta/2)$.  Each
$\hat\psi_{n}^{(\alpha)}(\theta)$ has a finite, nonzero limit as
$\theta\to0^+$ and $\theta\to\pi^-$, and as a function of
$z=\cos\theta$ is proportional to
\begin{displaymath}
  {}_2F_1\left(\frac12-\alpha-2n,\frac12+\alpha+2n;\,\frac32;\, \frac{1-z}2\right).
\end{displaymath}
In terms of these trigonometric functions $\hat\psi_{n}^{(\alpha)}(\theta)$,
one has (formally) a bilateral expansion of an arbitrary function $f=f(z)$
defined on $-1<z<1$, namely
\begin{equation}
\label{eq:Pinskyexp}
  f(z=\cos\theta) = \lim_{N\to\infty}
\sum_{n=-N}^N \hat c_n^{(\alpha)}\hat \psi_{n}^{(\alpha)}(\theta)
\end{equation}
for all $\theta\in(0,\pi)$, where
\begin{equation}
  \label{eq:hasden}
  \hat c_n^{(\alpha)} = \frac {\int_0^\pi
    \hat\chi^{(\alpha)}_{n}(\theta)\,f(\cos(\theta))\,\sin^2(\theta/2)\,{\rm
      d}\theta} {\int_0^\pi
    \hat\chi^{(\alpha)}_{n}(\theta)\,\hat\psi_{n}^{(\alpha)}(\theta)\,\sin^2(\theta/2)\,{\rm
      d}\theta}.
\end{equation}
The denominator in~(\ref{eq:hasden}) equals $(\pi/2)\sin(\alpha\pi)$
for all~$n$, as before.  One can clearly restrict $\alpha$ from
$\mathbb{R}\setminus\mathbb{Z}$ to the interval $(0,1)$ without
losing generality.

\begin{theorem}
If\/ $f=f(z)$ is piecewise continuous on\/ $-1\le z\le 1$, then in the
symmetric case\/ $\alpha=\frac12$, the bilateral series in\/
{\rm(\ref{eq:Pinskyexp})} will converge as\/ $N\to\infty$ to\/ $f(z)$ at
all points of continuity, including the endpoints, and in general to\/
$\left[f(z+)+\nobreak f(z-)\right]/2$.
\end{theorem}
\begin{proof}
The Chebyshev polynomials $W_j$ of the fourth kind, for $j=0,1,2,\dots$,
are defined by~\cite{Mason2003}
\begin{displaymath}
  W_j(\cos\theta) =
  \frac{\sin\left[(j+\frac12)\theta\right]}{\sin(\frac12\theta)}
  = \sum_{m=-j}^j {\rm e}^{{\rm i}m\theta}.
\end{displaymath}
It follows from (\ref{eq:psichihatdefs}) that when $n=0,1,2,\dots$,
both of
$\hat\psi_n^{(\frac12)}(\theta),\hat\chi_n^{(\frac12)}(\theta)$ equal
$W_{2n}(\cos\theta)$; and when $n=-1,-2,\dots$, they equal
$-W_{-2n-1}(\cos\theta)$.  The bilateral expansion
in~(\ref{eq:Pinskyexp}) thus reduces if $\alpha=\frac12$ to a
\emph{unilateral} expansion in the polynomials~$W_j$, $j=0,1,2,\dots$.
%% and the biorthogonality of
%% $\hat\psi_n^{(\frac12)}(\theta),\hat\chi_n^{(\frac12)}(\theta)$ to a
%% more conventional orthogonality.

The Chebyshev polynomials~$T_k$ of the first kind, for $k=0,1,2,\dots$, are
given by
\begin{displaymath}
T_k(\cos\theta) = {\cos(k\theta)}.
\end{displaymath}
By standard Fourier series theory, the expansion of $g=g(u)$ in
the~$T_k(u)$, when $g$~is piecewise continuous on $-1\le\nobreak
u\le\nobreak 1$, will converge to~$g$ at all points of continuity, and in
general to $\left[g(u+)+\nobreak g(u-)\right]/2$.  But (see
\cite[\S\,5.8.2]{Mason2003}), if one writes $z=1-\nobreak 2u^2$ (so that
$u=\sin(\theta/2)$ if $z=\cos\theta$), then $W_j(z)$ equals $(-1)^j u^{-1}
T_{2j+1}(u)$.  Therefore an expansion of $f=f(z)$ in the fourth-kind
$W_j(z)$ on $-1\le\nobreak z\le\nobreak 1$ is effectively an expansion of
$g(u)=\allowbreak uf(1-\nobreak2u^2)$ on $-1\le\nobreak u\le\nobreak1$ in
the first-kind $T_k(u)$, each even\nobreakdash-$k$ term of which must
vanish.  The theorem follows.\qed
\end{proof}

It is useful to compare this convergence result, which refers to an
expansion of~$f$ in the Ferrers functions ${\rm P}_{2n}^{-\frac12}$,
with the pointwise convergence result of~\cite{Pinsky99}.  The
latter deals with an expansion in the functions ${\rm
  P}_{\nu_0+2n}^{\mu}$, where $\nu_0$~is arbitrary and
$\mu\in(-\frac12,\frac12)$.  However, it requires that $f$ be
piecewise smooth, not merely piecewise continuous.

As the above theorem reveals, this assumption can be relaxed;
at~least, in the seemingly difficult `corner' case when
$(\nu_0,\mu)=(0,-\frac12)$.  Whether smoothness can also be dropped as
a hypothesis for the pointwise convergence of Love--Hunter expansions
with $\mu\in(-\frac12,\frac12)$, or with $(\nu_0,\mu)=\allowbreak
(-\frac12+\nobreak\alpha,-\frac12)$ when $\alpha\neq\frac12$, remains
to be explored.

\smallskip
It must be mentioned that the octahedral and tetrahedral formulas of
Theorems \ref{thm:oct2} and~\ref{thm:tetr3} facilitate the calculation of
the coefficients in Love--Hunter expansions of the form
(\ref{eq:seriesLove}), with $(\nu_0+\nobreak\frac12,\mu)$ equal to
$(\pm\frac13,\pm\frac14)$ and $(\pm\frac14,\pm\frac13)$, respectively.
Because these values satisfy $\mu\in(-\frac12,\frac12)$, the convergence
result of~\cite{Pinsky99} applies.

\section{Ladder Operators, Lie Algebras, and Representations}
\label{sec:last}

In the preceding sections, explicit formulas for the Legendre and
Ferrers functions in the octahedral, tetrahedral, dihedral, and cyclic
families were derived.  Each such family (in the first-kind Ferrers
case) is of the form $\{{\rm P}_{\nu_0+n}^{\mu_0+m}(z=\cos\theta)\}$,
where $\nu_0,\mu_0$ are or may be fractional, and $(n,m)$ ranges
over~$\mathbb{Z}^2$.  In this section, the connection between such a
family and conventional $SO(3)$-based harmonic analysis on the sphere
$S^2=SO(3)/SO(2)$, coordinatized by the angles~$(\theta,\varphi)$, is
briefly explored.

The connection goes through the corresponding family of generalized
spherical harmonics, ${\rm P}_\nu^\mu(\cos\theta){\rm e}^{{\rm
    i}\mu\varphi}$, with $(\nu,\mu)\in(\nu_0,\mu_0)+\mathbb{Z}^2$.
But the connection is not as strong as one would like.  If
$\nu_0,\mu_0$ are rational but not integral, these harmonic functions
will not be single-valued on the symmetric space~$S^2$.  (In the cases
of interest here, each ${\rm P}_\nu^\mu(z)$ in the family is algebraic
in~$z$, and they can be viewed as finite-valued.)  They may not be
square-integrable, because the leading behavior of ${\rm
  P}_\nu^\mu(z)$ as~$z\to1^-$ is proportional to $(1-\nobreak
z)^{-\mu/2}$ unless $\mu$~is a positive integer.

For these reasons, the focus is on the action of Lie \emph{algebras} (of
`infinitesimal transformations') on a function family of this type,
specified by~$(\nu_0,\mu_0)$, rather than the action of a Lie \emph{group}
such as~$SO(3)$.  The space spanned by the classical spherical harmonics
$Y_n^m(\theta,\varphi)\propto\allowbreak {\rm P}_n^m(\cos\theta){\rm
  e}^{{\rm i}m\varphi}$, with $n\ge0$ and $m\in\mathbb{Z}$, admits an
action of the rotation group $SO(3)$.  The Lie
algebra~$\mathfrak{so}(3,\mathbb{R})$ of $3\times3$ real skew-symmetric
matrices can be represented by differential operators on~$S^2$, with real
coefficients, and acts on the space of spherical harmonics.  The resulting
infinite-dimensional representation is reducible: for $n=0,1,2,\dots$, it
includes the usual $(2n+\nobreak1)$-dimensional representation on the span
of $Y_n^{-n},\dots,Y_n^n$.  But, $\mathfrak{so}(3,\mathbb{R})$ is not the
only Lie algebra to be considered.

A larger Lie algebra than $\mathfrak{so}(3,\mathbb{R})$ acts naturally
on the spherical harmonics, or rather, on the (regular) \emph{solid
  harmonics} ${r^nY_n^m(\theta,\varphi)}$, which satisfy Laplace's
equation on~$\mathbb{R}^3$.  (See~\cite[\S\,3.6]{Miller77}.)  This is
the $10$-dimensional real Lie algebra $\mathfrak{so}(4,1)$ that is
generated by `ladder' operators that increment and decrement the
degree~$n$, as well as the order~$m$.  They are represented by
differential operators on~$\mathbb{R}^3$, with real coefficients.  The
real span of these operators exponentiates to the Lie group $SO_0(4,1)$,
which contains as subgroups (i)~the $3$\nobreakdash-parameter group
$SO(3)$ of rotations about the origin, (ii)~a
$3$\nobreakdash-parameter abelian group of translations
of~$\mathbb{R}^3$, (iii)~a $1$\nobreakdash-parameter group of
dilatations (linear scalings of~$\mathbb{R}^3$), and (iv)~a
$3$\nobreakdash-parameter abelian group of `special conformal
transformations.'  The last are quadratic rational self-maps
of~$\mathbb{R}^3$ (or~rather the real projective space
$\mathbb{RP}^3$, because they can interchange finite and infinite
points).

The preceding results, now standard, are extended below to any family of
\emph{generalized} solid harmonics $\{r^\nu {\rm P}_\nu^\mu(\cos\theta){\rm
  e}^{{\rm i}\mu\varphi}\}$, with $(\nu,\mu)\in\allowbreak
(\nu_0,\mu_0)+\nobreak\mathbb{Z}^2$ for specified~$\nu_0,\mu_0$.
In~\S\,\ref{subsec:last1}, the differential and non-differential
recurrences on $\nu$ and~$\mu$ are derived.  (See Theorems
\ref{thm:diffrecs} and~\ref{thm:nondiffrecs}.)  In~\S\,\ref{subsec:last2},
it is shown that the ladder operators in the differential recurrences
generate a $10$\nobreakdash-dimensional real Lie algebra, and an
isomorphism from this algebra not to $\mathfrak{so}(4,1)$ but to
$\mathfrak{so}(3,2)$ is exhibited.  The treatment closely follows Celeghini
and del Olmo~\cite{Celeghini2013}, but the explicit isomorphism in
Theorem~\ref{thm:newso32} is new.

In the setting of special function identities, which typically involve
\emph{real} linear combinations of differential operators,
$\mathfrak{so}(3,2)$ arises more naturally than does
$\mathfrak{so}(4,1)$.  But by a limited form of complexification,
$\mathfrak{so}(3,2)$~can be converted to $\mathfrak{so}(4,1)$, and
indeed to~$\mathfrak{so}(5,\mathbb{R})$.  These are alternative real
forms of the rank\nobreakdash-2 complex Lie algebra
$\mathfrak{so}(5,\mathbb{C})$, to which they complexify, and the eight
displacement vectors $\Delta(\nu,\mu)=\allowbreak(0,\pm1)$,
$(\pm1,0)$, $(\pm1,\pm1)$ can be identified with the roots
of~$\mathfrak{so}(5,\mathbb{C})$.

In~\S\,\ref{subsec:last3}, it is shown that irrespective
of~$(\nu_0,\mu_0)$, the representation of $\mathfrak{so}(3,2)$ [or~of
  $\mathfrak{so}(4,1)$ or $\mathfrak{so}(5,\mathbb{R})$] carried by the
solid harmonics $r^\nu{\rm P}_\nu^\mu(\cos\theta){\rm e}^{{\rm
    i}\mu\varphi}$ with $(\nu,\mu)\in(\nu_0,\mu_0)+\mathbb{Z}^2$ is of a
special type: its quadratic Casimir operator takes a fixed value, and its
quartic one vanishes.  (See Theorem~\ref{thm:mostdegenerate}.)  The former
fact was found in~\cite{Celeghini2013}, but the latter is new.  The
representation of $\mathfrak{so}(3,2)$ on the solid harmonics of integer
degree and order, and its representation on the ones of half-odd-integer
degree and order, have irreducible constituents that are identified as the
known Dirac singleton representations of~$\mathfrak{so}(3,2)$.

\subsection{Differential and Non-differential Recurrences}
\label{subsec:last1}

In any family $\{{\rm P}_{\nu_0+n}^{\mu_0+m}(z)\}_{(n,m)\in\mathbb{Z}^2}$,
where ${\rm P}$ can be taken as any of ${\rm P},{\rm Q},P,\hat Q$, any
three distinct members are linearly dependent, over the field of functions
that are rational in~$z$ and $\sqrt{1-z^2}$ (Ferrers case) or
$\sqrt{z^2-1}$ (Legendre case).  In particular, any three contiguous
members are so related, by a three-term ladder recurrence.

The underlying recurrences are differential ones, which generally permit
any single~${\rm P}^\mu_\nu$ and its derivative to generate any member
contiguous to~it, as a linear combination; and by iteration, to generate
any ${\rm P}_{\nu+\Delta\nu}^{\mu+\Delta \mu}$ in which
$\Delta(\nu,\mu)\in\mathbb{Z}^2$.

\begin{theorem}
  \label{thm:diffrecs}
  The Ferrers functions\/ ${\rm P}_\nu^\mu = {\rm P}_\nu^\mu(z)$ satisfy
  eight differential recurrences, divided into four ladders, i.e., 
  $\pm$-pairs, with \/ $\Delta(\nu,\mu)=\pm(0,1)$, $\pm(1,0)$, $\pm(1,1)$,
  and\/ $\pm(1,-1)$.  Each pair is of the form
  \begin{displaymath}
    \alpha_\pm\,{\rm P}_{\nu+\Delta\nu}^{\mu +\Delta\mu} = \mp
    z^{-\sigma_0^\pm+\varepsilon_0}(1-z^2)^{-\sigma_1^\pm/2+\varepsilon_1/2}
    \frac{{\rm d}}{{\rm d}z}\left[
      z^{\sigma_0^\pm}(1-z^2)^{\sigma_1^\pm/2}\, {\rm P}_\nu^\mu
      \right],
  \end{displaymath}
  and for each pair, the exponents\/ $\sigma_0^\pm,\sigma_1^\pm$, the
  exponents\/ $\varepsilon_0,\varepsilon_1$, and the
  prefactor\/~$\alpha_\pm$ are given in Table\/~{\rm\ref{tab:2}}.  The
  second-kind functions\/~${\rm Q}_\nu^\mu$ satisfy identical
  recurrences.

  The Legendre functions\/ $P_\nu^\mu,Q_\nu^\mu$ {\rm[}the latter
    \emph{unnormalized}, i.e., the functions\/ ${\rm e}^{\mu\pi{\rm
        i}}\hat Q_\nu^\mu${\rm ]} satisfy recurrences obtained from
  the preceding by\/ {\rm(i)}~multiplying the right-hand side by a
  sign factor, equal to\/~${\rm i}^{\varepsilon_1+\Delta\mu}$; and
  {\rm(ii)}~replacing\/ $1-\nobreak z^2$ by\/ $z^2-\nobreak 1$.
\end{theorem}
\begin{proof}
  The four non-diagonal recurrences on the order and degree, with
  $\Delta(\nu,\mu)=\pm(0,1)$ and $\pm(1,0)$, are classical and can be
  found in many reference works
  \cite{Erdelyi53,Olver2010,Truesdell48}.  They can be deduced from
  the differential recurrences of Jacobi, which increment or decrement
  the parameters of the function ${}_2F_1(a,b;c;x)$.  (See
  \cite[2.8(20)--2.8(27)]{Erdelyi53}.)

  The final four diagonal ones, at least for ${\rm P}^\mu_\nu$ when
  $\nu,\mu$ are integers, are due to Celeghini and del
  Olmo~\cite{Celeghini2013}.  Each can be derived from the
  non-diagonal ones by a tedious process of elimination, but the
  process can be systematized as the calculation of the commutator of
  two differential operators.  (See~\S\,\ref{subsec:last2},
  below.)\qed
\end{proof}

\begin{table}
  \caption{Parameters for the differential recurrences of
    Theorem~\ref{thm:diffrecs}.  In the rightmost column, the notation
    $[a|b]_\pm$ signifies $a$, resp.~$b$, in the $+$, resp.~$-$ case.}
  \label{tab:2}
\begin{tabular}{llll}
\hline\noalign{\smallskip}
$\Delta({\nu},{\mu})$ & $\sigma_0^\pm,\sigma_1^\pm$ & $\varepsilon_0,\varepsilon_1$ & $\alpha_\pm$  \\
\noalign{\smallskip}\hline\noalign{\smallskip}
$\pm(0,1)$ & $0,\mp\mu$ & $0,1$ & $[1,(\nu+\mu)(\nu-\mu+1)]_\pm$ \\  
$\pm(1,0)$ & $0, \frac12 \pm(\nu+\frac12)$ & $0,2$ &  $[\nu-\mu+1,\nu+\mu]_\pm$\\  
$\pm(1,1)$ & $\frac12 \pm(\nu+\frac12)\pm\mu, \mp\mu$ & $1,1$ & $[1,(\nu+\mu)(\nu+\mu-1)]_\pm$ \\  
$\pm(1,-1)$ & $-\frac12 \pm(\nu+\frac12)\mp\mu, \pm\mu$ & $1,1$  & $[(\nu-\mu+1)(\nu-\mu+2),1]_\pm$ \\
\noalign{\smallskip}\hline
\end{tabular}
\end{table}

The differential recurrences satisfied by~${\rm P}_\nu^\mu$ can be
written in circular-trigonometric forms that will be needed below.
Substituting $z=\cos\theta$ yields
\begin{subequations}
\label{eq:rewrite}
\begin{align}
  \alpha_\pm {\rm P}_\nu^{\mu\pm1} &= \left[\pm D_\theta - \mu\cot\theta \right] {\rm P}_\nu^{\mu},\\
  \alpha_\pm {\rm P}_{\nu\pm1}^{\mu} &= \left\{\pm(\sin\theta)D_\theta + \left[(\nu+\tfrac12)\pm\tfrac12\right]\cos\theta     \right\} {\rm P}_\nu^{\mu},\\
  \alpha_\pm {\rm P}_{\nu\pm1}^{\mu\pm1} &= \left\{\pm (\cos\theta)D_\theta -
    \mu\csc\theta + \left[-(\nu+\tfrac12)\mp\tfrac12\right]\sin\theta\right\} {\rm P}_\nu^{\mu},\label{eq:rplusminus}\\
  \alpha_\pm {\rm P}_{\nu\pm1}^{\mu\mp1} &= \left\{\mp (\cos\theta)D_\theta -
    \mu\csc\theta + \left[+(\nu+\tfrac12)\pm\tfrac12\right]\sin\theta\right\} {\rm P}_\nu^{\mu},\label{eq:splusminus}
\end{align}
\end{subequations}
which are satisfied by ${\rm P}_\nu^\mu ={\rm P}_\nu^\mu(\cos\theta)$.
Here, $D_\theta \defeq {\rm d}/{\rm d}\theta$, and the four
prefactors~$\alpha_\pm$ are listed in the last column of
Table~\ref{tab:2}, in order.  The recurrences (\ref{eq:rplusminus})
and~(\ref{eq:splusminus}) have appeared in the literature but are not
well known; the only appearances that we have found, with $\nu,\mu$
restricted to integer values, are in \cite{Bogdanovic75} and
\cite[\S\,A.2]{Ilk83}.  Equations (\ref{eq:rplusminus})
and~(\ref{eq:splusminus}) imply each other because ${\rm
  P}_{-\nu-1}^\mu=\allowbreak{\rm P}_{\nu}^\mu$ for all $\nu,\mu$.
That is, ${\rm P}_\nu^\mu$ is unaffected by the negating of the
shifted degree parameter $\nu+\nobreak\frac12$.

\smallskip
The three-term ladder recurrences derived from the four pairs of
differential recurrences are given in the following theorem.  The
diagonal ones, coming from the ladders with $\Delta(\nu,\mu)=\pm(1,1)$
and~$\pm(1,-1)$, appear to be new.

\begin{theorem}
  \label{thm:nondiffrecs}
  The Ferrers functions\/ ${\rm P}_\nu^\mu={\rm P}_\nu^\mu(z)$ satisfy
  second-order\/ {\rm(}i.e., three-term{\rm)} recurrences on the
  order\/~$\mu$ and degree\/~$\nu$, namely
  \begin{gather*}
    \sqrt{1-z^2}\:{\rm P}_\nu^{\mu+1} + 2\mu z\,{\rm P}_\nu^\mu +
    (\nu+\mu)(\nu-\mu+1)\sqrt{1-z^2}\:{\rm P}_\nu^{\mu - 1} = 0,\\
    (\nu-\mu+1)\,{\rm P}_{\nu+1}^\mu - (2\nu+1)z\,{\rm P}_\nu^\mu +
    (\nu+\mu)\,{\rm P}_{\nu-1}^\mu=0,
  \end{gather*}
  and the two diagonal recurrences
  \begin{multline*}
    \sqrt{1-z^2}\:\,{\rm P}_{\nu\pm1}^{\mu+1}(z)
    +
    \bigl[\pm(2\nu+1)(1-z^2) + 2\mu\bigr]\,{\rm P}_\nu^\mu(z)
    \\
    +
    \bigl[(\nu+\tfrac12)\pm(\mu-\tfrac12)\bigr]\,
    \bigl[(\nu+\tfrac12)\pm(\mu-\tfrac32)\bigr]\,
    \sqrt{1-z^2}\:\,{\rm P}_{\nu\mp1}^{\mu-1}(z)
    =0.
  \end{multline*}
  The second-kind functions\/ ${\rm Q}_\nu^\mu$ satisfy
  identical second-order recurrences.

  The Legendre functions\/ $P_\nu^\mu,Q_\nu^\mu$\/ {\rm(}the latter
  \emph{unnormalized}, as above{\rm)}, satisfy recurrences obtained
  from the preceding by\/ {\rm(i)}~multiplying each term containing a
  function of order\/ $\mu+\delta$ and a coefficient proportional to\/
  $[\sqrt{1-z^2}]^\alpha$ by a sign factor, equal to\/ ${\rm
    i}^{\alpha-\delta}$; and\/ {\rm(ii)}~replacing\/ $\sqrt{1-z^2}$
  by\/ $\sqrt{z^2-1}$.
\end{theorem}
\begin{proof}
  Eliminate the derivative terms from the recurrences of
  Theorem~\ref{thm:diffrecs}.  This is the procedure used to derive
  Gauss's three-term, nearest-neighbor `contiguous function relations'
  for ${}_2F_1(a,b;c;x)$ from Jacobi's differential recurrences
  on~$a,b;c$.\qed
\end{proof}

It was noted in \S\,\ref{sec:prelims} that if $\nu+\mu$ is a negative
integer, ${\rm Q}_\nu^\mu$ and~${Q}_\nu^\mu$ are generally undefined
(though there are exceptions).  The recurrences for ${\rm Q}_\nu^\mu$
and~${Q}_\nu^\mu$ in Theorems \ref{thm:diffrecs} and~\ref{thm:nondiffrecs}
remain valid in a limiting sense even when $(\nu,\mu)$ is such that one or
more of the functions involved is undefined.

\subsection{Lie Algebras}
\label{subsec:last2}

The raising and lowering of the degree and order, in any doubly
indexed family of (generalized) solid harmonics
\begin{equation}
\label{eq:gensolid}
  \mathcal{S}_\nu^\mu =   \mathcal{S}_\nu^\mu(r,\theta,\varphi)
  \defeq
  r^\nu{\rm P}_\nu^\mu(\cos\theta){\rm e}^{{\rm i}\mu\varphi},
\end{equation}
where $(\nu,\mu)\in(\nu_0,\mu_0)+\mathbb{Z}^2$, can be performed by
differential operators that do not need to depend explicitly
on~$(\nu_0,\mu_0)$ if they are allowed to involve, instead, the
derivative operators $D_r,D_\varphi$ in addition to~$D_\theta$.  

The basic idea is due to Miller~\cite{Miller73b}, and there is freedom
in its implementation: either or both of the factors $r^\nu = {\rm
  e}^{\nu\log r}$ and ${\rm e}^{{\rm i}\mu\varphi}$ could include
an~`${\rm i}$' in its exponent, and the Ferrers functions ${\rm
  P}_\nu^\mu(\cos\theta)$ could be replaced by the Legendre ones
$P_\nu^\mu(\cosh\xi)$.  With the choices made in~(\ref{eq:gensolid}),
$\mathcal{S}_\nu^\mu$~can be viewed as a (typically multi-valued)
function of the spherical coordinates $r,\theta,\varphi$, which
satisfies Laplace's equation on~$\mathbb{R}^3$.  Define ladder
operators by
\begin{subequations}
\label{eq:JKRS}
  \begin{align}
    J_\pm &= {\rm e}^{\pm{\rm i}\varphi}\left[\pm D_\theta + {\rm i}(\cot\theta)D_\varphi\right],\\
    K_\pm &= r^{\pm1} \left[ \pm(\sin\theta)D_\theta +
      (\cos\theta)(rD_r+\tfrac12\pm\tfrac12)\right],\\
    R_\pm &= r^{\pm1}{\rm e}^{\pm{\rm i}\varphi}\left[\pm(\cos\theta)D_\theta
      +{\rm i}(\csc\theta)D_\varphi -
      (\sin\theta)(rD_r+\tfrac12\pm\tfrac12)\right],\\
    S_\pm &= r^{\pm1}{\rm e}^{\mp{\rm i}\varphi}\left[\mp(\cos\theta)D_\theta
      +{\rm i}(\csc\theta)D_\varphi +
      (\sin\theta)(rD_r+\tfrac12\pm\tfrac12)\right].
  \end{align}
\end{subequations}
Then, the differential recurrences~(\ref{eq:rewrite}) can be rewritten
in terms of the $\mathcal{S}_\nu^\mu$ as
\begin{subequations}
  \label{eq:JKRSplus}
  \begin{align}
    J_\pm\, {\mathcal{S}}_\nu^\mu &= [1,(\nu+\mu)(\nu-\mu+1)]_\pm \,{\mathcal{S}}_\nu^{\mu\pm1},
    \\
    K_\pm\, {\mathcal{S}}_\nu^\mu &= [\nu-\mu+1,\nu+\mu]_\pm \,{\mathcal{S}}_{\nu\pm1}^\mu,
    \\
    R_\pm\, {\mathcal{S}}_\nu^\mu &= [1,(\nu+\mu)(\nu+\mu-1)]_\pm \,{\mathcal{S}}_{\nu\pm1}^{\mu\pm1},
    \\
    S_\pm\, {\mathcal{S}}_\nu^\mu &= [(\nu-\mu+1)(\nu-\mu+2),1]_\pm \,{\mathcal{S}}_{\nu\pm1}^{\mu\mp1}.
  \end{align}
\end{subequations}
For each $(\nu_0,\mu_0)$, the solid harmonics
$\{\mathcal{S}_\nu^\mu=\mathcal{S}_{\nu_0+n}^{\mu_0+m}\}$, or more
accurately their real linear span, carry a representation of
$J_\pm,K_\pm,\allowbreak R_\pm,S_\pm$.  Generating a Lie algebra by
working commutators out, one finds (with $[A,B]$ signifying $AB-BA$)
\begin{displaymath}
  R_\pm = \pm[J_\pm,K_\pm],\qquad\quad S_\pm = \pm[J_\mp,K_\pm],
\end{displaymath}
which explains why the diagonal recurrences in Theorems
\ref{thm:diffrecs} and~\ref{thm:nondiffrecs} can be most efficiently
obtained by commutator calculations, as claimed.

It is useful additionally to define `labeling' or `maintaining' operators
$J_3,K_3$ by 
\begin{equation}
\label{eq:J3K3}
    J_3 = -{\rm i}D_\varphi,\qquad\qquad
    K_3 = rD_r+\tfrac12,
\end{equation}
so that
\begin{equation}
\label{eq:J3K3plus}
  J_3\,{\mathcal{S}}_\nu^\mu = \mu\,\mathcal{S}_\nu^\mu,
  \qquad\qquad   K_3\,{\mathcal{S}}_\nu^\mu = (\nu+\tfrac12)\,\mathcal{S}_\nu^\mu.
\end{equation}
By further calculations, one finds that the real Lie algebra generated
by $J_\pm,K_\pm$ \emph{closes}, in the sense that it is
finite-dimensional.  In particular,
\begin{alignat*}{3}
[J_3,J_\pm]&=\pm J_\pm,&&\qquad\qquad& [J_+,J_-]&=2J_3 ,\\  
[K_3,K_\pm]&=\pm K_\pm,&&\qquad\qquad& [K_+,K_-]&=-2K_3 ,\\
[R_3,R_\pm]&=\pm 2R_\pm,&&\qquad\qquad& [R_+,R_-]&=-4R_3 ,\\
[S_3,S_\pm]&=\pm 2S_\pm,&&\qquad\qquad& [S_+,S_-]&=-4S_3,
\end{alignat*}
where $R_3\defeq K_3+J_3$ and $S_3\defeq K_3-J_3$.  To interpret
these, recall that any real linear space with basis $\{X_+,X_-,X_3\}$,
given a Lie algebra structure by
\begin{displaymath}
  [X_3,X_\pm]=\pm X_\pm,\qquad\qquad [X_+,X_-]=2\sigma X_3,
\end{displaymath}
is isomorphic to $\mathfrak{so}(3,\mathbb{R})$ if $\sigma>0$, and to
$\mathfrak{so}(2,1)$ (or equivalently $\mathfrak{sl}(2,\mathbb{R})$) if
$\sigma<0$.  Hence, the real Lie algebras spanned by $\{J_+,J_-,J_3\}$,
$\{K_+,K_-,K_3\}$, $\{R_+,R_-,R_3\}$, and $\{S_+,S_-,S_3\}$, coming from
the ladders with $\Delta(\nu,\mu)=\pm(0,1)$, $\pm(1,0)$, $\pm(1,1)$, and\/
$\pm(1,-1)$, are isomorphic to $\mathfrak{so}(3,\mathbb{R})$ (the first)
and $\mathfrak{so}(2,1)$ (the remaining three).  The last two turn~out to
commute.

The real Lie algebra generated by $J_\pm,K_\pm$, of which these copies
of $\mathfrak{so}(3,\mathbb{R})$ and $\mathfrak{so}(2,1)$ are
subalgebras, is $10$\nobreakdash-dimensional and is spanned
over~$\mathbb{R}$ by $J_\pm,J_3;\allowbreak K_\pm,K_3;\allowbreak
R_\pm;S_\pm$.  It of course has real structure constants.  For
any~$(\nu_0,\mu_0)$, its representation by differential operators
on~$\mathbb{R}^3$, as above, is carried by the real span of the solid
harmonics $\mathcal{S}_{\nu_0+n}^{\mu_0+m}$, $(n,m)\in\mathbb{Z}^2$.
This result was obtained by Celeghini and del
Olmo~\cite{Celeghini2013}, though they confined themselves to
integer~$\nu,\mu$, i.e., in effect to
$(\nu_0,\mu_0)=(0,0)$.\footnote{The reader of~\cite{Celeghini2013}
  should note that in~\S\,5,
  $R_\pm\defeq[K_\pm,J_\pm]$ and $S_\pm\defeq[K_\pm,J_\mp]$ should be
  emended to read $R_\pm\defeq\mp[K_\pm,J_\pm]$ and
  $S_\pm\defeq\mp[K_\pm,J_\mp]$.}

To identify this $10$-dimensional real algebra, it is useful to relabel its
basis elements.  First, let
\begin{equation}
\label{eq:lastneeded1}
  (P_+,P_-,P_3)\defeq(S_+,-R_-,K_-),\qquad
  (C_+,C_-,C_3)\defeq(-R_+,S_-,K_+),
\end{equation}
in each of which the three elements commute.  The algebra can then be
viewed as the span over~$\mathbb{R}$ of $J_\pm,J_3$; $P_\pm,P_3$;
$C_\pm,C_3$ and~$K_3$, which will be written as~$D$ henceforth.  Define
\begin{subequations}
\label{eq:lastneeded2}
\begin{align}
  (PC_+^\pm,PC_-^\pm,PC_3^\pm) &\defeq \tfrac12\left[(P_+,P_-,P_3) \pm (C_+,C_-,C_3)\right]\\
&\hphantom{:}= \tfrac12(\mp R_++S_+,\,-R_-\pm S_-,\, K_-\pm K_+).
\end{align}
\end{subequations}
Also, for $X=J,P,C,PC^+,PC^-$, define the `skew-Cartesian' elements
\begin{displaymath}
\mathcal{X}_1\defeq   (X_++X_-)/2, \qquad
\mathcal{X}_2\defeq   (X_+-X_-)/2, \qquad
\mathcal{X}_3\defeq X_3,
\end{displaymath}
so that $X_\pm = \mathcal{X}_1\pm\mathcal{X}_2$.  The algebra will
then be the real span of $\mathcal{J}_1,\mathcal{J}_2,\mathcal{J}_3$;
$\mathcal{P}_1,\mathcal{P}_2,\mathcal{P}_3$;
$\mathcal{C}_1,\mathcal{C}_2,\mathcal{C}_3;D$, or equivalently of
$\mathcal{J}_1,\mathcal{J}_2,\mathcal{J}_3$;
$\mathcal{PC}^\pm_1,\mathcal{PC}^\pm_2,\mathcal{PC}^\pm_3$;~$D$.

It is readily verified that $\mathcal{J}_i$ commutes with
$\mathcal{PC}^{+}_i$ and $\mathcal{PC}^{-}_i$ for $i=1,2,3$, and that
\begin{subequations}
\label{eq:commreps1}
\begin{align}
[\mathcal{J}_i, \mathcal{J}_j] &=\{-1,+1,-1\}_k\:\mathcal{J}_k,\\
[\mathcal{J}_i, \mathcal{PC}^\pm_j] &=\{-1,+1,-1\}_k\:\mathcal{PC}_k^\pm,\\
[\mathcal{PC}^\pm_i, \mathcal{PC}^\pm_j] &=\mp\{-1,+1,-1\}_k\:\mathcal{J}_k,
\end{align}
\end{subequations}
where $i,j,k$ is any cyclic permutation of $1,2,3$, with $\{a,b,c\}_k$
meaning $a,b,c$ when $k=1,2,3$.  Also, the $3\times3$ matrix of
commutators $[\mathcal{PC}_i^+,\mathcal{PC}_j^-]$ indexed by $1\le
i,j\le 3$ equals $\textrm{diag}\,(-D,+D,-D)$.  Additionally,
\begin{equation}
\label{eq:commreps2}
  [D,\mathcal{J}_i] = 0, \qquad\qquad
  [D,\mathcal{PC}_i^\pm] = - \mathcal{PC}_i^\mp,
\end{equation}
for $i=1,2,3$.  These identities specify the structure of the algebra.

Now, recall that the real Lie algebra $\mathfrak{so}(p,q)$ with
$p+q=n$ has the following defining representation.  If
$\Gamma=(g_{ij})={\rm diag}\,(+1,\dots,+1,-1,\dots,-1)$, with
$q$\,~$+1$'s and $p$\,~$-1$'s, then $\mathfrak{so}(p,q)$ comprises all
real $n\times n$ matrices~$A$ for which $\Gamma A$ is skew-symmetric.
There is a sign convention here, and a $p\leftrightarrow q$ symmetry;
without loss of generality, $p\ge q$ will be assumed.  It is sometimes
useful to permute the $+1$'s and~$-1$'s.

More concretely, $\mathfrak{so}(p,q)$ can be realized as the real span
of the $n\times n$ matrices ${\bf M}_{ab}$, $1\le a< b \le n$, where
${\bf M}_{ab} =\allowbreak \Gamma\mathcal{E}_{ab} -\nobreak
\mathcal{E}_{ba}\Gamma$.  In this, $\mathcal{E}_{ab}$ is the $n\times
n$ matrix with a~$1$ in row~$a$, column~$b$, and zeroes elsewhere.
One often extends the size-$\left({{n}\atop2}\right)$ basis $\{{\bf
  M}_{ab}\}$ to a `tensor operator,' i.e., a skew-symmetric $n\times
n$ matrix of elements~$({\bf M}_{ab})$, by requiring that ${\bf
  M}_{ba}=-{\bf M}_{ab}$ for $1\le a,b\le n$.  The commutation
relations
\begin{equation}
\label{eq:commreps3}
  [{\bf M}_{ab},{\bf M}_{cd}] =
  g_{ad}{\bf M}_{bc} + g_{bc}{\bf M}_{ad} - g_{ac}{\bf M}_{bd} - g_{bd}{\bf M}_{ac}
\end{equation}
are easily checked.

\begin{theorem}
\label{thm:newso32}
  The real Lie algebra generated by\/ $J_\pm,K_\pm$ is isomorphic to\/
  $\mathfrak{so}(3,2)$, an isomorphism being specified by the tensor
  operator
  \begin{displaymath}
    ({\bf M}_{ab}) =
    \left(
    \begin{array}{cc|ccc}
      0 & \mathcal{PC}_2^- & -\mathcal{PC}_1^- & -\mathcal{PC}_3^- &  -D \\
      -\mathcal{PC}_2^- & 0 & \mathcal{J}_3 & -\mathcal{J}_1 &  \mathcal{PC}_2^+\\
      \hline
      \mathcal{PC}_1^- & -\mathcal{J}_3 & 0 & \mathcal{J}_2 & -\mathcal{PC}_1^+ \\
      \mathcal{PC}_3^- & \mathcal{J}_1 & -\mathcal{J}_2 & 0 & -\mathcal{PC}_3^+ \\
      D & -\mathcal{PC}_2^+ & \mathcal{PC}_1^+ & \mathcal{PC}_3^+ & 0
    \end{array}
    \right)
  \end{displaymath}
with\/ $\Gamma={\rm diag}\,(+1,+1,-1,-1,-1)$.
\end{theorem}
\begin{proof}
  By inspection.\qed
\end{proof}

This $\mathfrak{so}(3,2)$ has the Cartan decomposition
$\mathfrak{k}\oplus_\mathbb{R}\mathfrak{p}$, where
  \begin{align*}
    \mathfrak{k} &= {\rm sp}_{\mathbb{R}}\{ \mathcal{J}_2,
    \mathcal{PC}^+_1, \mathcal{PC}^+_3, \mathcal{PC}^-_2 \},\\
    \mathfrak{p} &= {\rm sp}_{\mathbb{R}}\{ \mathcal{J}_1,\mathcal{J}_3,\mathcal{PC}^+_2,\mathcal{PC}_1^-,\mathcal{PC}_3^-,D \}
  \end{align*}
are the `compact' and `non-compact' subspaces.  (The terms refer to the Lie
subgroups of $SO_0(3,2)$ to which they exponentiate.)  Real Lie algebras
isomorphic to $\mathfrak{so}(4,1)$ and $\mathfrak{so}(5,\mathbb{R})$ can be
obtained by Weyl's trick of redefining some or all of the basis elements
of~$\mathfrak{p}$ to include `{\rm i}'~factors.  In doing this, a slightly
changed notation will be useful.  For $X=J,P,C,PC^+,PC^-$, define the
`Cartesian' elements
\begin{equation}
\label{eq:cartesian}
  X_1 \defeq (X_++X_-)/2,\qquad X_2 \defeq -{\rm i}(X_+-X_-)/2,
\end{equation}
so that $X_\pm=X_1\pm{\rm i}X_2$ and
$(\mathcal{X}_1,\mathcal{X}_2,\mathcal{X}_3) = (X_1,{\rm i}X_2,X_3)$.

\begin{theorem}
\label{thm:newso41}
  {\rm (i)}~The real span of\/ ${\rm i}J_1,{\rm i}J_2, {\rm i}J_3;
  P_1,P_2,P_3; C_1,C_2,C_3; D$, or equivalently of\/ ${\rm
    i}J_1,\allowbreak{\rm i}J_2, {\rm i}J_3;\allowbreak
  PC_1^\pm,PC_2^\pm,PC_3^\pm; D$, is a real Lie algebra isomorphic to\/
  $\mathfrak{so}(4,1)$, an isomorphism being specified by the tensor
  operator
  \begin{displaymath}
    ({\bf M}_{ab}) =
    \left(
    \begin{array}{c|cccc}
      0 & -{PC}_1^- & -{PC}_2^- & -{PC}_3^- &  -D \\
      \hline
      {PC}_1^- & 0 & -{\rm i}{J}_3 & {\rm i}{J}_2 & -{PC}_1^+\\
      {PC}_2^- & {\rm i}{J}_3 & 0 & -{\rm i}{J}_1 & -{PC}_2^+ \\
      {PC}_3^- & -{\rm i}{J}_2 & {\rm i}{J}_1 & 0 & -{PC}_3^+ \\
      D & {PC}_1^+ & {PC}_2^+ & {PC}_3^+ & 0
    \end{array}
    \right)
  \end{displaymath}
  with\/ $\Gamma={\rm diag}\,(+1,-1,-1,-1,-1)$.\hfil\break {\rm
    (ii)}~The real span of\/ ${\rm i}J_1,{\rm i}J_2, {\rm
    i}J_3;\allowbreak {\rm i}PC_1^+,{\rm i}PC_2^+,{\rm i}PC_3^+;
  \allowbreak PC_1^-,PC_2^-,PC_3^-; {\rm i}D$ is a real Lie algebra
  isomorphic to\/ $\mathfrak{so}(3,2)$, an isomorphism being specified
  by the tensor operator
  \begin{displaymath}
    ({\bf M}_{ab}) =
    \left(
    \begin{array}{c|ccc|c}
      0 & -{PC}_1^- & -{PC}_2^- & -{PC}_3^- &  -{\rm i}D \\
      \hline
      {PC}_1^- & 0 & -{\rm i}{J}_3 & {\rm i}{J}_2 & -{\rm i}{PC}_1^+\\
      {PC}_2^- & {\rm i}{J}_3 & 0 & -{\rm i}{J}_1 & -{\rm i}{PC}_2^+ \\
      {PC}_3^- & -{\rm i}{J}_2 & {\rm i}{J}_1 & 0 & -{\rm i}{PC}_3^+ \\
      \hline
      {\rm i}D & {\rm i}{PC}_1^+ & {\rm i}{PC}_2^+ & {\rm i}{PC}_3^+ & 0
    \end{array}
    \right)
  \end{displaymath}
with\/ $\Gamma={\rm diag}\,(+1,-1,-1,-1,+1)$.
\end{theorem}
\begin{proof}
(i)~Multiply the second row and the second column of the $({\bf
    M}_{ab})$ in Theorem~\ref{thm:newso32} by~`${\rm i}$', and
  (innocuously) interchange the second and third rows, and the second
  and third columns.  (ii)~Continuing (or~in a sense reversing),
  multiply the last row and the last column by~`${\rm i}$'.\qed
\end{proof}

\begin{theorem}
\label{thm:newso5r}
  The real span of\/ ${\rm i}J_1,{\rm i}J_2, {\rm i}J_3;
  PC_1^+,PC_2^+,PC_3^+; {\rm i}PC_1^-,{\rm i}PC_2^-,{\rm i}PC_3^-;
  {\rm i}D$ is a real Lie algebra isomorphic to\/
  $\mathfrak{so}(5,\mathbb{R})$, an isomorphism being specified by the
  tensor operator
  \begin{displaymath}
    ({\bf M}_{ab}) =
    \left(
    \begin{array}{ccccc}
      0 & -{\rm i}{PC}_1^- & -{\rm i}{PC}_2^- & -{\rm i}{PC}_3^- &  -{\rm i}D \\
      {\rm i}{PC}_1^- & 0 & -{\rm i}{J}_3 & {\rm i}{J}_2 & -{PC}_1^+\\
      {\rm i}{PC}_2^- & {\rm i}{J}_3 & 0 & -{\rm i}{J}_1 & -{PC}_2^+ \\
      {\rm i}{PC}_3^- & -{\rm i}{J}_2 & {\rm i}{J}_1 & 0 & -{PC}_3^+ \\
      {\rm i}D & {PC}_1^+ & {PC}_2^+ & {PC}_3^+ & 0
    \end{array}
    \right)
  \end{displaymath}
with\/ $\Gamma={\rm diag}\,(-1,-1,-1,-1,-1)$.
\end{theorem}
\begin{proof}
Multiply the first row and the first column of the $({\bf M}_{ab})$ in
part~(i) of Theorem~\ref{thm:newso41} by~`${\rm i}$'.\qed
\end{proof}

With `${\rm i}$' factors in basis elements, the $\mathfrak{so}(3,2)$,
$\mathfrak{so}(4,1)$ and $\mathfrak{so}(5,\mathbb{R})$ of Theorems
\ref{thm:newso41} and~\ref{thm:newso5r} look awkward.  But it follows from
(\ref{eq:JKRS}), (\ref{eq:J3K3}), and~(\ref{eq:cartesian}) that each basis
element in the $\mathfrak{so}(4,1)$ of Theorem~\ref{thm:newso41}(i), i.e.,
each of ${\rm i}J_i$, $P_i$, $C_i$, and~$D$, is realized by a differential
operator in $r,\theta,\varphi$ with \emph{real} coefficients.  This is not
the case for the basis elements of $\mathfrak{so}(3,2)$
and~$\mathfrak{so}(5,\mathbb{R})$.

The geometric significance of this realization of $\mathfrak{so}(4,1)$
is revealed by changing from spherical coordinates
$(r,\theta,\varphi)$ to Cartesian ones, $(x_1,x_2,x_3)$
on~$\mathbb{R}^3$.  Using
\begin{displaymath}
  x\defeq (x_1,x_2,x_3) = (r\sin\theta\cos\varphi, r\sin\theta\sin\varphi, r\cos\theta),
\end{displaymath}
one finds that
\begin{subequations}
\label{eq:41ops}
  \begin{alignat}{2}
    {\rm i}J_i &= x_j\partial_k - x_k\partial_j, &\qquad& i=1,2,3, \\
    P_i &= \partial_i, &\qquad& i=1,2,3, \\
    \label{eq:Ci}
    C_i &= x_i - (x\cdot x)\partial_i + 2x_i(x\cdot \partial), &\qquad & i=1,2,3, \\
    \label{eq:D}
    D &= x\cdot\partial +\tfrac12, &\qquad&
  \end{alignat}
\end{subequations}
where $\partial_i\defeq {\rm d}/{\rm d}x_i$ and $i,j,k$ is a cyclic
permutation of~$1,2,3$.  That is, the~${\rm i}J_i$ generate rotations
about the origin, the~$P_i$ generate translations, and $D$~generates
dilatations (linear scalings of~$\mathbb{R}^3$).  The~$C_i$ generate
special conformal transformations, which are degree\nobreakdash-2
rational maps of~$\mathbb{R}^3$ (or rather~$\mathbb{RP}^3$) to itself.
The commutation relations
%% \begin{align*}
%% [J_i,J_j] &= {\rm i}\epsilon_{ijk} J_k,\\
%% [J_i,PC_j^\pm] &= {\rm i}\epsilon_{ijk}PC_k^\pm,\\
%% [D,PC_i^\pm] &= - PC_i^\mp,\\
%% [PC_i^\pm,PC^\pm_j] &= \mp{\rm i}\epsilon_{ijk}J_k,\\
%% [PC_i^+,PC_j^-] &= -\delta_{ij}D,\\
%% [D,J_i] &=0\\
%% \end{align*}
\begin{alignat*}{2}
[J_i,J_j] &= {\rm i}\epsilon_{ijk} J_k, &\qquad [PC_i^\pm,PC^\pm_j] &= \mp{\rm i}\epsilon_{ijk}J_k,\\
[J_i,PC_j^\pm] &= {\rm i}\epsilon_{ijk}PC_k^\pm, &\qquad [PC_i^+,PC_j^-] &= -\delta_{ij}D,\\
[D,J_i] &=0, & \qquad [D,PC_i^\pm] &= - PC_i^\mp,
\end{alignat*}
written in terms of $PC_i^\pm = \tfrac12(P_i\pm C_i)$, follow either from
(\ref{eq:commreps1}),(\ref{eq:commreps2}), from (\ref{eq:commreps3}), or
from (\ref{eq:41ops}).  Here, the summation convention of tensor analysis
is employed.  The Levi-Civit\`a tensor $\epsilon_{ijk}$ is skew-symmetric
in all indices, with $\epsilon_{123}=+1$, and $\delta_{ij}$ is the
Kronecker delta.  Together with
\begin{equation}
\label{eq:JiDdef}
J_1 = \frac12(J_+ + J_-), \qquad J_2 = -\frac{{\rm i}}2(J_+
- J_-),\qquad D=K_3,
\end{equation}
the formulas
\begin{subequations}
\label{eq:PCidef}
\begin{align}
PC_1^\pm &= \tfrac14 (\mp R_+ - R_-  \pm S_+ + S_-),\\  
PC_2^\pm &= -\tfrac{\rm i}4 (\mp R_+ + R_- \mp S_++ S_- ),\\  
PC_3^\pm &= \tfrac12(\pm K_+ + K_-),
\end{align}
\end{subequations}
express all these differential operators in terms of the original
$J_\pm,K_\pm,\allowbreak R_\pm,S_\pm;\allowbreak J_3,K_3$ of
(\ref{eq:JKRS}),(\ref{eq:J3K3}).

The ten operators in~(\ref{eq:41ops}) span (over~$\mathbb{R}$) the Lie
algebra of \emph{conformal} differential operators on~$\mathbb{R}^3$, which
is known to have an $\mathfrak{so}(4,1)$ structure.  (See
Miller~\cite[\S\,3.6]{Miller77}.)  This is the symmetry algebra of the
Laplacian~$\nabla^2$ on~$\mathbb{R}^3$, which comprises all real
first-order operators~$L$ for which $[L,\nobreak\nabla^2]\propto\nabla^2$,
i.e., for which $[L,\nabla^2]$ has $\nabla^2$ as a right factor.  It can be
viewed as acting on any suitable space of functions on~$\mathbb{R}^3$, and
exponentiates to the group $SO_0(4,1)$ of conformal transformations,
realized as flows on~$\mathbb{R}^3$ (or~$\mathbb{RP}^3$).  But the starting
point used here was their action on the span of the generalized solid
harmonics $\mathcal{S}_{\nu_0+n}^{\mu_0+m}$ with $(n,m)\in\mathbb{Z}^2$,
which are (multi-valued) solutions of Laplace's equation.

In~the physics literature on conformal Lie algebras and groups, the terms
`$x_i$' in~(\ref{eq:Ci}) and `$\tfrac12$' in~(\ref{eq:D}) often appear as
$2\delta x_i$ and~$\delta$ respectively, where $\delta$~is the so-called
scaling dimension; though the resulting commutation relations do not
involve~$\delta$.  The value $\delta=\tfrac12$ is specific to the symmetry
algebra of the Laplacian.

There are many variations on the present technique of using differential
recurrences to construct real Lie algebras, realized by differential
operators, that are isomorphic to the real forms
of~$\mathfrak{so}(5,\mathbb{C})$.  The solid
harmonics~$\mathcal{S}_\nu^\mu$ that were employed here are extensions
to~$\mathbb{R}^3$ of the (surface) spherical harmonics ${\rm
  P}_\nu^\mu(\cos\theta){\rm e}^{{\rm i}\mu\varphi}$ on the symmetric space
$S^2=SO(3)/SO(2)$.  If not Ferrers but Legendre functions were used, the
starting point would be the hyperboloidal ones $P_\nu^\mu(\cosh\xi){\rm
  e}^{{\rm i}\mu\varphi}$, defined using coordinates $(\xi,\varphi)$ on the
hyperboloid $H^2=SO(2,1)/SO(2)$, i.e., the surface $x_1^2 + x_2^2 - x_3^2 +
\rm{const}=0$.  Their extensions to~$\mathbb{R}^3$ satisfy the
$(2+\nobreak1)$-dimensional wave equation, rather than Laplace's equation.
(See \cite[Chap.~4]{Miller77} and~\cite{Durand2003b}.)  But isomorphic
algebras could be constructed.

\subsection{Lie Algebra Representations}
\label{subsec:last3}

In~\S\,\ref{subsec:last2}, it was shown that for any $(\nu_0,\mu_0)$,
there are representations of the real Lie algebras
$\mathfrak{so}(3,2)$, $\mathfrak{so}(4,1)$,
$\mathfrak{so}(5,\mathbb{R})$ that are carried by the span of the
family of (generically multi-valued) solid harmonics
$\mathcal{S}_\nu^\mu(r,\theta,\varphi)$,
$(\nu,\mu)\in(\nu_0,\mu_0)+\mathbb{Z}^2$.  These arise from the action
of the ladder operators on the Ferrers functions ${\rm
  P}_\nu^\mu(\cos\theta)$.  Solid harmonics are harmonic functions
on~$\mathbb{R}^3$, satisfying Laplace's equation; and the ones in the
octahedral, tetrahedral, dihedral, and cyclic families are or can be
finite-valued.

These infinite-dimensional representations are restrictions of the
representation of the common complexification
$\mathfrak{so}(5,\mathbb{C})$, which is carried by the (complex) span of
the family.  They are generically irreducible, and are also generically
non-skew-Hermitian, so that except in special cases, they do not
exponentiate to unitary representations of the corresponding Lie groups,
even formally.  This will now be investigated.

Each of the three real Lie algebras is of rank~$2$, so the center of
its universal enveloping algebra is generated by two elements, called
Casimir invariants; and any irreducible representation must represent
each Casimir by a constant.  The analysis of such representations
resembles the unified classification of the irreducible
representations of $\mathfrak{so}(2,1)$ and
$\mathfrak{so}(3,\mathbb{R})$, the real forms
of~$\mathfrak{so}(3,\mathbb{C})$, which is well known.  (See, e.g.,
\cite[Chap.~3]{Dong2007}.)  In this, representations are classified by
the value taken by their (single) Casimir, and by their reductions
with respect to a ($1$\nobreakdash-dimensional) Cartan subalgebra.
This leads to an understanding of which representations are
skew-Hermitian and which are finite-dimensional.  However, no
comparable unified approach to all representations of
$\mathfrak{so}(3,2)$, $\mathfrak{so}(4,1)$, and
$\mathfrak{so}(5,\mathbb{R})$ seems to have been published.  The
literature has dealt almost exclusively with the skew-Hermitian ones.
($\mathfrak{so}(3,2)$ and $\mathfrak{so}(4,1)$ are treated separately
in \cite{Ehrman57,Evans67} and~\cite{Dixmier61}, and
$\mathfrak{so}(4,1)$ and $\mathfrak{so}(5,\mathbb{R})$ are treated
together in~\cite{Kuriyan68}.)  

The starting point is the complexification $\mathfrak{so}(5,\mathbb{C})$,
which is generated over~$\mathbb{C}$ by $J_\pm$ and~$K_\pm$, the ladder
operators on the order and degree.  It is the complex span of
$J_\pm,K_\pm,\allowbreak R_\pm,S_\pm;\allowbreak J_3,K_3$, each of which is
represented as in (\ref{eq:JKRSplus}) and~(\ref{eq:J3K3plus}) by an
infinite matrix indexed by
$(\nu,\mu)\in\allowbreak(\nu_0,\mu_0)+\nobreak\mathbb{Z}^2$.  The elements
$J_3,K_3$ span a Cartan subalgebra (an~abelian subalgebra of maximal
[complex] dimension, here~$2$), which is represented diagonally:
\begin{displaymath}
  J_3\,{\mathcal{S}}_\nu^\mu = \mu\,\mathcal{S}_\nu^\mu,
  \qquad\qquad   K_3\,{\mathcal{S}}_\nu^\mu = (\nu+\tfrac12)\,\mathcal{S}_\nu^\mu.
\end{displaymath}
When the representation of $\mathfrak{so}(5,\mathbb{C})$ is reduced with
respect to this subalgebra, it splits into an infinite direct sum of
$1$\nobreakdash-dimensional representations, indexed by $(\nu,\mu)$.  The
corresponding real Cartan subalgebras of the $\mathfrak{so}(3,2)$ and
$\mathfrak{so}(4,1)$ in Theorems \ref{thm:newso32} and~\ref{thm:newso41}(i)
are the real spans of $\{J_3,K_3\}$ and~$\{{\rm i}J_3,K_3\}$.  For the
$\mathfrak{so}(3,2)$ and $\mathfrak{so}(5,\mathbb{R})$ in Theorems
\ref{thm:newso41}(ii) and~\ref{thm:newso5r}, they are the real span of
$\{{\rm i}J_3,{\rm i}K_3\}$.  (Recall that $D\defeq K_3$.)  Only for the
last two will the real Cartan subalgebra be represented by skew-Hermitian
matrices; in~fact, by imaginary diagonal ones.

It is readily verified that $J_\pm,K_\pm,\allowbreak
R_\pm,S_\pm;\allowbreak J_3,K_3$ can serve as a Cartan--Weyl basis of
$\mathfrak{so}(5,\mathbb{C})$, their complex span.  That is, when the
adjoint actions of $H_1\defeq J_3$ and $H_2\defeq K_3$ on this
$10$\nobreakdash-dimensional Lie algebra are simultaneously diagonalized,
the common eigenvectors (`root vectors') include $J_\pm,K_\pm,\allowbreak
R_\pm,S_\pm$.  The associated roots $\alpha\in\mathbb{R}^2$ are
$2$\nobreakdash-tuples of eigenvalues, which can be identified with the
displacements $\Delta(\nu,\mu)$, i.e., $\pm(0,1)$, $\pm(1,0)$, $\pm(1,1)$,
$\pm(1,-1)$.  These form the $B_2$ root system.  One can write
\begin{displaymath}
  [H_i,H_j] = 0, \qquad [H_i,E_\alpha]=\alpha_iE_\alpha,
\end{displaymath}
where $E_\alpha$~is the root vector associated to root~$\alpha$.  The
commutators $[E_\alpha,E_\beta]$ also prove to be consistent with the
$B_2$ root system.

The Casimir invariants of $\mathfrak{so}(5,\mathbb{C})$ and its three
real forms can be computed from the commutation relations of the
Cartan--Weyl basis elements.  (For instance, the Killing form for the
algebra yields a quadratic Casimir.)  But it is easier to express them
using the tensor operator ${\bf M}_{ab}$ of any of Theorems
\ref{thm:newso32}, \ref{thm:newso41}, and~\ref{thm:newso5r}.  As
elements of the universal enveloping algebra, the two Casimirs,
quadratic and quartic, are defined thus~\cite{Ehrman57,Evans67}:
\begin{align*}
  c_2 &\defeq -\tfrac12 {\bf M}_{ab}{\bf M}^{ab},\\
  c_4 &\defeq -w_aw^a,
\end{align*}
where $w^a = \tfrac18 \epsilon^{abcde}{\bf M}_{bc}{\bf M}_{de}$ and
the summation convention is employed, indices being raised and lowered
by the tensors $\Gamma^{-1}=(g^{ab})$ and $\Gamma=(g_{ab})$.  The
Levi-Civit\`a tensor $\epsilon_{abcde}$ is skew-symmetric in all
indices, with $\epsilon_{12345}=+1$.  The normalization and sign
conventions are somewhat arbitrary.

\begin{theorem}
\label{thm:mostdegenerate}
  In the representation of the universal enveloping algebra of any of
  the real Lie algebras\/ $\mathfrak{so}(3,2)$, $\mathfrak{so}(4,1)$,
  and\/ $\mathfrak{so}(5,\mathbb{R})$ on the span of the generalized
  solid harmonics\/
  $\mathcal{S}_{\nu_0+n}^{\mu_0+m}(r,\theta,\varphi)$,
  $(n,m)\in\mathbb{Z}^2$, the Casimirs\/ $c_2$ and\/~$c_4$ are
  represented by the constants\/ $-\tfrac{5}{4}$ and\/~$0$,
  irrespective of\/~$\nu_0,\mu_0$.
\end{theorem}
\begin{proof}
  By the expressions for ${\bf M}_{ab},\Gamma$ given in any of
  Theorems \ref{thm:newso32}, \ref{thm:newso41}, and
  \ref{thm:newso5r},
  \begin{subequations}
  \begin{align}
    c_2 &= J\cdot J - PC^+\!\cdot PC^+ + PC^- \!\cdot PC^- + D^2\\
    &= J_3^2 + K_3^2 + \tfrac12\{J_+,J_-\}
    - \tfrac12\{K_+,K_-\} - \tfrac14\{R_+,R_-\} - \tfrac14\{S_+,S_-\},\label{eq:lastminute}
  \end{align}
  \end{subequations}
  where $\{\cdot,\cdot\}$ is the anti-commutator.  This expresses~$c_2$
  in~terms of $J_3,K_3$ and the root vectors.  The formula
  (\ref{eq:lastminute}) can be viewed as subsuming
  $J_3^2+\frac12\{J_+,J_-\}$, which is the Casimir of the
  $\mathfrak{so}(3,\mathbb{R})$ subalgebra spanned by $\{J_+,J_-,J_3\}$;
  and $K_3^2-\frac12\{K_+,K_-\}$, which is the Casimir of the
  $\mathfrak{so}(2,1)$ subalgebra spanned by $\{K_+,K_-,K_3\}$; and also,
  the Casimirs of the remaining two $\mathfrak{so}(2,1)$ subalgebras.  From
  the representations (\ref{eq:JKRSplus}),(\ref{eq:J3K3plus}) of
  $J_\pm,K_\pm,\allowbreak R_\pm,S_\pm$ and $J_3,K_3$ as infinite matrices,
  one calculates from~(\ref{eq:lastminute}) that $c_2$ (like $J_3,K_3$) is
  diagonal in~$(n,m)$, with each diagonal element equaling~$-\tfrac54$.

  For $\mathfrak{so}(5,\mathbb{R})$, which is representative, the five
  components of~$w^a$ include (i)~the scalar ${\rm i}\,J\cdot PC^+$,
  (ii)~the three components of the vector $-{\rm i}\,PC^-\times PC^+ +
  DJ$, and (iii)~the scalar $J\cdot PC^-$.  These expressions,
  involving the scalar and vector product of three-vectors, must be
  interpreted with care: any product $AB$ of two Lie algebra elements
  signifies the symmetrized product $\frac12\{A,B\}$.  But by direct
  computation, one finds from
  (\ref{eq:lastneeded1}),(\ref{eq:lastneeded2}) and the infinite
  matrix representations (\ref{eq:JKRSplus}),(\ref{eq:J3K3plus}) that
  each component of~$w^a$ is represented by the zero matrix, even
  (surprisingly) without symmetrization.\qed
\end{proof}

This result is plausible, if not expected.  In any unitary
representation of a semi-simple Lie group~$G$ on~$L^2(S)$, $S$~being a
homogeneous space $G/K$ of rank~$1$, all Casimir operators except the
quadratic one must vanish.  (See \cite{Boyer71} and
\cite[Chap.~X]{Helgason62}.)  Admittedly, the present representations
of $\mathfrak{so}(4,1)$, by real differential operators acting on
multi-valued, non-square-integrable functions, are non-skew-Hermitian,
and cannot be exponentiated to \emph{unitary} representations
of~$SO_0(4,1)$ of this `most degenerate' type.  The value~$-\frac54$
computed for the quadratic Casimir~$c_2$, irrespective
of~$(\nu_0,\mu_0)$, can be viewed as the value of $j(j+\nobreak1)$,
where $j$~is a formal `angular momentum' parameter equal to
$-\frac12\pm\nobreak{\rm i}$.

\smallskip
For each $(\nu_0,\mu_0)$, the resulting representation of the real Lie
algebra $\mathfrak{g}=\mathfrak{so}(3,2)$, $\mathfrak{so}(4,1)$ or
$\mathfrak{so}(5,\mathbb{R})$, or its universal enveloping algebra
$\mathfrak{U}(\mathfrak{g})$, on the span of the generalized solid
harmonics $\mathcal{S}_{\nu_0+n}^{\mu_0+m}$, $(n,m)\in\mathbb{Z}^2$, can be
viewed linear-algebraically: as a homomorphism~$\rho$ of real vector
spaces, taking $\mathfrak{g}$ (or~$\mathfrak{U}(\mathfrak{g})$) into the
space of infinite matrices indexed by~$(n,m)$.  For each basis element
$A\in\mathfrak{g}$, $\rho(A)$~is determined by
(\ref{eq:JKRSplus}),(\ref{eq:J3K3plus}); and because the basis elements
given in Theorems \ref{thm:newso41} and~\ref{thm:newso5r} include `${\rm
  i}$'~factors, the matrix elements of~$\rho(A)$ may be complex.

To show that certain of these representations are substantially the same as
known ones by infinite matrices that are skew-Hermitian, consider the
effect of replacing the family $\{\mathcal{S}_\nu^\mu =
\mathcal{S}_{\nu_0+n}^{\mu_0+m}\}$ by $\{\hat{\mathcal{S}}_\nu^\mu =
\hat{\mathcal{S}}_{\nu_0+n}^{\mu_0+m}\}$, where the latter are `twisted' by
a square-root factor:
\begin{equation}
\label{eq:replacement}
  \hat{\mathcal{S}}_\nu^\mu   = \hat{\mathcal{S}}_\nu^\mu (r,\theta,\varphi)
  \defeq \sqrt{\frac{\Gamma(\nu-\mu+1)}{\Gamma(\nu+\mu+1)}}\, r^\nu
  {\rm P}_\nu^\mu(\cos\theta){\rm e}^{{\rm i}\mu\varphi}.
\end{equation}
That is, $\hat{\mathcal{S}}_\nu^\mu = r^\nu Y_\nu^\mu$, where
\begin{equation}
\label{eq:replacement2}
  Y_\nu^\mu(\theta,\varphi) = 
  \sqrt{\frac{\Gamma(\nu-\mu+1)}{\Gamma(\nu+\mu+1)}}\,
  {\rm P}_\nu^\mu(\cos\theta){\rm e}^{{\rm i}\mu\varphi}.
\end{equation}
When $\nu=0,1,2,\dots$, with $\mu=-\nu,\dots,\nu$, this $Y_\nu^\mu$ is
the classical (complex) spherical harmonic on~$S^2$, of degree~$\nu$
and order~$\mu$.\footnote{The orthonormalization factor
  $\sqrt{(2\nu+1)/4\pi}$, appropriate for an inner product on~$S^2$,
  is omitted.  But this~$Y^\mu_\nu$ automatically includes the
  so-called Condon--Shortley factor, owing to the definition of~${\rm
    P}_\nu^\mu$ used here (see~\S\,\ref{sec:prelims}).}  In this case,
the square root factor equals
$\left[(\nu+\nobreak\mu+\nobreak1)_{2\mu}\right]^{-1/2}$ and is
positive by convention; a~discussion of how to interpret it in other
cases is deferred.  At least formally, the representation of
$\mathfrak{g}$ or~$\mathfrak{U}(\mathfrak{g})$ on
$\{\hat{\mathcal{S}}_{\nu_0+n}^{\mu_0+m}\}$ comes from that on
$\{\mathcal{S}_{\nu_0+n}^{\mu_0+m}\}$ by a diagonal similarity
transformation.  The formulas
\begin{subequations}
  \label{eq:JKRSnew}
  \begin{align}
    J_\pm\, \hat{\mathcal{S}}_\nu^\mu &= \sqrt{\left(\nu-\mu+\tfrac12\mp\tfrac12\right)\left(\nu+\mu+\tfrac12\pm\tfrac12\right)} \:\hat{\mathcal{S}}_\nu^{\mu\pm1},
    \\
    K_\pm\, \hat{\mathcal{S}}_\nu^\mu &= \sqrt{\left(\nu-\mu+\tfrac12\pm\tfrac12\right)\left(\nu+\mu+\tfrac12\pm\tfrac12\right)} \:\hat{\mathcal{S}}_{\nu\pm1}^\mu,
    \\
    R_\pm\, \hat{\mathcal{S}}_\nu^\mu &=  \sqrt{\left(\nu+\mu+\tfrac12\pm\tfrac12\right)\left(\nu+\mu+\tfrac12\pm\tfrac32\right)} \:\hat{\mathcal{S}}_{\nu\pm1}^{\mu\pm1},
    \\
    S_\pm\, \hat{\mathcal{S}}_\nu^\mu &= \sqrt{\left(\nu-\mu+\tfrac12\pm\tfrac12\right)\left(\nu-\mu+\tfrac12\pm\tfrac32\right)}\: \hat{\mathcal{S}}_{\nu\pm1}^{\mu\mp1}
  \end{align}
\end{subequations}
  and
\begin{equation}
\label{eq:J3K3plusnew}
  J_3\,\hat{\mathcal{S}}_\nu^\mu = \mu\,\hat{\mathcal{S}}_\nu^\mu,
  \qquad\qquad   K_3\,\hat{\mathcal{S}}_\nu^\mu = (\nu+\tfrac12)\,\hat{\mathcal{S}}_\nu^\mu
\end{equation}
now replace (\ref{eq:JKRSplus}) and (\ref{eq:J3K3plus}); but the actions of
the elements $J_i$, $K_i$, $PC^\pm_i$,~$D$ are still defined in~terms of
these by (\ref{eq:JiDdef}) and~(\ref{eq:PCidef}).

\begin{theorem}
\label{thm:last}
  If\/ $(\nu_0,\mu_0)=(0,0)$ or\/ $\left(\tfrac12,\tfrac12\right)$, the
  representation\/~$\rho$ of\/ $\mathfrak{g}=\mathfrak{so}(3,2)$ on the
  span of\/ $\{\hat{\mathcal{S}}_\nu^\mu =
  \hat{\mathcal{S}}_{\nu_0+n}^{\mu_0+m}\}$, $(n,m)\in\mathbb{Z}^2$, which
  is obtained from\/ {\rm(\ref{eq:JKRSnew})},{\rm(\ref{eq:J3K3plusnew})} by
  identifying\/ $\mathfrak{so}(3,2)$ with the real span of\/ ${\rm
    i}J_1,{\rm i}J_2,{\rm i}J_3;\allowbreak {\rm i}PC_1^+,{\rm
    i}PC_2^+,{\rm i}PC_3^+;\allowbreak PC_1^-,PC_2^-,PC_3^-;\allowbreak{\rm
    i}D$ as in Theorem\/~{\rm\ref{thm:newso41}(ii)}, has an irreducible
  constituent that is defined on the subspace spanned by\/
  $\{\hat{\mathcal{S}}_\nu^\mu = \hat{\mathcal{S}}_{\nu_0+n}^{\mu_0+m}\}$
  with\/ $n=0,1,2,\dots$ and\/
  $\mu=-\nu,-\nu+\nobreak1\dots,\nu-\nobreak1,\nu$.  On this subspace,
  every element of\/~$\mathfrak{g}$ is represented by an infinite matrix
  that is skew-Hermitian.
\end{theorem}
\begin{proof}
  By~(\ref{eq:JKRSnew}), if $\mu=\nu$ then $J_+,K_-,S_-$ give zero when
  acting on $\hat{\mathcal{S}}_\nu^\mu$, and if ${\mu=-\nu}$ then
  $J_-,K_-,R_-$ give zero.  Thus $\rho$~is reducible: it can be restricted
  to the stated subspace.  On this subspace, the formal similarity
  transformation performed by the square root factor
  in~(\ref{eq:replacement}) is not singular: only if $\nu\pm\mu$ is a
  negative integer will one of the gammas be infinite.  For $(\nu_0,\mu_0)$
  equal to either of $(0,0)$ and $\left(\tfrac12,\tfrac12\right)$, the
  square root factor simply equals
  $\left[(\nu+\nobreak\mu+\nobreak1)_{2\mu}\right]^{-1/2}$.
  
  By (\ref{eq:JKRSnew}) and~(\ref{eq:J3K3plusnew}), each of
  $\rho(J_\pm),\rho(K_\pm),\rho(R_\pm),\rho(S_\pm);\rho(J_3),\rho(K_3)$
  is a real matrix, the plus and minus versions being transposes of
  each other, and $\rho(J_3),\rho(K_3)$ being symmetric (and
  diagonal).  It follows from (\ref{eq:JiDdef}) and~(\ref{eq:PCidef})
  that the $\rho(J_i)$, the $\rho(PC_i^+)$, and~$\rho(D)$ are
  Hermitian, and the $\rho(PC_i^-)$ are skew-Hermitian.  The claim
  follows.\qed
\end{proof}

Being skew-Hermitian, the two infinite-dimensional representations of
$\mathfrak{so}(3,2)$ in Theorem~\ref{thm:last} exponentiate to
(irreducible) unitary representations of the so-called
anti-\allowbreak{de}~Sitter group $SO_0(3,2)$, or its universal cover.  The
latter have been classified~\cite{Ehrman57,Evans67}, and the ones arising
from the theorem can be identified.  They are the remarkable Dirac
singleton representations, with whimsical names~\cite{Dobrev91,Flato80}:
the $(\nu_0,\mu_0)=(\tfrac12,\tfrac12)$ one is~`Di' and the
$(\nu_0,\mu_0)=(0,0)$ one is~`Rac.'  For the Dirac singletons, the Casimirs
$(c_2,c_4)$ have long been known to equal $(-\tfrac54,0)$.
(See~\cite[\S\,III]{Bohm85}.)  They are singleton representations in the
sense that if they are reduced with respect to the subalgebra
$\mathfrak{g}_0 = \mathfrak{so}(3,\mathbb{R}) \oplus_{\mathbb{R}}
\mathfrak{so}(2,\mathbb{R})$, thereby being split into representations
of~$\mathfrak{g}_0$, each of the latter that appears, does so with unit
multiplicity.  The ones that appear are labeled uniquely by $\nu=\nu_0+n$,
$n=0,1,2,\dots$.

The Rac representation of $\mathfrak{so}(3,2)$ is realized by differential
operators on~$\mathbb{R}^3$ (with complex coefficients), expressions for
which follow immediately from~(\ref{eq:41ops}).  They act on the span of
the \emph{classical} solid harmonics, $\{\hat{\mathcal{S}}_n^m\}$ with
$n=0,1,2,\dots$ and $m=-n,\dots,n$.  Kyriakopoulos~\cite{Kyriakopoulos68}
in~effect discovered that the Rac has such a realization, before the name
was coined, and extended this result to higher dimensions.  The~Di is
realized by the same operators, acting on the span of the `spinorial' solid
harmonics $\{\hat{\mathcal{S}}_{n+\frac12}^m\}$ with $n=0,1,2,\dots$ and
$m=-n-\nobreak\tfrac12,\dots,n+\nobreak\tfrac12$.  The existence of this
realization seems not to be known.  This is perhaps because the solid
harmonics of half-odd-integer degree and order are double-valued
on~$\mathbb{R}^3$, are typically non-square-integrable, and are based on
the little-known dihedral Ferrers functions.  The expressions for the
dihedral Ferrers functions in~terms of Jacobi polynomials, given in
Theorem~\ref{thm:dihedralPQ} above, are new.

The representations of $\mathfrak{so}(3,2)$ carried by the octahedral and
tetrahedral families of solid harmonics, $\{{\mathcal{S}}_\nu^\mu =
{\mathcal{S}}_{\nu_0+n}^{\mu_0+m}\}$ with $(\nu_0+\nobreak\frac12,\mu_0)$
equal to $(\pm\frac13,\pm\frac14)$, $(\pm\frac14,\pm\frac13)$, and
$(\pm\frac13,\pm\frac13)$, are not skew-Hermitian, even up~to diagonal
equivalence.  Twisting the basis to $\{\hat{\mathcal{S}}_\nu^\mu =
\hat{\mathcal{S}}_{\nu_0+n}^{\mu_0+m}\}$ does not help matters, because
only if $\nu_0,\mu_0$ are both integers or both half-odd-integers, which
without loss of generality may be taken to be $0,0$ or $\frac12,\frac12$,
does it permit the representation to be restricted to a subspace spanned by
the harmonics with $\nu=\nu_0+n$, $n=0,1,2,\dots$, and
$\mu=-\nu,-\nu+\nobreak1\dots,\nu-\nobreak1,\nu$.

For general $(\nu_0,\mu_0)$, there is accordingly no restriction on
the index $(n,m)\in\mathbb{Z}^2$ of the basis functions of the
representation, and the square roots in
(\ref{eq:replacement}),(\ref{eq:replacement2}),(\ref{eq:JKRSnew}) may
be square roots of negative quantities.  Irrespective of what sign
convention for the square root is adopted, the resulting imaginary
factors will interfere with skew-Hermiticity; and upon integration of
the representation, with unitarity.  In~fact, the familiar definition
of the (surface) spherical harmonic $Y_\nu^\mu=Y_\nu^\mu
(\theta,\varphi)$ given in~(\ref{eq:replacement2}), incorporating the
square root factor, seems to be useful only when the degree~$\nu$ and
the order~$\mu$ are both integers or both half-odd-integers.

%\bibliographystyle{spbasic}    % basic style, author-year citations
%\bibliographystyle{spphys}     % APS-like style for physics

% OUR PREFERRED:
%\bibliographystyle{spmpsci} 	% mathematics and physical sciences
%\bibliography{general}   % name your BibTeX data base

%% \begin{thebibliography}{10}

%% \end{bibliography}

\end{document}